\documentclass[a4paper]{article}
\usepackage{amsthm,amssymb,amsmath,enumerate,graphicx,epsf,bbold}

\newcommand{\COLORON}{1}
\newcommand{\NOTESON}{0}
\newcommand{\Debug}{0}
\usepackage[usenames,dvips]{color} 
\usepackage{amsthm,amssymb,amsmath,enumerate,graphicx,epsf,mathbbol,stmaryrd}
\usepackage[dvips, bookmarks, colorlinks=false, breaklinks=true]{hyperref}

\newcommand{\fa}{topological amalgamation}

\newcommand{\vapfg}{VAP-free graph}
\newcommand{\ps}{dividing}
\newcommand{\psc}{\ps\ cycle}
\newcommand{\mpsc}{shortest dividing cycle}
\newcommand{\pccg}{plane cubic Cayley graph}
\newcommand{\iicon}{2-connected}
\newcommand{\tcon}{3-connected}

\newcommand{\sepe}{hinge}
\newcommand{\pr}{consistent}
\newcommand{\prem}{\pr\ embedding}
\newcommand{\vapf}{VAP-free}
\newcommand{\ccfg}{\ensuremath{\cc_f(\G)}} 

\newcommand{\red}{simple}
\newcommand{\topem}{topological embedding}
\newcommand{\auto}{colour-automorphism}
\newcommand{\plpr}{planar presentation}
\newcommand{\sedo}{semi-dominant}
\newcommand{\ciso}{colour-preserving isomorphism}
\newcommand{\ncp}{non-crossing pattern}
\newcommand{\soc}{society}
\newcommand{\socs}{societies}

\newcommand{\gt}{\ensuremath{G_2}}
\newcommand{\gtp}{\ensuremath{G'_2}}
\newcommand{\gat}{\ensuremath{\Gam_2}}

\newcommand{\rota}{rotation}
\newcommand{\rot}[2]{\ensuremath{#1}\text{--}\ensuremath{#2}~rotation}

\newcommand{\dray}{double ray}
\renewcommand{\theenumi}{(\roman{enumi})}
\renewcommand{\labelenumi}{\theenumi}

\newcommand{\comment}[1]{}
\newcommand{\COMMENT}[1]{}

\definecolor{darkgray}{rgb}{0.3,0.3,0.3}
\newcommand{\defi}[1]{{\color{darkgray}\emph{#1}}}

\newcommand{\acknowledgements}{\section*{Acknowledgements}}

\comment{
	\begin{lemma}\label{}	
\end{lemma}
\begin{proof}

\end{proof}

\begin{theorem}\label{}
\end{theorem} 
\begin{proof} 	

\end{proof}

}

\newtheorem{proposition}{Proposition}[section]
\newtheorem{definition}[proposition]{Definition}
\newtheorem{theorem}[proposition]{Theorem}
\newtheorem{corollary}[proposition]{Corollary}
\newtheorem{lemma}[proposition]{Lemma}

\newtheorem{conjecture}{{Conjecture}}[section]

\newtheorem{problem}[conjecture]{{Problem}}

\newtheorem{examp}[proposition]{Example}%[section]

\newcommand{\kreis}[1]{\mathaccent"7017\relax #1}

\newcommand{\FIG}{0}

\ifnum \NOTESON = 1 \newcommand{\note}[1]{  

	{\color{blue} \hspace*{-60pt} NOTE: \color{Turquoise}{\small  \tt \begin{minipage}[c]{1.1\textwidth}  #1 \end{minipage} \ignorespacesafterend }} 
	
	}
\else \newcommand{\note}[1]{} \fi

\newcommand{\afsubm}[1]{ \ifnum \Debug = 1 {\mymargin{#1}}
\fi} %For notes on after-submission changes

\ifnum \Debug = 1 \newcommand{\sss}{\ensuremath{\color{red} \bowtie \bowtie \bowtie\ }}
\else \newcommand{\sss}{} \fi

\ifnum \FIG = 1 \newcommand{\fig}[1]{Figure ``{#1}''}
\else \newcommand{\fig}[1]{Figure~\ref{#1}} \fi

\ifnum \Debug = 1 \usepackage[notref,notcite]{showkeys}
\fi

\ifnum \COLORON = 0 \renewcommand{\color}[1]{}
\fi

\newcommand{\showFig}[2]{
   \begin{figure}[htbp]
   \centering
   \noindent
   \epsfbox{#1.eps}
   \caption{\small #2}
   \label{#1}
   \end{figure}
}

\newcommand{\N}{\ensuremath{\mathbb N}}
\newcommand{\R}{\ensuremath{\mathbb R}}
\newcommand{\Z}{\ensuremath{\mathbb Z}}

\newcommand{\cc}{\ensuremath{\mathcal C}}

\newcommand{\cf}{\ensuremath{\mathcal F}}

\newcommand{\cp}{\ensuremath{\mathcal P}}

\newcommand{\ct}{\ensuremath{\mathcal T}}

\newcommand{\cx}{\ensuremath{\mathcal X}}

\newcommand{\oo}{\ensuremath{\omega}}

\newcommand{\Gam}{\ensuremath{\Gamma}}

\newcommand{\kap}{\ensuremath{\kappa}}

\newcommand{\sig}{\ensuremath{\sigma}}

\newcommand{\zero}{\mathbb 0}
\newcommand{\sm}{\backslash}

\newcommand{\isom}{\cong}

\newcommand{\nin}{\ensuremath{{n\in\N}}}
\newcommand{\iin}{\ensuremath{{i\in\N}}}

\newcommand{\sgl}[1]{\ensuremath{\{#1\}}}
\newcommand{\pth}[2]{\ensuremath{#1}\text{--}\ensuremath{#2}~path}

\newcommand{\pths}[2]{\ensuremath{#1}\text{--}\ensuremath{#2}~paths}

\newcommand{\seq}[1]{\ensuremath{(#1_i)_{i\in\N}}} 
\newcommand{\fml}[1]{\ensuremath{\{#1_i\}_{i\in I}}} 
\newcommand{\g}{\ensuremath{G\ }}
\newcommand{\G}{\ensuremath{G}}

\newcommand{\Lr}[1]{Lemma~\ref{#1}}
\newcommand{\Tr}[1]{Theorem~\ref{#1}}
\newcommand{\Sr}[1]{Section~\ref{#1}}
\newcommand{\Prr}[1]{Pro\-position~\ref{#1}}

\newcommand{\Cr}[1]{Corollary~\ref{#1}}
\newcommand{\Cnr}[1]{Con\-jecture~\ref{#1}}

\newcommand{\Dr}[1]{De\-fi\-nition~\ref{#1}}

\newcommand{\lf}{locally finite}
\newcommand{\lfg}{locally finite graph}

\newcommand{\Cg}{Cayley graph}

\newcommand{\hcy}{Hamilton circle}

\renewcommand{\iff}{if and only if}
\newcommand{\fe}{for every}
\newcommand{\Fe}{For every}

\newcommand{\st}{such that}

\newcommand{\ti}{there is}
\newcommand{\ta}{there are}

\newcommand{\obda}{without loss of generality}

\newcommand{\Btco}{By the construction of}
\newcommand{\wrt}{with respect to}

\newcommand{\istc}{is straightforward to check}
\newcommand{\ises}{is easy to see}

\newcommand{\labtequ}[2]{ \begin{equation} \label{#1} 	\begin{minipage}[c]{0.9\textwidth}  #2 \end{minipage} \ignorespacesafterend \end{equation} }

\newcommand{\mymargin}[1]{% <- dieses % verhindert ein ungewolltes Leerzeichen
  \marginpar{%
    \begin{minipage}{\marginparwidth}\small%
      \begin{flushleft}%
        {\color{blue}#1}%
      \end{flushleft}%
   \end{minipage}%
  }%
}%
\newcommand{\mySection}[2]{}

\newcommand{\citeCayIIFignospiral}{\cite[Figure~10]{cay2con}}
\newcommand{\citeIIconCorNosp}{\cite[Proposition~5.7]{cay2con}}
\newcommand{\citeCayIIFigtarget}{\cite[Figure~5]{cay2con}}
\newcommand{\citeIICorplpr}{\cite[Corollary~6.3]{cay2con}}
\newcommand{\citeObsfinfac}{\cite[Observation 5.8]{cay2con}}
\newcommand{\citVapfL}{\cite[Lemma 3.4]{vapf}}

\newcommand{\twam}{twist-amalgamation}
\newcommand{\twsqam}{twist-squeeze-amalgamation}

\renewcommand{\labelenumi}{\arabic{enumi}.}

\title{The planar cubic Cayley graphs}
\author{Agelos Georgakopoulos\thanks{Supported by FWF grant P-19115-N18.} \medskip \\
  {Technische Universit\"at Graz}\\
  {Steyrergasse 30, 8010}\\
  {Graz, Austria}\\
}

\date{}
\begin{document}
\maketitle

\begin{abstract}
We obtain a complete description of the planar cubic Cayley graphs, providing an explicit presentation and embedding for each of them. This turns out to be a rich class, comprising several infinite families. We obtain counterexamples to conjectures of Mohar, Bonnington and Watkins. Our analysis makes the involved graphs accessible to computation, corroborating a conjecture of Droms.
\end{abstract}

\section{Introduction}

\subsection{Overview}
The study of  planar Cayley graphs has a tradition starting in 1896 with Maschke's characterization of the finite ones. Among the infinite planar \Cg s, those corresponding to a discontinuous action on the plane have received a lot of attention. Their groups are important in complex analysis, and they are closely related to the surface groups \cite[Section 4.10]{ZVC}. These graphs and groups are now well understood due to the work of Macbeath \cite{macCla}, Wilkie \cite{wilNon}, and others; see \cite{ZVC} for a survey. The remaining ones are harder to analyse. They have been the subject of more recent work \cite{droInf,DrSeSeCon,dunPla,vapf}, and they are not yet completely classified. For example, we do not know if they can be effectively enumerated \cite{droInf,DrSeSeCon}.

In this paper we study those planar Cayley graphs that are \defi{cubic}, which means that every vertex is adjacent with precisely three other vertices. It turns out that this class is restricted enough to allow for a complete description of all of its elements, while offering enough variety to allow an insight into the general planar Cayley graphs.

Our main result is
\begin{theorem} \label{main}
Let \g be a planar cubic \Cg. Then \g has precisely one of the presentations listed in Table~\ref{table}. Conversely, each of these presentations, with parameters chosen in the specified domains, yields a non-trivial planar cubic \Cg.
\end{theorem}

Some of the entries  of Table~\ref{table} yield counterexamples to a conjecture of Bonnington and Watkins \cite{BoWaPla} and Bonnington and Mohar \cite{mohPers}; see \Sr{intBW}. 

The presentations of Table~\ref{table} have a special structure that is related to the embedding of the corresponding \Cg, and yield geometric information about the corresponding Cayley complex. The ideas of this paper are used in \cite{agmh} to prove that every planar Cayley graph admits such a presentation. This solves the aforementioned problem of \cite{droInf,DrSeSeCon} asking for an effective enumeration; see \Sr{intAP}. 

Motivated by Stallings' celebrated theorem, Mohar conjectured that every  planar \Cg\ with more than 1 end can be obtained from simpler ones by a glueing operation reminiscent of group amalgamation. Some of the entries  of Table~\ref{table} disprove Mohar's conjecture, but we will show that the conjecture becomes true in the cubic case after a slight modification. This modified version might be true for all \Cg s, not just the planar ones, yielding a refinement of Stallings' theorem. See \Sr{intEx} for details. 

In \cite{fleisch} I asked for a characterization of the \lf\ \Cg s that admit a \defi{Hamilton circle}, i.e.\ a homeomorphic image of $S^1$ in the end-compactification of the graph containing all vertices.
Mohar and I conjectured that every \tcon\ planar \Cg\ does. As explained in \Sr{secHam}, perhaps the hardest case for this conjecture is the cubic case, and our classification constitutes significant progress in this direction.

%Some of the \Cg s appearing in Table~\ref{table} were already well-known: types \ref{Aoi} to \ref{Aziv} yield the cubic, semi-regular, archimedian tilings of the euclidean or hyperbolic plane and the sphere. Types \ref{Aici} to \ref{Aicii} correspond to free products of finite groups. The focus of this paper lies in the remaining ones, especially the \tcon\ multi-ended ones, which are much harder to analyse.

\subsection{Some examples and Mohar's conjecture} \label{intEx}

Let us consider some examples. Suppose \g is a finite or 1-ended \Cg\ embedded in the plane,  that some generator $z$ of \g spans a finite cyclic subgroup, and the corresponding cycles of \g bound faces in this embedding, see \fig{fiamal} (i). Then, considering the amalgamation product \wrt\ the subgroup $\left<z\right>$, and using the generators of \G, we obtain a multi-ended \Cg\ $G'$ which can also be embedded in the plane: the face that was bounded by some coset $C$ of $\left<z\right>$ in $G$ can be used to recursively accommodate the copy $K$ of \g glued along $C$ and the further copies sharing a $z$ cycle with $K$ and so on; see \fig{fiamal} (ii). 

\showFig{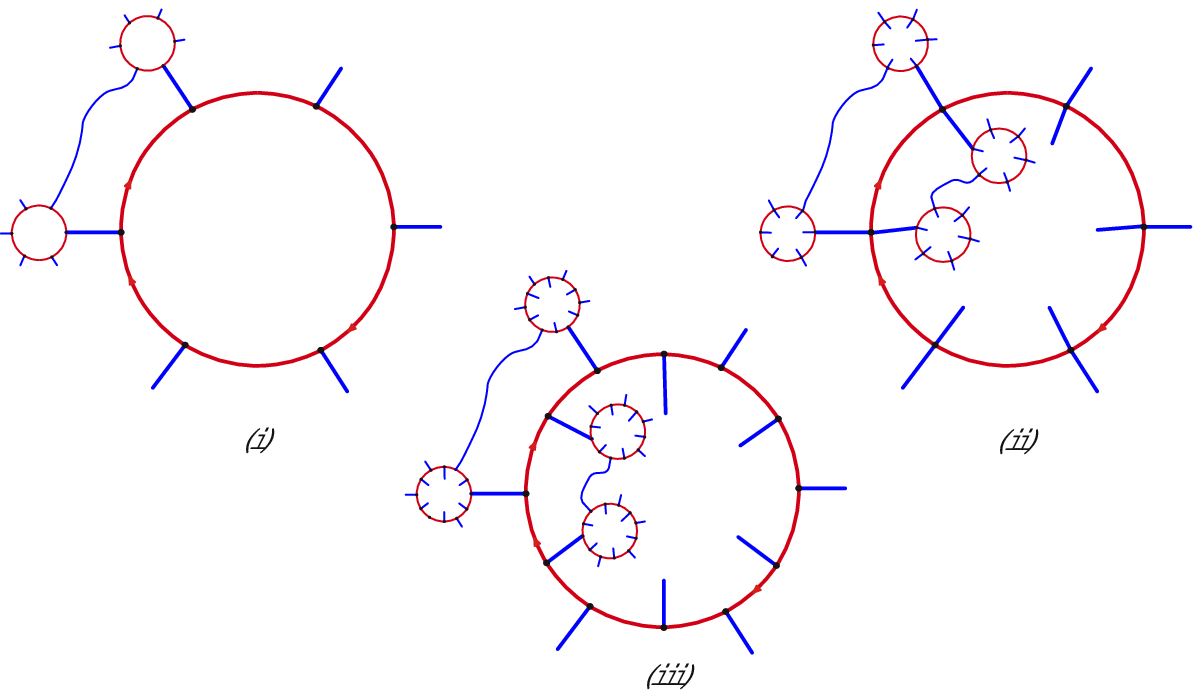}{A \twam.}

This kind of amalgamation can be used to produce new planar \Cg s from simpler ones, but it cannot yield cubic graphs since the degree of a vertex is increased. To amend this, Mohar \cite{mohPers} proposed the following variant of this operation. For every $z$ cycle $C$ of \fig{fiamal} (ii), rotate one of its sides in such a way that the edges incident with $C$ on either side do not have common endvertices, but appear in an in-out alternating fashion instead; see  \fig{fiamal} (iii). It is at first sight not clear why the new graph $G''$ produced like this is a \Cg, but in fact it is, and its group is an overgroup of the group of \G, see \cite{am}. We call the operation of \fig{fiamal} a \defi{\twam}. %Thus, both these examples corroborate \Cnr{conjTM}: in the first one $G'$ is a subgraph of $G$, and in the second $G''$  fits within \g if we subdivide each edge $e$ of the former into two by putting a new vertex on the midpoint of $e$.

Our next example is slightly more complicated. Let this time \g be the \Cg\ of a finite dihedral group shown in the left part of \fig{fidih}. Then, \fe\ 4-cycle $C$ of \g bounding a face $F$, we embed a copy of \g in $F$, with $C$ being glued with a corresponding cycle of the copy. As in our last example, this glueing does not identify vertices with vertices, but rather puts some vertices of one copy of \g midway along some edges of the other, see the right part of \fig{fidih}. 

\showFig{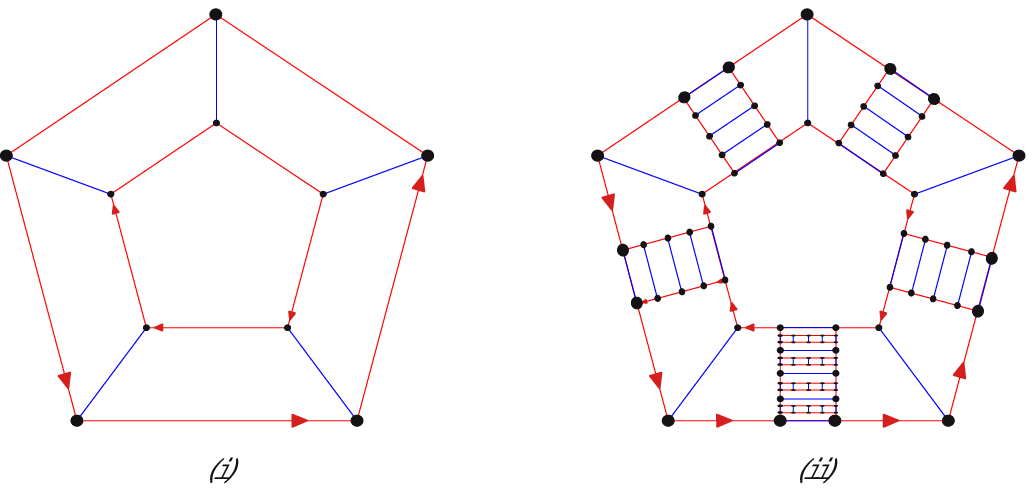}{A \twsqam.}

Unlike the previous example, where every edge of $C$ was subdivided into two, this time every other edge of $C$ is subdivided into three while every other edge is left intact. We recursively repeat this kind of glueing operation for the newly appeared face-bounding 4-cycles. Again, we obtain a new planar \Cg. %\ \cite{am} that corroborates \Cnr{conjTM}. \sss disproves Mohar \sss
We call the operation of \fig{fidih} a \defi{\twsqam}. 

%In fact, for the graphs studied here we can prove more than is asked for by \Cnr{conjTM}. We will use the examples of \fig{fiamal} and \fig{fidih} again to illustrate this. For a \Cg\ $G''$ as in \fig{fiamal} we show that if $S\ni z$ is the set of generators of the original \Cg\ \G, and $R$ is a set of defining relations for the group of \G, then $G''$ is the \Cg\ of the group $\left<S''\mid R''\right>$ \wrt\ to the generating set $S''$ obtained from $S$ by replacing $z$ by a new letter $w$, where $R''$ is obtained from $R$ by replacing each occurence of the symbol $z$ by $w^2$. Similarly, a presentation for the \Cg\ of the right part of \fig{fidih} can be obtained from one of the original graph \g by replacing each occurence of the symbol $z$ by ... . 

These two examples are special cases of a more general, and more complicated, phenomenon: the hardest task addressed in this paper is to show that for every multi-ended cubic planar \Cg\ $G'$, it is possible to obtain a presentation of $G'$ from one of a finite or 1-ended cubic \Cg\ $G$ embedded in $G'$, by replacing some of the generators by new ones, and replacing each occurrence of an old generator $z$ by a word $W_z$ in the new generators. We call this operation a \defi{word extension of \G}. In our two examples it was enough to replace just one generator, and the corresponding word had length two or three. There are many cases where this is enough: the entries  \ref{Aiii}, \ref{AIa2i} and \ref{AIIa2i} of Table~\ref{table} for example were obtained like that. However, \ta\ much harder cases, where one needs to replace all generators by new ones, and there is no upper bound to the length of the replacing words $W_z$ needed. This is the case for the last five entries of Table~\ref{table}.

In our two examples we saw how a \Cg\ $G'$ with infinitely many ends can be obtained from a finite or 1-ended \Cg\ $G$ by glueing copies of $G$ together.
Mohar's  aforementioned conjecture \cite[Conjecture 3.1]{mohTre}  was that every planar \Cg\ graph with more than one end can be obtained in a similar way:
\begin{conjecture} \label{conjBM}
Let \G\ be a planar Cayley graph. Then \G\ can be obtained as the tree amalgamation of (subdivisions of) one or more planar Cayley graphs, each of which is either finite or 1-ended.
Moreover, the identifying sets in these amalgamations correspond to cosets of finite subgroups.
\end{conjecture}
We will refrain from repeating the definition of Mohar's \defi{tree amalgamation}, which can be found in  \cite[Section~2]{mohTre}, and contend ourselves with examples.
The phrase ``subdivisions of'' is missing in Mohar's conjecture, but apparently it was intended. In any case, omitting it makes the first sentence false as indicated by the above examples (the first of which is due to Mohar). 

The results of this paper corroborate  the first sentence of \Cnr{conjBM}. For example, the graph of \fig{fiamal} (iii) is the tree amalgamation of the graph $H$ obtained from \fig{fiamal} (i) by subdividing every $z$ edge into two, the amalgamations taking place along the facial cycles. Similarly, the graph of  \fig{fidih} (ii) can be obtained from that of \fig{fidih} (i) by subdividing every directed edge into three, and then amalgamating copies of the resulting graph along the trapezoid cycles. 

The second sentence of \Cnr{conjBM} is easily shown to be false by the latter example.

\subsection{The conjecture of Bonnington and Watkins} \label{intBW}

Bonnington and Watkins \cite{BoWaPla} made the following conjecture.
\begin{conjecture}[\cite{BoWaPla}] \label{bowa}
No locally finite vertex-transitive graph of connectivity 3 admits a planar embedding wherein some vertex is incident with more than one infinite face-boundary.
\end{conjecture}

Mohar \cite{mohTre} disproved this using his aforementioned amalgamation construction to obtain vertex-transitive graphs of connectivity 3 with arbitrarily many  infinite face-boundaries incident with each vertex. This motivated Bonnington and Mohar \cite{mohPers} to ask whether \ti\ a planar \tcon\ vertex-transitive graph all face-boundaries of which are infinite. 

In this paper we show that, surprisingly, such graphs do exist: we construct \tcon\ planar cubic \Cg s in which no face is bounded by a finite cycle. Note that a \tcon\ planar graph has an essentially unique embedding (see \Tr{imrcb} below), which makes the existence of such examples more surprising.

In fact, we describe all \tcon\ planar cubic \Cg s with the property that  no face is bounded by a finite cycle: they are precisely the members of the families \ref{AIIciii} and \ref{AIId2iii} of Table~\ref{table}. See \Cr{monsters} and \Cr{monsters2}.

\subsection{Planar presentations and  effective enumeration} \label{intAP}

Some of the groups appearing in Table~\ref{table} were already known and well-studied, namely those admitting a Cayley complex   that can be embedded in the plane after removing some redundant simplices, see \cite{vapf}. These are the 1-ended ones, those of connectivity 1, and some of those appearing in Table~\ref{table} with connectivity 2 \cite{cay2con}. This paper is mainly concerned with the remaining ones. It turns out that none of the \Cg s of Table~\ref{table} from \ref{AIIa2i} on admits a  Cayley complex as above\footnote{This does not immediately imply that their groups are distinct from the groups of earlier entries, as a group can have various planar \Cg s of different nature; in fact, for very small values of the involved parameters it can happen that the corresponding group coincides with one from an earlier entry. For example, as pointed out by M.~Dunwoody (personal communication), the group of case \ref{AIId2i} for $n=2$ and $m=1$ coincides with that of case \ref{Avi} for $n=2$ and $m=4$. However, for larger values of the parameters all groups from \ref{AIIa2i} on should be distinct from the previous ones; a proof of this fact is in progress.} \cite{vapf}. 

However, it follows from our results that they admit a Cayley complex $X$ for which \ti\ a mapping $\sig: X \to \R^2$ \st\ for every two 2-simplices of $X$, the images of their interiors under \sig\ are either disjoint or one of these images is contained in the other, or their intersection is a 2-simplex bounded by the two parallel edges corresponding to some involution in the generating set. We call such an $X$ an \defi{almost planar} Cayley complex.
\begin{theorem} \label{thmAPX}
Every cubic planar \Cg\ is the 1-skeleton of an almost planar Cayley complex of the same group. 
\end{theorem}

A \defi{planar presentation} is a group presentation giving rise to an almost planar Cayley complex. This property can be recognised by an algorithm, see \Dr{dncp} for an example or \cite{agmh} for details.

All presentations in Table~\ref{table} are planar. The semi-colons contained in some of them are used to distinguish relators that induce face-boundaries, which appear before the semi-colon, from relators that induce cycles separating the graph in infinite components, and appear after the semi-colon.

\Tr{main} yields an effective enumeration of the planar cubic \Cg s, corroborating the aforementioned conjecture of Droms et.\ al.\ \cite{droInf,DrSeSeCon}. Moreover, our constructions imply that every planar cubic \Cg\ is effectively computable.

Using the ideas of this paper \Tr{thmAPX} is extended in \cite{agmh} to arbitrary planar \Cg s. As planar presentations can be recognised by an algorithm, this settles the general case of the aforementioned conjecture.

\subsection{Hamilton circles in \Cg s} \label{secHam}

The following conjecture was motivated by \cite{fleisch}:

\begin{conjecture}[Georgakopoulos \& Mohar (unbublished)] \label{conGM}
Every finitely generated \tcon\ planar \Cg\ admits a {Hamilton circle}.
\end{conjecture}

A \defi{Hamilton circle} of a graph \g is a homeomorphic image of the circle $S^1$ in the end-compactification of \g containing all vertices. It has been proved \cite{CWY} that every 4-connected \lf\ \vapf\ planar graph has a Hamilton circle, and it is conjectured that the \vapf ness requirement can be dropped in that theorem; this generalises a classical theorem of Tutte for finite graphs. A \Cg\ of degree 4 or more will either be 4-connected, in which case one can try to apply the above result, or its small separators will give away information about its structure; e.g.\ we know that a non-4-connected \Cg\ of degree 5 or more cannot be 1-ended \cite[Lemma 2.4.]{BaGro}. 

This means that the cubic case plays an important role for \Cnr{conGM}. Thus this paper constitutes an important step towards its proof, as it describes the structure of these graphs in a way that can be exploited to prove hamiltonicity. Indeed, in graphs like the one in \fig{fiamal} (iii) Mohar and I found a way to construct a Hamilton circle of the whole graph using Hamilton circles of the basis graph \fig{fiamal} (i).

%\end{comment

\begin{table}[h!b!p!] 

\renewcommand{\labelenumi}{\arabic{enumi}.}
\renewcommand{\theenumi}{(\arabic{enumi})}
\newcommand{\row}[2]{\begin{minipage}[c]{0.32\textwidth} #1 \end{minipage}\ignorespacesafterend & 	\begin{minipage}[c]{0.99\textwidth} \begin{enumerate} \setcounter{enumi}{\value{enumii_saved}} \addtolength{\itemsep}{-0.63\baselineskip} #2 \setcounter{enumii_saved}{\value{enumi}} \end{enumerate}
 \end{minipage}\ignorespacesafterend}

\begin{tabular}{ll}%{|m{.5\linewidth}|m{.5\linewidth}|}
\newcounter{enumii_saved}

\row{$\kap(G)=1$}{
		\item \label{Aici} $\left<a,b\mid b^2, a^n\right>$, $n\in \{\infty, 2, 3, \ldots\}$
		\item \label{Aicii} $\left<b,c,d\mid b^2,c^2,d^2, (bc)^n \right>$, $n\in \{\infty, 1, 2, 3, \ldots\}$
		}					
\\
\hline
\row{$\kap(G)=2$}{%\begin{enumerate}\addtolength{\itemsep}{-0.5\baselineskip}
\item \label{Ai} $G \isom Cay \left<a,b\mid b^2, (ab)^n\right>$, $n\geq 2$ 
\item \label{Aii} $G \isom Cay \left<a,b\mid b^2, (aba^{-1}b^{-1})^n\right>$, $n\geq 1$
\item \label{Aiii} $G\isom Cay \left<a,b \mid b^2, a^4, (a^2b)^n \right>, n\geq 2$;

\item \label{Aiv} $G \isom Cay \left< b,c,d \mid b^2, c^2, d^2, (bc)^2, (bcd)^m\right>$, $m\geq 2$
\item \label{Av} $G \isom Cay \left<b,c,d \mid b^2, c^2, d^2, (bc)^{2n}, (cbcd)^m\right>$, $n,m\geq 2$ 
\item \label{Avi} $G \isom Cay \left<b,c,d\mid b^2, c^2,d^2, (bc)^n, (bd)^m\right>$, $n,m\geq 2$
\item \label{Avii} $G \isom Cay \left<b,c,d\mid b^2, c^2,d^2, (b(cb)^nd)^m\right>, n\geq 1, m\geq 2$
\item \label{Aviii} $G \isom Cay \left<b,c,d\mid b^2, c^2,d^2, (bcbd)^m \right>$, $m\geq 1$
\item \label{Aix} $G \isom Cay\left<b,c,d\mid b^2, c^2, d^2, (bc)^n, cd\right>$, $n\geq 1$ %(degenerate cases) 
}
\\
\hline
\row{$\kap(G)=3$,\\ 
 \g is 1-ended or finite,\\
 with two generators}{
\item \label{Aoi} $G \isom Cay \left<a,b\mid b^2, a^n, (ab)^m\right>$, $n\geq 3$, $m\geq 2$ %(and $a,b$ preserve spin);
\item \label{Aoii} $G \isom Cay \left<a,b\mid b^2, a^n, (aba^{-1}b)^m\right>$, $n\geq 3, m\geq 1$ %(and only $a$ preserves spin);
\item \label{Aoiii} $G \isom Cay \left<a,b\mid b^2, (a^2b)^m\right>$, $m\geq 1$ %(and only $b$ preserves spin);
\item \label{Aoiv} $G \isom Cay \left<a,b\mid b^2, (a^2ba^{-2}b)^m\right>$, $m\geq 1$ %(and $a,b$ reverve spin);
}
\\
\hline
\row{$\kap(G)=3$,\\ 
 \g is 1-ended or finite,\\
  with three generators}{
	\item \label{Azi} $G \isom Cay\left<b,c,d\mid b^2, c^2, d^2, (bcd)^n \right>$, $n\geq 1$ %(all colours preserve spin);
\item \label{Azii} $G \isom Cay\left<b,c,d\mid b^2, c^2, d^2, (cbcdbd)^n \right>$, $n\geq 1$ %(only $c,d$  preserve spin);
\item \label{Aziii} $G \isom Cay\left<b,c,d\mid b^2, c^2, d^2, (bc)^n, (bdcd)^m\right>$, $n\geq 2, m\geq 1$ %(only $d$  preserves spin);
\item \label{Aziv} $G \isom Cay\left<b,c,d\mid b^2, c^2, d^2, (bc)^n, (cd)^m, (db)^p \right>$, $n,m,p \geq 2$ %(all colours reverse spin);
}
\\
\hline
\row{$\kap(G)=3$,\\ 
 \g is multi-ended,\\
  with two generators}{
			\item \label{AIa2i} $G \isom Cay \left<a,b\mid b^2,  (a^2b)^m; a^{2n}\right>, n\geq 3, m\geq 2$
			\item \label{AIa2ii} $G \isom Cay \left<a,b\mid b^2,  (a^2ba^{-2}b)^m; a^{2n}\right>$, $n\geq 3, m\geq 1$
			\item \label{Abi} $G \isom Cay\left<a,b\mid b^2, a^2ba^{-2}b;  (baba^{-1})^n \right>$, $n\geq 2$
			\item \label{Abii} $G \isom Cay\left<a,b\mid b^2, (a^2 ba^{-2} b)^m; (baba^{-1})^n \right>$, $n,m,p \geq 2$
			\item \label{Abiii} $G \isom Cay \left<a,b\mid b^2, (a^2b)^2; (ab)^{2m} \right>$, $m\geq 2$
			\item \label{Abiv} $G \isom Cay \left<a,b\mid b^2, (a^2b)^{n}; (ab)^{2m} \right>$, $n\geq 3$, $m\geq 2$
}
\\
\hline
\row{$\kap(G)=3$,\\ 
 \g is multi-ended,\\
  with three generators}{
	\item \label{AIIa2i} $G \isom Cay \left<b,c,d\mid b^2,c^2,d^2,(bcd)^m; (bc)^n\right>$, $n\geq 2$, $m\geq 2$
	\item \label{AIIa2ii}$G \isom Cay \left<b,c,d\mid b^2,c^2,d^2, (bcdcbd)^m; (bc)^n\right>$, $n\geq 2, m\geq 1$
	\item \label{AIId11} $G \isom Cay\left<b,c,d\mid b^2, c^2, d^2, (cd)^m, (dbcb)^p; (bc)^{2n} \right>$, $n,m,p \geq 2$
	\item \label{Addi} $G \isom Cay\left<b,c,d\mid {b^2}, c^2, d^2, (bcbdcd)^m; (bc)^{2n} \right>, n\geq 2, m\geq 1$
	\item \label{Addii} $G \isom Cay\left<b,c,d\mid b^2, c^2, d^2, (bcbdcd)^p ; (dc)^{2n}, (bc)^{2m}\right>$, $n,m,p \geq 2$
	\item \label{Addiii} $G \isom Cay\left<b,c,d\mid b^2, c^2, d^2, (bcbdcd)^p; (dbcb)^{2n}, (bc)^{2m} \right>$, $n,m,p \geq 2$ 
	
	\item \label{AIIci} $G \isom Cay\left<b,c,d\mid {b^2}, c^2, d^2, (bcd)^2; (bcdc)^n \right>, n\geq 2$ %(faces of size $6$);
	\item \label{AIIcii} $G \isom Cay\left<b,c,d\mid {b^2}, c^2, d^2, (bcd)^k; \cp \right>$,  $k\geq 3$,\\  \text{ \cp\ is a \ncp} %(faces of size $3k\geq 9$);
	\item \label{AIIciii} $G \isom Cay\left<b,c,d\mid {b^2}, c^2, d^2; \cp \right>$,\\ \text{ \cp\ is a non-regular \ncp} %(no finite faces) 

	\item \label{AIId2i} $G \isom Cay\left<b,c,d\mid {b^2}, c^2, d^2, (bdb cdc)^k; (c(bc)^nd)^{2m}\right>, k\geq 2,\\ n,m\geq 1, n+m\geq 3$
	\item \label{AIId2ii} $G \isom Cay\left<b,c,d\mid {b^2}, c^2, d^2, (bdb cdc)^q; (c(bc)^{n-1}d)^{2m}, (c(bc)^{n} d)^{2r} \right>$,\\  $n,r,m,q\geq 2$
	\item \label{AIId2iii} $\g \isom Cay \left< b,c,d\mid b^2, c^2, d^2;  (c(bc)^{n-1}d)^{2m}, (c(bc)^{n} d)^{2r} \right>$, $n,r,m\geq 2$
	}
\end{tabular}

\caption{Classification of the cubic planar Cayley graphs. All presentations are \defi{planar} in the sense of \Sr{intAP}.}
\label{table}
\end{table}

%END OF COMMENT
%\end{section}

%\section{Notes}
 
%There are 2-connected planar \Cg s that admit no transitive map embedding, see Droms Servatius and Servatius 

%\end{section}

\subsection{Structure of the proof} \label{secSke}

This paper is structured as follows. After some definitions and basic facts, we handle the finite and 1-ended case, corresponding to entries \ref{Aoi} to \ref{Aziv} of Table~\ref{table}, in \Sr{secfin}; entries \ref{Aici} and \ref{Aicii} are easy and entries \ref{Ai} to \ref{Aix} were handled in \cite{cay2con}; see \Sr{lowk}. Most of this paper is concerned with entries from \ref{AIa2i} on. They are divided naturally into cases according to the number of generators (2 or 3), the existence of cycles avoiding one of the generators, and the spin behaviour (see \Sr{secDem}) in the corresponding embedding; they occupy Sections \ref{secI} and \ref{secII}.

These cases vary considerably in difficulty, in general becoming more difficult as we progress towards the end of Table~\ref{table}. However, there is a common structure: \fe\ graph \g as in the entries from \ref{AIa2i} on, we begin by finding a subgroup, typically finite or 1-ended, that has a \Cg\ \gt\ topologically embedded in \G. This yields a Dunwoody structure tree whose nodes are the copies of \gt\ in \G, in which two nodes are adjacent if the corresponding copies share a cycle of \G. These cycles are typically cycles of minimal length bounding two infinite components of \G. We then show that the desired presentation of \G\ can be derived from a presentation of \gt\ by translating each generator $g$  of \gt\ into a word in the generators of \G, which word can be read off the path in \g onto which $g$ is mapped when embedding \gt\ into \G. The presentation of \gt\ is obtained from some earlier entry of Table~\ref{table}, and many of the entries, even multi-ended ones, find themselves embedded in other entries, giving rise to a rich structure.

In \Sr{afterIId} we pose some further problems, one of which seeks for a generalisation of this proof-structure to a general multi-ended \Cg.

\section{Definitions}

\subsection{\Cg s and group presentations}
We will follow the terminology of \cite{diestelBook05} for graph-theoretical terms and that of \cite{bogop} for group-theoretical ones. Let us recall the definitions most relevant for this paper.

Let \Gam\ be group and let $S$ be a symmetric generating set of \Gam. The Cayley graph $Cay(\Gam,S)$ of $\Gam$ \wrt\ $S$ is a coloured directed graph $\g= (V,E)$ constructed as follows. The vertex set of \g is \Gam, and the set of colours we will use is $S$.  For every $g\in \Gam, s\in S$ join $g$ to $gs$ by an edge coloured $s$ directed from $g$ to $gs$. Note that $\Gam$ acts on \g by multiplication on the left; more precisely, \fe\ $g\in \Gam$ the mapping from $V(G)$ to $V(G)$ defined by $x \mapsto gx$ is a \defi{\auto} of \G, that is, an automorphism of \g that preserves the colours and directions of the edges. In fact, \Gam\ is precisely the group of \auto s of \G. Any presentation of \Gam\ in which $S$ is the set of generators will also be called a presentation of $Cay(\Gam,S)$.

If $s\in S$ is an \defi{involution}, i.e.\ $s^2=1$, then every vertex of \g is incident with a pair of parallel edges coloured $s$ (one in each direction). However, when calculating the degree of a vertex of a \Cg\ we will count only one edge for each such pair, and we will draw only one, undirected, edge in our figures. For example, if $S$ consists of three involutions then we consider the corresponding \Cg\ to be cubic. This convention is common in the literature, and it is necessary if one wants to study the property of being a \Cg\ as a graph-theoretical invariant, like e.g.\ in Sabidussi's theorem \cite[Proposition~3.1]{BaAut}. 

Given a group presentation $\left< S \mid \mathcal R \right>$ we will use the notation $Cay \left< S \mid \mathcal R \right>$ for the \Cg\ of this group \wrt\ $S$.

If $R\in \mathcal R$ is any relator in such a presentation and $g$ is a vertex of $G=Cay \left< S \mid \mathcal R \right>$, then starting from $g$ and following the edges corresponding to the letters in $R$ in order we obtain a closed walk $W$ in $G$. We then say that $W$ is \defi{induced} by $R$; note that for a given $R$ \ta\ several walks in \g induced by $R$, one for each starting vertex $g\in V(G)$. If $R$ induces a cycle then we say that $R$ is \defi{\red}; note that this does not depend on the choice of the starting vertex $g$. A presentation $\left<a,b,\ldots \mid R_1,R_2,\ldots\right>$ of a group \Gam\ is called \defi{\red}, if $R_i$ is \red\ \fe\ $i$. In other words, if for every $i$ no proper subword of any $R_i$ is a relation in $\Gamma$. 

Define the (finitary) \defi{cycle space} \ccfg\ of a graph $G=(V,E)$ to be the vector space over $\Z_2$ consisting of those subsets of $E$ such that can be written as a sum (modulo 2) of a finite set of {circuits}, where a set of edges $D\subseteq E$ is called a \defi{circuit} if it is the edge set of a cycle of \G. Thus \ccfg\ is isomorphic to the first simplicial homology group of \g over $\Z_2$. The \defi{circuit} of a closed walk $W$ is the set of edges traversed by $W$ an even number of times. Note that the direction of the edges is ignored when defining circuits and \ccfg. The cycle space will be a useful tool in our study of \Cg s because of the following well-known fact which is easy to prove.
\begin{lemma} \label{relcc}
Let $G= Cay \left< S \mid R \right>$ be a \Cg\ of the group \Gam. Then the set of circuits of walks in \g induced by relators in $R$ generates $\cc_f(G)$. 

Conversely, if $R'$ is a set of words, with letters in a set $S\subseteq \Gam$ generating \Gam, such that the set of circuits of cycles of $Cay(G,S)$ induced by $R'$ generates \ccfg, then $\left< S \mid R' \right>$ is a presentation of \Gam.
\end{lemma}

\subsection{Graph-theoretical concepts}
Let $G=(V,E)$ be a connected graph fixed throughout this section. Two paths in \g are \defi{independent}, if they do not meet  at any vertex except perhaps at common endpoints. 
If $P$ is a path or cycle we will use $|P|$ to denote the number of vertices in $P$ and  $||P||$ to denote the number of edges of $P$. Let $xPy$ denote the subpath of $P$ between its vertices $x$ and $y$.

A cycle $C$ of $G$ is \defi{induced} if every edge of \g that has both endvertices on $C$ is an edge of $C$.

A \defi{\sepe} of \g is an edge $e=xy$ \st\ the removal of the pair of vertices $x,y$  disconnects \G. A \sepe\ should not be confused with a \defi{bridge}, which is an edge whose removal separates \g  although its endvertices are not removed.

The set of neighbours of a vertex $x$ is denoted by \defi{$N(x)$}.

\G\ is called \defi{$k$-connected} if $G - X$ is connected for every set $X\subseteq V$ with $|X | < k$. Note that if \g is $k$-connected then it is also $(k-1)$-connected. The \defi{connectivity $\kappa(G)$} of \g is the greatest integer $k$ such that \G\ is $k$-connected.

A $1$-way infinite path is called a \defi{ray}, a $2$-way infinite path is a \defi{\dray}. Two rays are {equivalent} if no finite set of vertices separates them. The corresponding equivalence
classes of rays are the \defi{ends} of $G$. A graph is \defi{multi-ended} if it has more than one end. Note that given any two finitely generated presentations of the same group, the corresponding \Cg s have the same number of ends. Thus this number, which is known to be one of $0,1,2, \infty$, is an invariant of finitely generated groups.

\subsection{Embeddings in the plane} \label{secDem}	%DON'T CHANGE ORDER; N(x) needed

An \defi{embedding} of a graph \g will always mean a topological embedding of the corresponding 1-complex in the euclidean plane $\R^2$; in simpler words, an embedding is a drawing in the plane with no two edges crossing.

A \defi{face} of an embedding $\sig: G \to \R^2$ is a component of $\R^2 \sm \sig(G)$. The \defi{boundary} of a face $F$ is the set of vertices and edges of \g that are mapped by \sig\ to the closure of $F$. The \defi{size} of $F$ is the number of edges in its boundary. Note that if $F$ has finite size then its boundary is a cycle of \G.

A walk in \g is called \defi{facial} \wrt\ \sig\ if it is contained in the boundary of some face of \sig. 

An embedding of a \Cg\ is called \defi{\pr} if, intuitively, it embeds every vertex in a similar way in the sense that the group action carries faces to faces. Let us make this more precise.
Given an embedding \sig\ of a \Cg\ $G$ with generating set $S$, we consider \fe\ vertex $x$ of \g the embedding of the edges incident with $x$, and define the \defi{spin} of $x$ to be the cyclic order of the set $L:=\{xy^{-1} \mid y\in N(x)\}$ in which $xy_1^{-1}$ is a successor of $xy_2^{-1}$ whenever the edge $xy_2$ comes immediately after the edge $xy_1$ as we move clockwise around $x$. Note that the set $L$ is the same \fe\ vertex of $G$, and depends only on $S$ and on our convention on whether to draw one or two edges per vertex for involutions. This allows us to compare spins of different vertices. Note that if \g is cubic, which means that $|L|=3$, then \ta\ only two possible cyclic orders on $L$, and thus only two possible spins. Call an edge of \g \defi{spin-preserving} if its two endvertices have the same spin in \sig, and call it \defi{spin-reversing} otherwise. Call a colour in $S$ \defi{\pr} if all edges bearing that colour are  spin-preserving or all edges bearing that colour are spin-reversing in \sig. Finally, call the embedding \sig\ \defi{\pr} if every colour is \pr\ in \sig\ (this definition is natural only if \g is cubic; to extend it to the general case, demand that every two vertices have either the same spin, or the spin of the one is obtained by reversing the spin of the other). 

It is straightforward to check that \sig\ is \pr\ \iff\ every \auto\ of \g maps every facial walk to a facial walk. 

It follows from Whitney's theorem mentioned in the introduction that if \g is \tcon\ then its essentially unique embedding must be \pr. \Cg s of connectivity 2 do not always admit a \prem\ \cite{DrSeSeCon}. However, in the cubic case they do; see \cite{cay2con}.

An embedding is \defi{Vertex-Accumulation-Point-free}, or \defi{\vapf} for short, if the images of the vertices have no accumulation point in $\R^2$.

%\end{section}

\section{Known facts}

In this section we recall some easy facts about \Cg s that we will use later. The reader may choose to skip this section and the next.

We begin with a well-known characterisation of \Cg s. Call an edge-colouring of a digraph \g \defi{proper}, if no vertex of \g has two incoming or two outgoing edges with the same colour.
\begin{theorem}[Sabidussi's Theorem \cite{sab,BaAut}] \label{sab}
A properly edge-coloured digraph is a Cayley
graph \iff for every $x, y\in V(G)$ there is a
\auto\  mapping $x$ to $y$.
\end{theorem}

The following classical result was proved by Whitney \cite[Theorem 11]{whitney_congruent_1932} for finite graphs and by Imrich \cite{ImWhi} for infinite ones.

\begin{theorem} \label{imrcb}
Let \g be a \tcon\ graph embedded in the sphere. Then every automorphism of \g maps each facial path to a facial path.
\end{theorem}

This implies in particular that if \sig\ is an embedding of the \tcon\ \Cg\ \G, then the cyclic ordering of the colours of the edges around any vertex of \g is the same up to orientation. In other words, at most two spins are allowed in \sig. Moreover, if two vertices  $x,y$ of \g that are adjacent by an edge, bearing a colour $b$ say, have distinct spins, then any two vertices $x',y'$ adjacent by a $b$-edge also have distinct spins. We just proved
\begin{lemma} \label{lprem}
 Let \g be a \tcon\ planar \Cg. Then every embedding of \g is \pr.
\end{lemma}

Finally, we recall the following fact that is reminiscent of MacLane's planarity criterion.

\begin{theorem}[\cite{cay2con}] \label{Macay}
Let  $\left< S \mid R \right>$ be a \red\ presentation and let $G= Cay \left< S \mid R \right>$ be the corresponding \Cg. If no edge of \g appears in more than two  circuits  induced by relators in $R$, then $G$ is planar and has a \vapf\ embedding the facial cycles of which are precisely the cycles of \g induced by relators in $R$. 
\end{theorem}

%\end{section}

\section{General facts regarding connectivity}

\comment{next 2 probably not needed:
	\begin{lemma} \label{LG1}
 Let $G,H$ be \tcon\ \pccg s each vertex of any of which is incident with two faces of length $m$ and one face of length $n$ with $m,n < \infty$. Then $G \isom H$
\end{lemma}

\begin{lemma} \label{LG1p}
 Let $G,H$ be \pccg s each vertex of any of which is incident with two faces of length $m< \infty$ and one infinite face. Then $G \isom H$
	\end{lemma} 
}

Call a relation of a group presentation \defi{cyclic}, if it induces a cycle in the corresponding \Cg. An involution is by convention not cyclic. Cyclic relations are useful because they allow us to formulate the following lemma. 
\begin{lemma} \label{Liicon}
A cubic \Cg\ is \iicon\ \iff\ each of its generators $a$ is in a cyclic relation.
\end{lemma}
\begin{proof}
Firstly, note that if \g is a cubic \Cg\ that is not \iicon, then it must have a bridge $e$. But then the generator $a$ corresponding to $e$ cannot be in any cyclic relation because a cycle cannot contain a bridge.

Conversely, an edge corresponding to a generator $a$ that is in no cyclic relation cannot lie in any cycle; thus it is a bridge, which means that any of its endpoints separates the graph.
\end{proof}

In many occasions we will use some of the graphs of Table~\ref{table} as building blocks in order to construct more complicated ones. Our next lemma will be useful in such cases, as it will allow us to deduce the fact that the new graphs are \tcon\ from the fact that the building blocks were. 
Let  $K,K_1,K_2$ be subsets of $V(G)$. We will say that $K$ is \defi{$k$-connected in \G}, if for every vertex set $S\subseteq V(G)$ with $|S|<k$, there is a component of $G - S$ containing $K - S$. Similarly, we will say that $K_1$ is \defi{$k$-connected  to} $K_2$ in \G, if for every vertex set $S\subseteq V(G)$ with $|S|<k$, there is a path from $K_1$ to $K_2$ in $G -S$; in particular, $|K_1|,|K_2|\geq k$.

\begin{lemma}[\cite{am}] \label{lkcon}
Let \g be a graph,  %let $\fml{H}$ be a family of subgraphs of \G, 
and %\fe\ $i\in I$
let $\fml{K}$ be a family of subsets of $V(G)$ such that the following three assertions hold:
\begin{enumerate}\addtolength{\itemsep}{-0.5\baselineskip}
\item \label{coi} $\bigcup_I K_i = V(G)$;
\item \label{coii} \Fe\ $i\in I$, $K_i$ is $k$-connected in $G$, and
\item \label{coiii} \Fe\ $i,j\in I$ \ti\ a finite sequence $i= n_0, \ldots, n_r = j$ \st\ ${K_{n_m}}$ is $k$-connected to ${K_{n_{m+1}}}$ in \G\ \fe\ relevant $m$.% \g contains $k$ pairwise disjoint paths from ${K(H_{n_m})}$ to ${K(H_{n_{m+1}})}$ \fe\ relevant $m$.%$0\leq m <r$.
\end{enumerate}
Then \g is $k$-connected.
\end{lemma}

\section{Graphs of connectivity 1 or 2} \label{lowk}

It follows easily from \Lr{Liicon} that if \g is a cubic \Cg\ of connectivity $\kappa(G)=1$, then its group $\Gam(G)$ has one of the following presentations. 
\begin{enumerate}\addtolength{\itemsep}{-0.5\baselineskip}
\item \label{ici} $\left<a,b\mid b^2, a^n\right>$, $n\in \{\infty, 2, 3, \ldots\}$
\item \label{icii} $\left<b,c,d\mid b^2,c^2,d^2, (bc)^n \right>$, $n\in \{\infty, 1, 2, 3, \ldots\}$,
\end{enumerate}
where $n=\infty$ means that the corresponding relator is omitted.

All \Cg s corresponding to these presentations are planar; this follows easily from $\kappa(G)=1$ and the assumption that \g is cubic.

%\section{Graphs of connectivity 2}
\medskip

The planar cubic \Cg s of connectivity 2 were completely analysed in \cite{cay2con}, yielding the following classification. 
\begin{theorem} \label{main2}
Let \g be a planar cubic \Cg\ of connectivity 2. Then precisely one of the following is the case:
\begin{enumerate}\addtolength{\itemsep}{-0.5\baselineskip}
\item \label{i} $G \isom Cay \left<a,b\mid b^2, (ab)^n\right>$, $n\geq 2$; 
\item \label{ii} $G \isom Cay \left<a,b\mid b^2, (aba^{-1}b^{-1})^n\right>$, $n\geq 1$;
\item \label{iii} $G\isom Cay \left<a,b \mid b^2, a^4, (a^2b)^n \right>, n\geq 2$;

\item \label{iv} $G \isom Cay \left< b,c,d \mid b^2, c^2, d^2, (bc)^2, (bcd)^m\right>$, $m\geq 2$;
\item \label{v} $G \isom Cay \left<b,c,d \mid b^2, c^2, d^2, (bc)^{2n}, (cbcd)^m\right>$, $n,m\geq 2$; 
\item \label{vi} $G \isom Cay \left<b,c,d\mid b^2, c^2,d^2, (bc)^n, (bd)^m\right>$, $n,m\geq 2$;
\item \label{vii} $G \isom Cay \left<b,c,d\mid b^2, c^2,d^2, (b(cb)^nd)^m\right>, n\geq 1, m\geq 2$;
\item \label{viii} $G \isom Cay \left<b,c,d\mid b^2, c^2,d^2, (bcbd)^m \right>$, $m\geq 1$;
\item \label{ix} $G \isom Cay\left<b,c,d\mid b^2, c^2, d^2, (bc)^n, cd\right>$, $n\geq 1$ (degenerate cases with redundant generators).
\end{enumerate}
Conversely, each of the above presentations, with parameters chosen in the specified domains, yields a planar cubic \Cg\ of connectivity 2.
\end{theorem}

The above presentations are planar; see \citeIICorplpr.

We use \Tr{main2} to obtain the entries \ref{Ai}--\ref{Aix} of Table~\ref{table}, but we also use these graphs as important building blocks in later constructions of \tcon\ \Cg s. 

\medskip
The following sections are devoted to the \tcon\ case.

\section{Crossings of \mpsc s} \label{secPsc}

Given a plane graph \g and a cycle $C$ of \G, the Jordan curve theorem yields two distinct regions \defi{$IN(C), OUT(C)$} of $\R^2 \sm C$ which we call the \defi{sides} of $C$. Define the  \defi{closed sides} $\overline{IN(C)}, \overline{OUT(C)}$ of $C$ to be the respective union of $IN(C), OUT(C)$ with $C$. 

Call a cycle $C$ of \g a  \defi{\psc} if both sides of $C$ contain infinitely many vertices, and call $C$ a  \defi{\mpsc} if its length is minimal among all \psc s of \G. The \mpsc s will play a very important role in this paper: in most of the group presentations we construct, any relator that does not induce a face boundary will induce a \mpsc. In this section we provide some general facts, that will be useful later, about how pairs of \mpsc s can meet.  

We will say that two cycles $C,C'$ of \g \defi{cross} each other if none of them is contained in the closure of a side of the other; equivalently, if each of $IN(C), OUT(C)$ meets both $IN(C'), OUT(C')$. It turns out that the  ways in which \mpsc s can cross are very restricted.
%\begin{lemma} \label{LII3}
%Let $C,C'$ be two crossing \mpsc s in a cubic graph \G. Then at least two of the regions of $G - C\cup D$ contains finitely many vertices.
%\end{lemma}
%\begin{proof}

We will first consider the simplest case, when two \mpsc s $C,C'$ cross only once, that is, when $C - C'$ consists of two components and so does $C' - C$. In this case,  $\R^2 \sm (C \cup C')$ consists of four regions $A,B,D,F$ as in \fig{crossing}.

\showFig{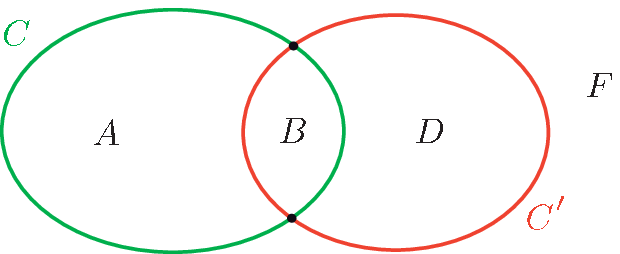}{The four regions of $\R^2 \sm (C \cup C')$.}

Since $C,C'$ are \psc s, both the inside and the outside of each of them contain infinitely many vertices. This immediately implies that either both $A,D$ are infinite or both $B,F$ are infinite. We may assume \obda\ that 
\labtequ{AD}{both $A,D$ are infinite}
since otherwise we could modify the embedding of \g so as to fix $C$ but exchange its inside with its outside, which would have the effect of renaming the regions $A,B,D,F$. 

As \g is cubic, every time $C$ and $C'$ intersect they  must have at least one common edge. This gives rise to the three cases displayed in \fig{crossing3}. In two of those cases (upper half of  \fig{crossing3}) we immediately obtain a contradiction: in each case we obtain two new cycles $K,K'$ (dashed lines) both of which are \ps\ by \eqref{AD}. Now an easy double counting argument shows that the length at least one of $K,K'$ must be less than $\ell:=|C|=|C'|$, the length of a \mpsc\ of \G: indeed, we have $|K|+|K'| < |C|+|C'|$.

\epsfxsize=1\hsize

\showFig{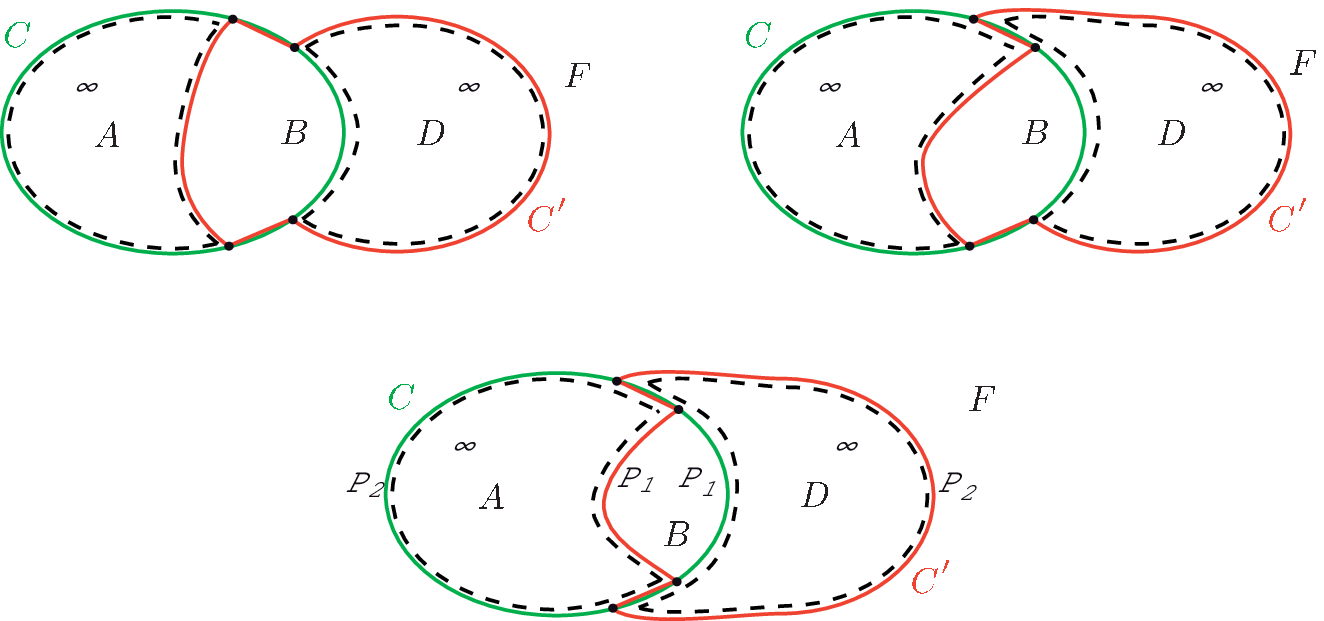}{The three ways in which two \psc s $C,C'$ can cross. The dashed cycles $K,K'$ are also \ps.}

Thus, whenever two \mpsc s $C,C'$ cross only once we must have the situation in the lower half of  \fig{crossing3}, with each of $A$ and $D$ containing infinitely many vertices. Moreover,
\labtequ{BF}{at least one of $B,F$ is finite,}
since otherwise one of the cycles $P_1 \cup P'_1$ and $P_2 \cup P'_2$ is \ps\ and has length less that $\ell$ by a double-counting argument as above.

If $C,C'$ cross more than once then more or less the same arguments apply, but we have to be a bit more careful. We define the regions $A,B,D,F$ as follows. We let $A:= IN(C) \cap OUT(C')$, $B:= IN(C) \cap IN(C')$, $D:= OUT(C) \cap IN(C')$, and   $F:= OUT(C) \cap OUT(C')$. Now each time $C,C'$ cross, the corresponding subpaths must be arranged as in the lower half of  \fig{crossing3} as can be shown by repeating the above arguments, except that this time $K,K'$ need not be cycles but could be more complicated closed walks. Still, we can decompose each of them into a finite collection of cycles, one of which will have to be \ps\ and shorter that $\ell$.
%We define these regions is as follows.
%Pick two vertices $x,y\in C\cap C'$ at which the two cycles cross\sss. These vertices separate $C$ into the paths $xP_1y$ and $yP_2 x$, and $C'$ into the paths $xP'_1y$ and $yP'_2 x$. 

Note that $|P_1|= |P'_1|$ must hold in \fig{crossing3} because if, say, $|P_1|< |P'_1|$ then we can replace $P'_1$ by $P_1$ in $C'$ to obtain the cycle $K'$ which is shorter than $C'$. But by \eqref{AD} $K'$ would then be a \psc, contradicting the fact that $C'$ is shortest possible. Similarly, we obtain $|P_2|= |P'_2|$. Thus, both dotted cycles in the lower half of  \fig{crossing3} have the same length as $C$, and by \eqref{AD} again they are \mpsc s. 

This means that whenever we have two crossing \mpsc s $C,C'$ in a cubic graph \G, 
\labtequ{CII3m}{\ti\ a \mpsc\ $K$ which is the union of a non-trivial subpath of each of $C,C'$ which subpaths witness the fact that $C,C'$ cross.} 
This will imply the following tool that we will use in many cases.
\begin{corollary}\label{CII3p}
Let $C,C'$ be two crossing \mpsc s in a cubic graph \G. Let $Q$ be a maximal common subpath of $C,C'$, and let $F$ be a facial subpath of $C$ or $C'$ that has maximal length among all facial subpaths of all \mpsc s of \G. Then if $Q$ and $F$ have a common endvertex $y$, their edges incident with $y$ are distinct.
\end{corollary}
\begin{proof}
Suppose, to the contrary, that such a path $Q$ finishes with the edge $e=xy$ which is also the last edge of the facial subpath $F$ of $C$ that has maximum length. Since $Q$ was a maximal common subpath of $C,C'$, the edge  $f = y z$ of $C$ following $Q$ is different from the edge  $f' = y z'$ of $C'$ following $Q$. By \eqref{CII3m}, we can combine $C$ and $C'$ into a new \mpsc\ $K$ that comprises a subpath $P$ of $C$ ending with $e$ and a subpath $P'$ of $C'$ starting with $f'$. %, these subpaths 
Recall that the path $F$ is facial, and let  $H$ denote its incident  face. Since $F$ was a maximal facial subpath of $C$, $H$ is not incident with the edge $f\in E(C)$. Thus, as \g is cubic, $H$ is incident with $f'$. Now note that $K$ contains $f'$ by construction, as well as the last edge $e$ of $F$. If $K$ contains all of $F$, then it contains the facial path $F \cup f'$, which contradicts the maximality of the length of $F$ among all facial subpaths of all \mpsc s. Thus $P'$ must interrupt $F$. This however yields a contradiction to the embedding. Indeed, since $H$ is a face, it is contained in either $IN(C)$ or $OUT(C)$. Recall also that $P'$ must by definition be contained in $IN(C)$ or $OUT(C)$. But as $P'$ contains the edge $f'$ which is incident with $H$, both $H$ and $P'$ must be accommodated in the same side of $C$. Thus $P'$ cannot contain any of the edges incident with $F$, since these edges lie in the other side of $C$. This shows that $P'$ cannot interrupt $F$, and we obtain a contradiction that completes the proof.
\end{proof}

%\end{section}

\section{The finite and 1-ended cubic planar \Cg s} \label{secfin}

In this section we analyse the cubic planar \Cg s that are either finite or infinite but with only 1 end. Many, perhaps all, of these graphs were already known. They appear as entries \ref{Aoi}--\ref{Aziv} of Table~\ref{table}, but also provide building blocks for many of the more interesting entries. 

We begin with some general properties of planar \Cg s with at most 1 end.

\begin{theorem} \label{endtcon}
 Every finite or 1-ended cubic \Cg\ is \tcon.
\end{theorem}
This is proved in \cite[Chapter 27, Theorem 3.7.]{BaAut}
for finite \g and in \cite[Lemma 2.4.]{BaGro} for infinite \G.

Combining this with \Lr{lprem} easily yields the following well-known fact.
\begin{lemma} \label{noinff}
 In an 1-ended plane \Cg\ all face-boundaries are finite.
\end{lemma}

Conversely, we have 

\begin{lemma}[\citVapfL] \label{LG3}
 A plane \iicon\ graph with no \psc\ and no infinite face-boundary is either finite or 1-ended.
\end{lemma} 

Our last lemma is

\begin{lemma} \label{facecc}
Let $G= Cay \left< S \mid \mathcal R \right>$ be finite or 1-ended and planar \note{generalisable to a graph with a \vapf\ embedding}, and let $\mathcal R'$ be a set of relations of $\Gamma(G)$ such that every face boundary of \g is induced by some relation in $\mathcal R'$. Then $\Gamma(G) \isom \left< S \mid \mathcal R' \right>$.
\end{lemma}
\begin{proof}
It suffices to show that the edge-set of every cycle $C$ of \g is a sum of edge-sets of finite face-boundaries. This is indeed the case, for as \g is at most 1-ended, there must be a side $A$ of $C$ containing only finitely many vertices, and so $E(C)$ is the sum of the edge-sets of the  face-boundaries lying in $A$. 
\end{proof}

We can now proceed with the main results of this section.

%\begin{corollary}
%Every finite or 1-ended planar \Cg\ has a planar Cayley complex.
%\end{corollary}

\begin{theorem} \label{LG2}
 Let $G= Cay\left<a,b\mid b^2, \ldots \right>$  be planar and finite or 1-ended. Then precisely one of the following is the case:
\begin{enumerate}\addtolength{\itemsep}{-0.5\baselineskip}
\item \label{oi} $G \isom Cay \left<a,b\mid b^2, a^n, (ab)^m\right>$, $n\geq 3$, $m\geq 2$ (and $a,b$ preserve spin);
\item \label{oii} $G \isom Cay \left<a,b\mid b^2, a^n, (aba^{-1}b)^m\right>$, $n\geq 3, m\geq 1$ (and only $a$ preserves spin);
\item \label{oiii} $G \isom Cay \left<a,b\mid b^2, (a^2b)^m\right>$, $m\geq 2$ (and only $b$ preserves spin);
\item \label{oiv} $G \isom Cay \left<a,b\mid b^2, (a^2ba^{-2}b)^m\right>$, $m\geq 1$ (and $a,b$ reverse spin);
\item \label{ov} $G \isom Cay \left<a,b\mid b^2, a^2, (ab)^n\right>$, $n\geq 2$ or $G \isom Cay \left<a,b\mid b^2, ab\right>$ (degenerate cases in which \g is not cubic).
\end{enumerate}
All presentations above are planar.

Conversely, each of the presentations \ref{oi}--\ref{oiv}, with parameters chosen in the specified domains, yields a planar, finite or 1-ended, non-trivial cubic \Cg.
\end{theorem}
\begin{proof}
For the forward implication, let $G= Cay\left<a,b\mid b^2, \ldots \right>$ be planar, with at most one end. \g is \tcon\  by \Tr{endtcon}, unless $a$ is an involution too in which case we have one of the degenerate cases of \ref{ov}. In all non-degenerate cases, the (essentially unique) embedding is consistent \wrt\ spin by \Lr{lprem}. Thus we have precisely one of the following cases.  

{\bf Case \ref{oi}:} both $a,b$ preserve spin. 

Since $a$ preserves spin the walk $P:=\zero, a, a^2, a^3, \ldots$ is facial. By \Lr{noinff} $P$ cannot be a ray, so it spans a finite cycle $C$ of length $n$ say. Similarly, the walk $Q:=\zero, a, ab, aba, abab, \ldots$ is facial because $b$ also preserves spin, and by the above argument it must span a finite cycle $C'$ with edges alternating between $a$ and $b$. Note that both $C, C'$ are face boundaries. Moreover, by \Tr{imrcb} every face-boundary of $G$ is a translate of one of $C, C'$. By \Lr{facecc} this means that the relations $a^n, (ab)^m$, inducing $C$ and $C'$ respectively, combined with $b^2$ yield a presentation of $\Gam(G)$. Thus
$G \isom Cay \left<a,b\mid b^2, a^n, (ab)^m\right>$ with $m\geq 2$ as claimed. %Note that if $m=2$ then this \Cg\ is a finite prism, so $m\geq 3$.

{\bf Case \ref{oii}:} $a$ preserves spin and $b$ reverses spin.

As in Case \ref{oi} we conclude that \g has finite $a$-coloured faces, induced by a relation $a^n$. As $b$ reverses spin now, a walk $Q$ as above is not facial any more, and instead $ \zero, a, ab, aba^{-1}, aba^{-1}b, \ldots$ is facial. By similar arguments we obtain the desired presentation $G \isom Cay \left<a,b\mid b^2, a^n, (aba^{-1}b)^m\right>$, $n\geq 3, m\geq 1$.

{\bf Case \ref{oiii}:} $b$ preserves spin and $a$ reverses spin.

The walk $ \zero, a, aa, aab, aaba,  \ldots$ is now facial, and spans a finite cycle $C$ of length $3m$ for some $m \geq 2$. It is also straightforward to check that every facial path of \g is of that form, in other words, every face boundary is a translate of $C$. Again by \Lr{facecc} we obtain the desired presentation $G \isom Cay \left<a,b\mid b^2, (a^2b)^m\right>$, $m\geq 1$.

{\bf Case \ref{oiv}:} both $a,b$ reverse spin. 

This case is similar to the previous one, except that the $a$ edges on a facial walk do not all have the same direction now, but instead their directions alternate after each $b$ edge. We thus obtain $G \isom Cay \left<a,b\mid b^2, (a^2ba^{-2}b)^m\right>$, $m\geq 1$.

\medskip

For the converse assertion, given any presentation of  the form \ref{oi}, let us show that the \Cg\ \g is planar. We begin with an auxiliary plane graph $H$, namely, the graph of the regular tiling of the sphere, euclidean plane, or hyperbolic plane with $n$ $m$-gons meeting at every vertex. The existence of $H$ is well-known and not hard to prove \cite{MitCon}. Note that $H$ is a vertex-transitive graph: for any two vertices $x,y\in V(H)$ it is straightforward to inductively construct an isomorphism between the balls of radius $r$ around these vertices. To obtain \g from $H$, replace every vertex $x$ of $H$ with a cycle $C_x$ of length $n$, and join each edge incident with $x$ to a distinct vertex of $C_x$, keeping the cyclic ordering, so that the graph remains planar. Then, assign to each edge of $C_x$ the label $a$, and direct it in such a way that $C_x$ is oriented clockwise. Moreover, assign to every other edge, coming from an edge of $H$, the label $b$. Let \g be the resulting coloured graph. We claim that \g is the \Cg\ corresponding to presentation \ref{oi}. Indeed, the fact that \g is a \Cg\ follows easily from Sabidussi's \Tr{sab} and the fact that $H$ was vertex-transitive. The fact that \g has the desired presentation now follows from the forward implication which we have already proved since, by construction, all edges of \g preserve spin. 

Given any presentation of  the form \ref{oii}, it is now easy to construct an embedding of the corresponding \Cg\ \G: we can start with a graph $G'$ of type \ref{oi} with parameters $n, m':=2m$, and then reverse the orientation of `every other' $a$-cycle to obtain $G$. More precisely, let $G' \isom Cay \left<a,b\mid b^2, a^n, (ab)^{2m}\right>$, and define a bipartition $\{X,Y\}$ of the $a$-cycles of $G'$ by letting $X$ (resp.\ $Y$) be the set of those $a$-cycles that can be reached from a fixed vertex $o\in V(G)$ by a path containing an even (resp.\ odd) number of $b$-labelled edges. To see that this is indeed a bipartition, note that every relation in the group of  $\left<a,b\mid b^2, a^n, (ab)^{2m}\right>$ contains an even number of appearances of the letter $b$. Note moreover, that any two $a$-cycles that are connected by a $b$ edge lie in distinct classes of this bipartition. Thus, if we reverse the orientation of every $a$-cycle lying in $X$ we obtain a plane  graph \g in which only the $a$  edges preserve spin. Again, we can check that \g is a \Cg\ using Sabidussi's theorem, and we apply the forward implication to show that \g has the presentation \ref{oii}.

We handle case \ref{oiii} similarly to case \ref{oi}, except that now the auxiliary graph $H$ is obtained by contracting the $b$-edges instead of the $a$-labelled ones. To achieve the desired orientation, replace every $b$-edge by a $2$-cycle, and orient all these cycles clockwise; then make sure that for every face boundary of type $(a^2b)^m$ all edges are oriented in the same direction. 

A graph of type \ref{oiv} can be obtained by one of type \ref{oiii} by reversing the orientation of every other $a$-path, similarly to the above reduction of type \ref{oii} to type \ref{oi}.

Finally, the \Cg s of type \ref{ov} are finite and easy to construct.

It follows easily from our construction that all our presentations are planar.
\end{proof}

%The finite ... two infinite families of prisms ... \cite[Chapter 27.3.9]{BaAut}

We now proceed with the case where \g is generated by three generators.

\begin{theorem} \label{LG5}
Let $G= Cay\left<b,c,d\mid b^2, c^2, d^2, \ldots \right>$ be planar and finite or 1-ended. Then precisely one of the following is the case:
\begin{enumerate}\addtolength{\itemsep}{-0.5\baselineskip}
\item \label{zi} $G \isom Cay\left<b,c,d\mid b^2, c^2, d^2, (bcd)^n \right>$, $n\geq 1$ (all colours preserve spin);
\item \label{zii} $G \isom Cay\left<b,c,d\mid b^2, c^2, d^2, (cbcdbd)^n \right>$, $n\geq 1$ (only $c,d$  preserve spin);
\item \label{ziii} $G \isom Cay\left<b,c,d\mid b^2, c^2, d^2, (bc)^n, (bdcd)^m\right>$, $n\geq 2, m\geq 1$ (only $d$  preserves spin);
\item \label{ziv} $G \isom Cay\left<b,c,d\mid b^2, c^2, d^2, (bc)^n, (cd)^m, (db)^p \right>$, $n,m,p \geq 2$ (all colours reverse spin);
\item \label{zv} $G \isom Cay\left<b,c,d\mid b^2, c^2, d^2, (bc)^n, cd\right>$, $n\geq 1$ (degenerate, non-\tcon\ cases with redundant generators).
\end{enumerate}
All presentations above are planar.

Conversely, each of the above presentations, with parameters chosen in the specified domains, yields a planar, finite or 1-ended, non-trivial cubic \Cg.
\end{theorem}
\begin{proof}

By the arguments of \Tr{LG2} \g uniquely embeds in the sphere, unless we are in the degenerate case \ref{zv}, and so the colours behave consistently \wrt\ spin. We thus have the following non-degenerate cases. Recall that by \Lr{noinff} all face boundaries of \g are finite; we are going to tacitly make use of this fact in all cases. 

{\bf Case \ref{zi}:} all colours preserve spin. 

Consider the path $\zero, b, bc$, which is facial for some face boundary $F$ since it only has two edges. Since $c$ preserves spin, the next edge on $F$ must be coloured $d$. Similarly, the edge after that must be coloured $b$ since $d$ preserves spin. Continuing like that we conclude that the edges of $F$ follow the pattern $bcdbcdb\ldots$, in other words, $F$ can be induced by the relation $(bcd)^n$. Since all edges preserve spin, any two faces of \g that share an edge $e=xy$ can be mapped to each other by the automorphism of \g exchanging $x$ with $y$, where we are using the fact that all edges correspond to involutions. Thus all face boundaries of \g have the same form as $F$. Similarly to \Tr{LG2}, we can now apply  \Lr{facecc} to conclude that $G \isom Cay\left<b,c,d\mid b^2, c^2, d^2, (bcd)^n \right>$, $n\geq 1$.

{\bf Case \ref{zii}:} precisely two colours, $c ,d$ say, preserve spin.

In this case, at every vertex of \g the situation looks locally like \fig{fzii}, and it is straightforward to check that every face boundary is of the form $(cbcdbd)^n$. Indeed, as $c ,d$ preserve spin, every facial walk of the form $cd$ is a subwalk of a facial walk of the form $bcdb$. Moreover, as $b$ reverses spin, every facial walk of the form $cb$ is a subwalk of a facial walk of the form $cbc$ and every facial walk of the form $db$ is a subwalk of a facial walk of the form $dbd$. Furthermore, there is no facial walk of the form $bcb$ or $bdb$. Combining these facts one obtains that every face boundary is indeed of the form $(cbcdbd)^n$. Thus \Lr{facecc} yields $G \isom Cay\left<b,c,d\mid b^2, c^2, d^2, (cbcdbd)^n \right>$, $n\geq 1$.

\showFig{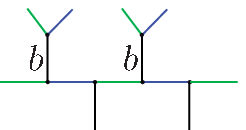}{The local situation around any vertex in case \ref{zii}.}

{\bf Case \ref{ziii}:} precisely one colour, $d$ say, preserves spin.

Then a walk alternating in $b,c$ is facial, and so \g has two coloured faces, induced by a relation $(bc)^n$ for some $n\geq 2$. It is straightforward to check, by observing the spin behaviour, that  every facial walk containing a $d$-edge is of the form $bdcdbd\ldots$. Thus, by \Lr{facecc} again, $G \isom Cay\left<b,c,d\mid b^2, c^2, d^2, (bc)^n, (bdcd)^m\right>$, $n\geq 2, m\geq 1$.

{\bf Case \ref{ziv}:} all colours reverse spin. 

Similarly to Case \ref{ziii} it follows that any two-coloured walk is facial, and so at every vertex we have three different kinds of incident face boundaries, each a two-coloured cycle. By \Lr{facecc} we obtain $G \isom Cay\left<b,c,d\mid b^2, c^2, d^2, (bc)^n, (cd)^m, (db)^p \right>$, $n,m,p \geq 2$.

\medskip

The converse assertion can again be proved using explicit constructions similarly to what we did in the proof of \Tr{LG2}, but these constructions become more complicated now and so we will follow a different, and interesting, approach that makes use of \Tr{Macay}. Given any presentation of one of these four forms, we will check that the presentation is simple. Moreover, it \istc\ that each of these presentations satisfies the conditions of \Tr{Macay}, with every edge of \g appearing in precisely two  circuits  induced by the specified relators. Thus, by that theorem, the corresponding \Cg\ is planar, and has a \vapf\ embedding \sig\ in which the finite face boundaries are induced by those relators, and each edge lies in precisely two such boundaries. This means that all face boundaries in \sig\ are finite. Moreover, \sig\ has no \psc s because it is \vapf. Thus we can apply \Lr{LG3} which yields that \g is 1-ended if it is infinite.

It only remains to check that our presentations are simple indeed. To see that \ref{zi} is simple, note that even if we set $d=1$ we can prove no subword of $(bc)^n$ to be a relation since  $(bc)^n$ is simple in the subgroup generated by $b,c$. Similarly, consider the subgroups of \ref{zii} generated by $b,d$ and $b,c$ to see that the presentation is simple. To see that \ref{ziv} is simple, consider the subgroup $F$ consisting of all elements that can be presented by a word of even length in the letters $b,c,d$; this is indeed a subgroup since every relation of the original group contains an even number of letters. Note that $F$ is generated by $bc,cd,db$, and in fact one of these generators is redundant. Thus $F \isom \left< a',b' \mid (a')^n, (b')^m, (a'b')^q\right>$, where $a':=bc, b':=cd$.
Note that if we impose $b'^2=1$ then this presentation reduces to \ref{oi} of \Tr{LG2}, and we implicitly checked there that $(a')^n$ is simple in that group. Thus $(a')^n$ must be simple in $F$, which means that $(bc)^n$ is simple in our original group \Gam\ as the element $bc$ cannot have a smaller order in \Gam\ than in a subgroup of \Gam. By symmetry, all relations in \ref{ziv} are simple. Instead of showing that presentation \ref{ziii} is also simple, let us rather explicitly construct the desired \Cg\ from one of type \ref{ziv}. For this, let $G'= Cay\left<b,c,d\mid b^2, c^2, d^2, (bc)^n, (cd)^{2m}, (db)^{2m} \right>$. We have already proved that $G'$ has an embedding with all edges reversing spin, and we have observed  that the set of elements that can be presented by a word of even length forms a subgroup $F$, which of course has index 2. Note that any two adjacent vertices of $G'$ lie in distinct left cosets of $F$. Now let $G$ be the graph obtained from $G'$ by exchanging, \fe\ $f\in F$, the labels of the edges incident with $f$ labelled $b$ and $c$. Note that every vertex of $G$ is still incident with all three labels by the above observation. It is easy to prove that \g is a \Cg\ using Sabidussi's theorem. Moreover, the $d$-edges now preserve spin. Thus, by the forward implication, \g has the desired presentation \ref{ziii}. Finally, presentation \ref{zv} gives rise to a finite \Cg\ which is easy to embed. 

It follows easily from our construction that all our presentations are planar.
%Similarly, presentation \ref{ziii} also only allows relations with an even number of letters, and so we can define $F$ as above, and we obtain $F \isom \left<a,b' \mid (b')^n, (b'a^2)^m\right>$, where this time we let $a:=cd, b':=bc$. Imposing $b'^2=1$ reduces to \ref{oiii} of \Tr{LG2}, from which we deduce that $(bdcd)^m$ is simple in the original group.  

%For \ref{ziii}, note that even if we impose $b=c$, the relation $(bdcd)^m$ remains simple, thus that relation is simple in the original presentation. 
%Moreover, the  relation $(bc)^n$ of \ref{ziii} is simple \sss.
\end{proof}

\section{The planar multi-ended \Cg s with 2 generators} \label{secI}

Having already characterised the 1-ended cubic planar \Cg s in the previous section, we turn our attention to our main object of interest, the planar, \tcon, multi-ended \Cg s. In this section we consider those generated by two generators, one of which must be an involution. We will distinguish two cases according to whether the other generator has finite or infinite order, discussed separately in the following two subsections. We will obtain \plpr s\ for each of those graphs, as well as explicit constructions of their embeddings. Our results are summarised in Theorems \ref{TIa2} and \ref{TIbx} below.

\subsection{Graphs with monochromatic cycles} \label{secIa}

In this section we consider the case when $a^m=1$ for some $m>2$.
It turns out that in this case 
\labtequ{arevI}{$a$ reverses spin,}
because of the following lemma.

\begin{lemma} \label{apres}
Let $G= Cay\left<a,b\mid b^2, \ldots \right>$ be a \iicon\ \Cg\ with a \prem\ in which $a$ preserves spin and \st\ $a$ has finite order. Then \g has at most one end.
\end{lemma}
\begin{proof}
We begin by showing that 
\labtequ{nops}{\g does not have a \psc.}
Indeed, suppose to the contrary that $C$ is a \mpsc, and choose $C$ so that it has a facial subpath $F$ that has maximum length among all \mpsc s and all their facial subpaths. We distinguish three cases, all of which will lead to a contradiction.

{\bf Case I:} one of the end-edges $e$ of $F$ is labelled $a$.

In this case, we can rotate the finite $a$-cycle containing $e$ by a \auto\ of \g  to translate $C$ to a further \mpsc\ $C'$ that crosses $C$ in such a way that $e$ is the last edge of a common subpath of $C$ and $C'$; see \fig{fIa12} (left). Indeed, since $e$ is the last edge of $F$, the edge of $F$ before it must have been labelled $b$, and the edge of $C$ following $e$ is labelled $a$ too, so that such a rotation is possible. This crossing immediately contradicts \Cr{CII3p}.

\epsfxsize=1\hsize
\showFig{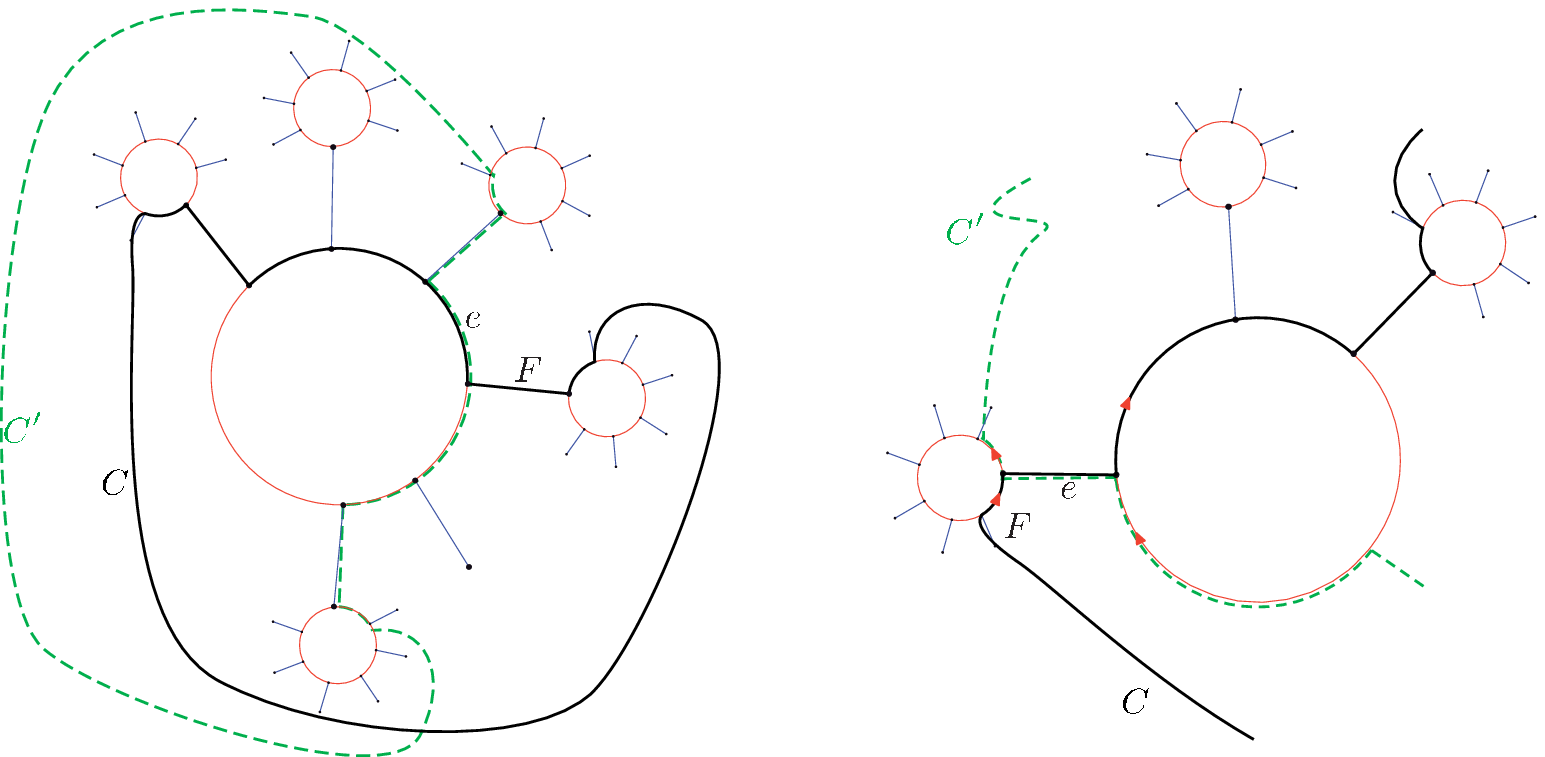}{Creating a crossing in Case I (left) and Case II (right).}

{\bf Case II:} both end-edges of $F$ are labelled $b$, and $b$ reverses spin.

In this case we can apply the \auto\ of \g that exchanges the endvertices of some end-edge $e$ of $F$ to translate $C$ to a new \mpsc\ $C'$ that crosses $C$;  see \fig{fIa12} (right). Again, the crossing we obtain  contradicts \Cr{CII3p}.

%\showFig{fIa2}{Creating a crossing in Case II.}

{\bf Case III:} both end-edges of $F$ are labelled $b$, and $b$ preserves spin.

Let $e=yz$ be an end-edge of $F$. Since $b$ is an involution, the edge $g=xy$ of $F$ preceding $e$ is labelled $a$, and so $g$ is not an end-edge of $F$, which means that $F$ has at least one more edge. The edge $f=wx$ of $F$  preceding $g$ is labelled $b$ again, because otherwise $F$ would not be facial, see \fig{fIa3}. We claim that the \auto\ $T$ of \g that maps $z$ to $x$ translates $C$ to a cycle $C'=T(C)$ that crosses $C$. To see this, let \cf\ be the face whose boundary contains $F$, and note that the edge $h=zv$ following $F$ on $C$ is not incident with \cf, for $F$ is maximally facial.  Assume \obda\ that $g$ is directed from $x$ to $y$. Then, as $b$ preserves spin, $h$ is directed from $v$ to $z$. This implies that $T(h)\neq g$, and so $C'$ has an edge, namely $T(h)$, that lies in the side of $C$ not containing \cf. 

\showFig{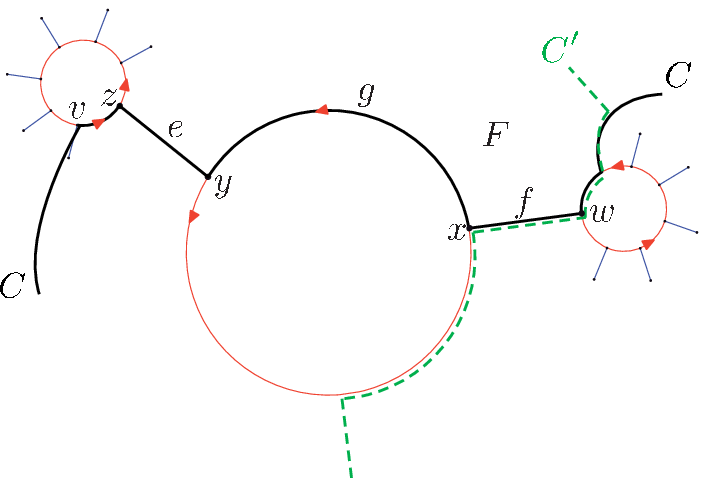}{Creating a crossing in Case III.}

On the other hand, $T(F)$, which is a subpath of $C'$ since $F\subseteq C$, is facial by \Tr{imrcb}.
Since $T(e)= f$ and we are assuming that both $a$ and $b$ preserve spin, it follows that $T$ `maps' \cf\ to itself, more precisely, that $T(F)$ is contained in the boundary of \cf. As $F$ is also  contained in the boundary of \cf, and $|F| = |T(F)|$, this means that if we start at $x$ and walk around \cf\ in the direction of $w$, then $T(F)$ spans more edges of the boundary of \cf\ than $F$ does. Since $T(F)\subseteq C'$, this proves that $C'$ has an edge in the side of $C$ that contains \cf. Combined with our earlier observation that $C'$ also meets the other side of $C$ proves that $C'$ crosses $C$, and in fact so that one of their common subpaths ends with $f$. As $f$ is also the final edge of $T(F)$, a facial subpath of maximum length, this contradicts \Cr{CII3p} again.

Thus, in all three cases we obtained a contradiction, and so we have established \eqref{nops}. We can now exploit this fact to prove our next claim.
\labtequ{finfIa}{Every face boundary of \g is finite.}
To see this, let $C$ be a cycle containing an edge $e$ coloured $b$; such a cycle exists because \g is \iicon. Now as $C$ cannot be \ps\ by \eqref{nops}, one of its sides contains only finitely many vertices. Thus, all faces contained in that side have a finite boundary. This means that $e$ itself is on a finite face boundary $F$. Now consider one of the $a$-cycles $D$ incident with $e$. Rotating along $D$ by \auto s of \g we can map $F$ to the other face incident with $e$. Since we can also map $e$ to any other $b$ edge by a \auto of \G, \Tr{imrcb} now implies that every face boundary containing a $b$ edge is finite. As we are assuming that $a$ preserves spin, any face boundary not containing a $b$ edge is a finite $a$-cycle of length $m$, the order of $a$. These two observations together prove \eqref{finfIa}.

We can now apply \Lr{LG3}, using \eqref{nops} and \eqref{finfIa}, to prove that \g %, contrary to our assumption, 
has at most one end. 
\end{proof}

It follows from \eqref{arevI} that the order $m$ of $a$ is even, since the $b$ edges incident with any $a$-cycle $C$ must alternate between the two sides of $C$ (\fig{fIa2}).

\showFig{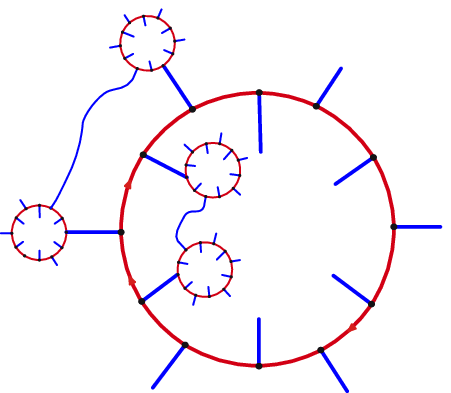}{An $a$-cycle  and some of its translates. Every $a$-edge reverses spin.}

Now consider the subgroup $\Gam_2$ of \Gam\ spanned by $a^2$ and $b$. We claim that  $\Gam_2$ is a proper subgroup of \Gam; in fact, that for every $a$-cycle $C$ of \Gam\ spanned by $a$, at most half of the vertices of $C$ lie in $\Gam_2$. To see this, note that if $x,y$ are two elements of $\Gam_2$, then there is an \pth{x}{y}\ $P$ in \g the $a$ edges of which can be decomposed into incident pairs. Now using \fig{fIa2} it \ises\ that whenever such a path $P$ meets an $a$-cycle $C$ of \G, the two edges of $P$ incident with $C$ lie in the same side of $C$. In other words, $P$ cannot cross any $a$-cycle $C$ of \G. This fact easily implies our claim.

By the same token, given an embedding \sig\ of \G, we can modify \g and \sig\ to obtain a \Cg\ $G_2$ of $\Gam_2$, \wrt\ the generating set $\{a^2, b\}$, and an embedding $\sig_2$ of $G_2$ as follows. For every $a$-cycle $C$ of \g that contains a vertex in $\Gam_2$, delete all vertices and edges in the side of $C$ that does not meet $\Gam_2$; such a side exists by the above argument. Let $G'_2$ be the graph obtained after doing so for every such cycle. Then, suppress all vertices of $G'_2$ that now have degree two; that is, replace any $a$-labelled path $xPy$ of length two whose middle vertex now has no incident $b$ edge by a single $x$-$y$~edge directed the same way as $P$, bearing a new label $z$ (corresponding to the generator $a^2$), to obtain $G_2$. 

We will soon see that $G_2$ uniquely determines \G. But let us first look at $G_2$ more closely. To begin with, using \Lr{Liicon} it is easy to show that 
\labtequ{g2con}{$G_2$ is \iicon.}
%
%The cubic planar \Cg s of connectivity two where characterized in \cite{cay2con}, and it follows from that characterization that $G_2$ is such a graph \iff\ $G_2 \isom Cay \left<z,b \mid b^2, z^4, (z^2b)^2 \right>$ where $z=a^2$.
%We have to distinguish two cases: if $\kap(G_2)=2$, then by the main result of \cite{cay2con} we have  If $\kap(G_2)=3$ then we 
Note that the embedding $\sig_2$ is by construction a \prem\ of $G_2$, but the $z$-labelled edges now preserve spin. Applying \Lr{apres} to $G_2$ thus implies that 
\labtequ{g21end}{$G_2$ has at most one end.}
It now follows from our characterization of such graphs in \Sr{secfin} that, 
\labtequ{presg2}{either $G_2  \isom Cay \left<z,b\mid b^2, z^n, (zb)^m\right>$, $n\geq 3$, $m\geq 2$\\ or $G_2 \isom Cay \left<z,b\mid b^2, z^n, (zbz^{-1}b)^m\right>$, $n\geq 3, m\geq 1$.}
%  $G_2 \isom Cay \left<z,b\mid b^2, z^2\right>$.}
%
Indeed, since \g has a monochromatic cycle, it must belong to one of the types \ref{oi}, \ref{oii} or \ref{ov} of \Tr{LG2}. However, the third type can immediately be eliminated, as it would either imply that a pair of vertices of $G_2$ adjacent by a $z$ edge disconnects $G$, which cannot be the case as \g is assumed to be \tcon, or it would imply that \g is a graph on four vertices. (If we drop the assumption that \g be \tcon\ though then such a graph does exist, and it is unique: it is described in \citeCayIIFigtarget.)

We are going to use the presentation \eqref{presg2} of $G_2$ to obtain a  presentation of $G$. To achieve this, we are going to find some relations of \Gam\ \st\ the set of cycles induced by these relations generates $\ccfg$. In fact, these relations are just the ones that appear in the presentation \eqref{presg2}. Intuitively, one way to prove this is as follows. Given an arbitrary cycle $C$ of \G, consider the finitely many translates $G_1, G_2, \ldots G_{k_C}$ of $G'_2$ in \g that intersect $C$, and observe that $C$ can be written as a  sum of cycles of the $G_i$. Now as any cycle in $G_2$ (and $G'_2$) is a sum of cycles induced by the relators in \eqref{presg2}, our claim follows by \Lr{relcc}. We are going to use a similar argument in several occasions throughout this paper, and rather than repeating the argumentation each time, we are going to use the following more abstract to obtain a rigorous proof of the fact that \g has the claimed presentation. 
\begin{lemma}[\cite{am}] \label{ccfdec}
Let \g be any graph, and let $\cx$ be a set of subgraphs of \g with the following properties:
\begin{enumerate}
\item \label{cci} $\bigcup_{H\in \cx} H = G$,
\item \label{ccii} no edge of \g lies in infinitely many elements of \cx; and
%\item \label{cciii} There is an enumeration $H_0,H_1, \ldots$ of $\cx$ \st\ \fe\ $i>0$ \ti\ a $j=j(i)<i$ and a common cycle $F_i$ of $H_i$ and $H_j$ \st\ $F_i$ separates $H_i$ from $\bigcup_{k<i} H_k$.
\item \label{cciiip} there is a tree $T(\cx, E_T)$ on \cx\ \st\ \fe\ edge $e\in E_T$, joining $H_i$ to $H_j$ say, \ti\ a common cycle $F_i$ of $H_i$ and $H_j$ \st\ $F_i$ separates $\bigcup_{H \in V(T_1)} H$ from $\bigcup_{H \in V(T_2)} H$ where $T_1,T_2$ are the two components of $T - e$.
\end{enumerate}
Then, \fe\ choice $(\cf_H)_{H\in \cx}$ of a generating set $\cf_H$ of $\cc_f(H)$ for each element $H$ of \cx, the union $\bigcup_\cx \cf_H$  generates $\ccfg$.
%Then, \fe\ choice \seq{\cf}, of a generating set $\cf_i$ of $\cc_f(H_i)$ for each element $H_i$ of \cx, the union $\bigcup \cf_i$  generates $\ccfg$.
\end{lemma}

%The last requirement \ref{cciii} can be replaced by the following, which we will use often in this paper.
%\labtequ{cciiip}{There is a tree $T(\cx, E_T)$ on \cx\ \st\ \fe\ edge $e\in E_T$, joining $H_i$ to $H_j$ say, \ti\ a common cycle $F_i$ of $H_i$ and $H_j$ \st\ $F_i$ separates $\bigcup_{H \in V(T_1)} H$ from $\bigcup_{H \in V(T_2)} H$ where $T_1,T_2$ are the two components of $T - e$.}
%
%To see that \eqref{cciiip} implies \ref{cciii} indeed, fix a root $H^r\in V(T)$ and enumerate $V(T)$ as $H^r= H_0, H_1, \ldots$ in such a way that whenever $H_j$ lies in the path in $T$ from $H_i$ to the root $H^r$ we have $j\leq i$; this can easily be achieved using a diagonal proceedure.

Now in order to apply this to our particular \Cg\ \G, consider \fe\ vertex $x\in V(G)$ the left coset of $\Gam_2$ in \Gam\ containing $x$, and let $H^x\subseteq G$ be the corresponding copy of $G'_2$; in other words, $H^x$ is the image of $G'_2$ under the \auto\ $x$. We are now going organize these copies into an auxiliary tree $T$ satisfying \ref{cciiip}. The vertex set of $T$ is the set of graphs $H^x$, for all vertices $x$ of $G$. Two vertices $H^x, H^y$ of $T$ are joined by an edge \iff\ they share an $a$-coloured cycle of \G. It follows easily from the definition of $G'_2$ that $T$ is indeed a tree, and that it satisfies \ref{cciiip}. The interested reader will be able to check that $T$ is in fact a Dunwoody structure tree \cite{dicks_dunw} of \G.

It is clear that  this set $\{H^x \mid x\in V(G)\}$ satisfies \ref{cci} and \ref{ccii} of \Tr{ccfdec}. 
Before we apply \Tr{ccfdec}, it remains to define the generating sets $\cf_x$. For this, let $\cf$ be the set of cycles of \g induced by the relators in the presentation \eqref{presg2} ---for which have distinguished two cases--- after replacing each appearance of the letter $z$ by the word $a^2$. Note that all these cycles are contained in $G'_2$. Moreover, \cf\ generates $\cc_f(G'_2)$ by \Lr{relcc} since  \eqref{presg2} is a presentation of $\Gam_2$. 

Thus, defining $\cf_x$ to be the image of \cf\ under the automorphism of \g that maps $G'_2$ to its copy $H^x$, we meet the requirement that $\cf_x$ generate $\cc_f(H^x)$ \fe\ $x$, and \Tr{ccfdec} yields that $\bigcup \cf_x$  generates $\ccfg$. By the second sentence of \Lr{relcc} and the definition of \cf\ it now follows that 
%
%\labtequ{presg2}
{either $G \isom Cay \left<a,b\mid b^2, a^{2n}, (a^2b)^m\right>$, $n\geq 3$, $m\geq 2$ or $G \isom Cay \left<a,b\mid b^2, a^{2n}, (a^2ba^{-2}b)^m\right>$, $n\geq 3, m\geq 1$.}
The first (respectively second) case occurs if $b$ preserves (resp.\ reverses) spin, as can be seen by applying \Tr{LG2} to $G_2$. Thus we have 
\begin{theorem} \label{TIa2}
Let $G= Cay\left<a,b\mid b^2, \ldots \right>$ be a \tcon\ \Cg\ with more than one end % and a \prem, 
and suppose that $a$ has finite order. Then $a$ reverses spin. If $b$ preserves spin then\\ $G \isom Cay \left<a,b\mid b^2,  (a^2b)^m; a^{2n}\right>, n\geq 3, m\geq 2$.\\ If $b$ reverses spin then\\ 
$G \isom Cay \left<a,b\mid b^2,  (a^2ba^{-2}b)^m; a^{2n}\right>$, $n\geq 3, m\geq 1$.

In both cases, the presentation is planar. %face boundaries of \g are the (finite) cycles induced by the last relator. 

Moreover, \g is the Mohar amalgamation of $G_2 \isom Cay \left<z,b\mid b^2, z^n, (zb)^m\right>$ or $G_2 \isom Cay \left<z,b\mid b^2, z^n, (zbz^{-1}b)^m\right>$ with itself.

Conversely, each of these presentations, with parameters chosen in the specified domains, yields a \Cg\ as above.
\end{theorem}

Here, a \defi{Mohar amalgamation} is the operation of \fig{fiamal} (iii) described in the Introduction.
%Conversely, each of the presentations in \Tr{TIa2}, with parameters chosen in the specified domains, yields a \Cg\ with the desired properties:

%\begin{theorem} \label{TIa2con}
%For every $n\geq 3, m\geq 2$ the cubic \Cg\ $Cay \left<a,b\mid b^2, a^{2n}, (a^2b)^m\right>$ is planar, \tcon\ and has more than one end.%  and a \prem.

%Similarly, \fe\ $Cay \left<a,b\mid b^2, a^{2n}, (a^2ba^{-2}b)^m\right>$, $n\geq 3, m\geq 1$ is planar, \tcon\ and has more than one end.
%\end{theorem}
\begin{proof}
The forward implication was proved in the above discussion. It remains to prove the converse and the fact that these graphs can be obtained by the claimed Mohar amalgamations. We will prove both these assertions simultaneously. For this, given one of the above presentations, consider the auxiliary presentation  $\Gam_2 = \left<z,b\mid b^2, z^n, (zb)^m\right>$ or $\Gam_2 = \left<z,b\mid b^2, z^n, (zbz^{-1}b)^m\right>$ obtained by replacing $a^2$ by $z$ throughout. Applying \Tr{LG2} \ref{oi} or \ref{oii} to this presentation shows that the corresponding \Cg\ $G_2$ is finite or 1-ended, and has an embedding in which all monochromatic cycles induced by $z^n$ bound faces. The other relation also induces facial cycles but we will not use this fact. Now construct a graph \g as the Mohar amalgamation of $G_2$ with itself \wrt\ the $z$-monochromatic cycles, orienting the pasted discs in such a way that all $b$ edges preserve (respectively, reverse) spin if the presentation we started with was of the first (resp.\ second) kind.

It follows easily from Sabidussi's theorem that the plane edge-coloured graph \g\ just constructed is a \Cg. Moreover, \g has infinitely many ends: by construction, any $z$-coloured cycle separates two infinite components. We claim that  \g is \tcon. To prove this, suppose $\{x,y\} \subset V(G)$ separates \G. %Recall that $G_2$ is \tcon\ by \Tr{endtcon}. Thus $\{x,y\}$ cannot separate 
It \ises\ that by the construction of \g the vertices of any $z$-coloured cycle cannot be separated by $\{x,y\}$. Moreover, as $G_2$ is itself \tcon\ by \Tr{endtcon}, no two $z$-coloured cycles that lie in a common copy of $G_2$ can be separated by $\{x,y\}$. But for any two  $z$-coloured cycles $Z,Z'$ of \g \ti\ by construction a finite sequence $Z_1= C_1, C_2, \ldots, C_k= Z'$ \st\ $C_i, C_{i+1}$ are $z$-coloured cycles lying in a common copy of $G_2$. Applying \Lr{lkcon} using these two facts yields that \g is \tcon\ as claimed. 

We can now apply the forward implication of \Tr{TIa2} to this graph \G; we thus obtain that the corresponding group has the desired presentation, namely the one we used for the construction of \G.
\end{proof}

\subsection{Graphs without monochromatic cycles} \label{secIb}

We now consider the case when $a$ has infinite order, and so \g has no monochromatic cycles. Instead, the $a$ edges span double rays in \G. Also in this case we will be able to prove that 
\labtequ{arevIb}{$a$ reverses spin,}
because of the following lemma which is similar to \Lr{apres}.

\begin{lemma} \label{apres2}
Let $G= Cay\left<a,b\mid b^2, \ldots \right>$ be a \tcon\ planar \Cg\ in which $a$ has infinite order. Then $a$ reverses spin.  %Then \g has at most one end.
\end{lemma}
\begin{proof}
Suppose, to the contrary, that $a$ preserves spin.  We will show that \g must have a \psc. For this, pick two vertices $x,y$ that lie in the same $a$-coloured double ray $R$ of \G. As \g is \tcon, there are three independent \pths{x}{y}\ $P_1,P_2,P_3$ by Menger's theorem \cite[Theorem~3.3.1]{diestelBook05}. By an easy topological argument, there must be a pair of those paths, say $P_1,P_2$, whose union is a cycle $C$ \st\ some side of $C$ contains a tail of $R$ and the other side of $C$ contains $P_3$, see \fig{fIb2}. We may assume \obda\ that $P_3$ is not a single $b$ edge, for we are allowed to choose $x$ and $y$ far apart. Thus the side of $C$ containing $P_3$ contains at least one vertex $z$. Now as all $a$ edges preserve spin, the $a$-coloured double ray $R'$ incident with $z$ is facial, and so it cannot exit the cycle $C$. This means that $C$ is \ps, since one of its sides contains $R'$ and the other contains $R$. 
%As we have already seen that \g cannot have a \psc, we get a contradiction that establishes our assertion.
\showFig{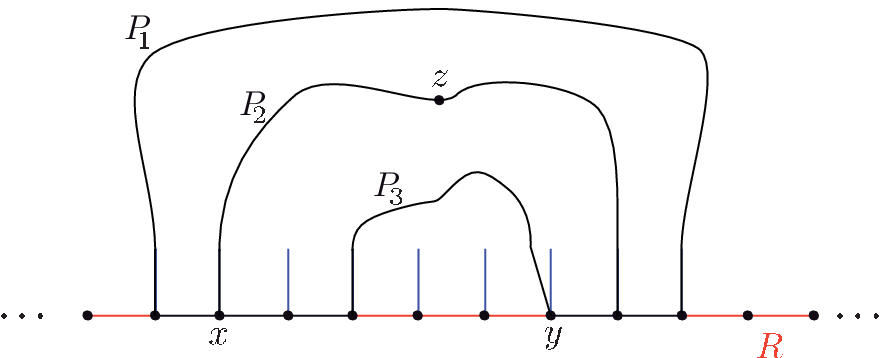}{Finding a \psc\ for the proof of \Lr{apres2}.}

Now imitating the proof of \Lr{apres} we can obtain a contradiction to the fact that \g has a \psc. Thus  $a$ must reverse spin.
\end{proof}

Using \eqref{arevIb} and \Tr{imrcb} it follows easily that any two face boundaries can be mapped to each other by a \auto\ of \G. Thus all faces of \g have the same size $N$. This implies that \g must have a \psc: if $N$ is infinite, then any induced cycle is \ps. If $N$ is finite, then the existence of a \psc\ follows immediately from \Lr{LG3}.

We now turn our attention to the \mpsc s of \G. We will be able to describe these cycles precisely, but in order to do so we have to start with a more modest task, namely to prove that
\labtequ{LIb1}{no \mpsc\ of \g has an $a$-labelled subpath comprising more than two edges.}
To show this, let $p$ be the maximum length of an $a$-labelled subpath of a  \mpsc\ of \G, and let $P$ be a an instance of such a path with $||P||=p$. We need to prove that $p\leq 2$. We distinguish two cases according to the parity of $p$. 

If $p$ is even, and at least 4, then we can shift a \mpsc\ $C$ containing $P$ by two edges of the $a$-coloured double ray containing $P$ to obtain a translate $C'$ of $C$ that crosses $C$ as in \fig{crossIb}. It follows from our discussion in \Sr{secPsc} that one of the regions, $D$, say resulting from this crossing contains only finitely many vertices, while the regions $B$ and $F$ must each contain infinitely many vertices. Moreover, the paths $xC'y$ and $xCy$ in  \fig{crossIb} must have equal lengths, for otherwise the shortest of them provides a shortcut for the $C$ or $C'$, contradicting the minimality of the latter. Thus we can replace $xCy$ by $xC'y$ in $C$ to obtain a new \mpsc\ $C''$ with an $a$-labelled subpath containing $P$ and two more $a$-edges. This contradicts the maximality of $P$. 
\showFig{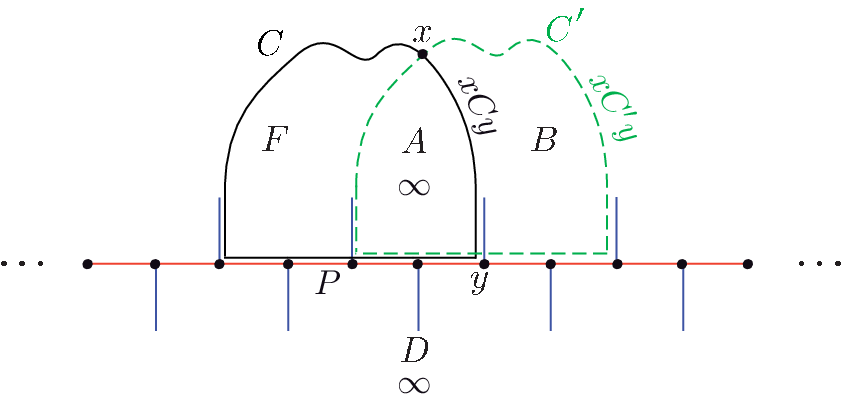}{A crossing in the case that $p$ is even.}

If $p$ is odd, and at least 3, then we can again shift a \mpsc\ $C$ containing $P$ along the $a$-coloured double ray containing $P$, this time shifting only by one edge, to obtain a translate $C'$ of $C$ that crosses $C$ as in \fig{crossIbb}. By the same arguments, one of the regions $A,D$ must be finite, and replacing a subpath of $C'$ for a subpath of $C$ we obtain a new \mpsc\ with a longer $a$ coloured subpath than $P$.
\showFig{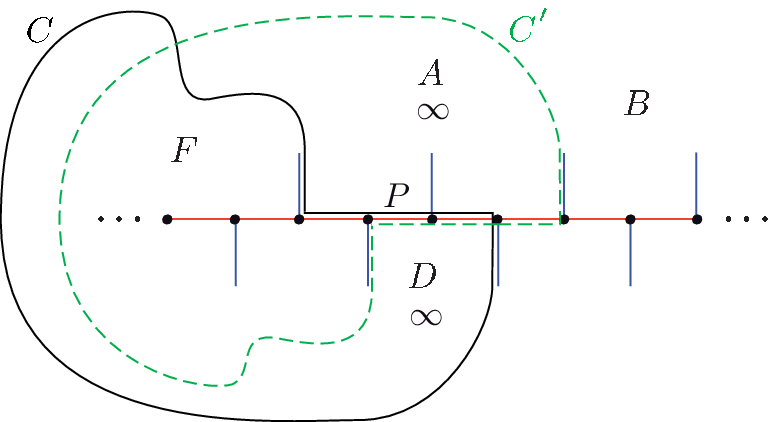}{A crossing in the case where $p$ is odd.}

Thus in both cases we obtained a contradiction to the maximality of $P$, which proves our claim \eqref{LIb1}.

Next, we prove that
\labtequ{LIb2}{no \mpsc\ of \g has a facial subpath comprising more than three edges.}
Indeed, let $F$ be a maximal facial subpath of the \mpsc\ $C$ and suppose that $||F||\geq 4$. Then $F$ does not contain a $bab$ subpath because $a$ reverses spin, so $F$ must contain an $aa$ subpath. It then follows from \eqref{LIb1} that any such subpath of $C$ lies within a $baab$ subpath of $C$. Since any $baab$ path is facial in our case,  $F$ has to contain a $baab$ subpath too. We distinguish two cases, according to the colour of the last edge $e$ of $F$. If that colour is $a$ then we have the situation on the left half of \fig{crossIbc}, while if it is $b$ then we have the situation on the right half of \fig{crossIbc}; here we are using the fact that $F$ cannot finish with an $aa$ subpath followed by an $a$-edge on $C$ because of \eqref{LIb1}. 
\showFig{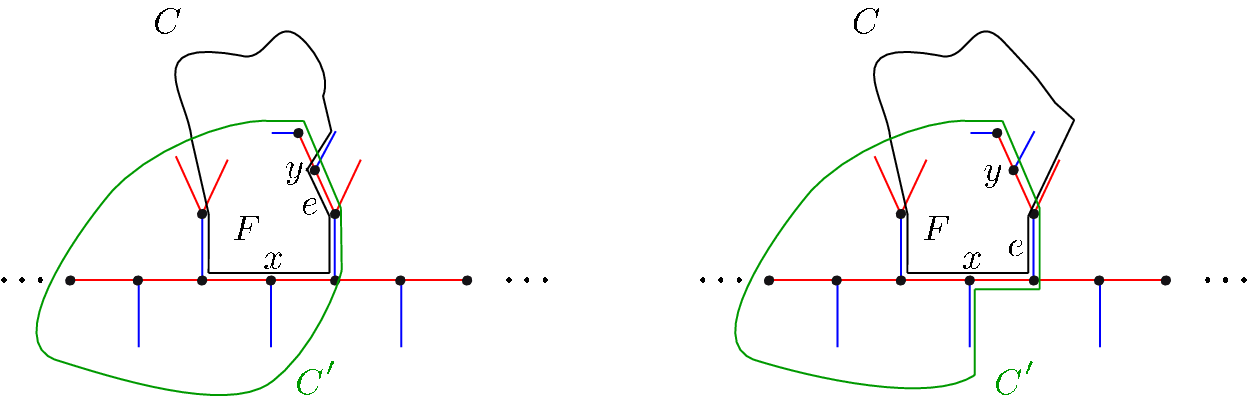}{Proof that a facial subpath of length four forces a crossing.}

In both cases, the \auto\ of \g mapping $x$ to $y$ translates $C$ to some other \mpsc\ $C'$ that intersects $C$. We would like to show that $C'$ crosses $C$. Note that this \auto\ maps $F$ to a facial subpath $F'$ of $C'$ that is incident with the same face \cf\ as $F$ was. It is now easy to see that $F'$ must meet both sides of $C$, for otherwise $C'$ has more edges along the boundary of \cf\ than $C$ has, which would contradict the maximality of $F$. This proves that $C'$ crosses $C$ indeed. As we could have chosen $F$ to have maximum length among all facial subpaths of all \mpsc s of \g \obda, we can apply \Cr{CII3p} to this crossing to obtain a contradiction that proves \eqref{LIb2}.

Combining \eqref{LIb2} with \eqref{LIb1} easily implies that a \mpsc\ cannot even contain an $a$-labelled subpath comprising more than one edge, which means that
\labtequ{LIb3}{The colours of the edges of every \mpsc\ of \g alternate between $a$ and $b$.}
%is induced by the word $(ab)^n$ for some \nin.} 
%

\subsubsection{Tidy cycles} \label{secTidy}
Now consider a $b$ edge $e$ of a \mpsc\ $C$. We will say that $e$ is \defi{tidy} in $C$ if the two edges incident with $e$ that do not lie in $E(C)$ lie in the same side of $C$. We will say that $C$ is \defi{tidy} if all its $b$ edges are tidy in $C$. The reason why we are interested in tidy \mpsc s is the following proposition, which will help us obtain the desired presentation of \G:
\labtequ{LIb5}{A \mpsc\ is tidy \iff\ it is induced by a word of the form $(ab)^n$ and $b$ preserves spin, or it is induced by a word of the form $(aba^{-1}b)^n$ and $b$ reverses spin.}
Indeed, this follows immediately from the fact that $a$ reverses spin and the definition of \defi{tidy}.

In order to be able to exploit this fact we need to show that 
\labtequ{LIb7}{\g has a tidy \mpsc.}
In fact, we will show that every \mpsc\ is tidy, unless every face of \g has size 6. For this, we first have to show the following.

\labtequ{LIb6}{
%\begin{lemma} \label{LIbnocr}
If \g has a face of size greater than 6 then no two \mpsc s of \g cross.}
%\end{lemma}
To prove this, note first that since $a$ reverses spin, any two faces of \g can be mapped to each other, and so every face has size greater than 6 in this case. 

Now suppose that two \mpsc s $C,C'$ cross. Then by \eqref{BF} there is a subpath 
$P$ of $C$ and a subpath $P'$ of  $C'$ \st\ $P \cup P'$ is a cycle $K$ bounding a finite region $B$. We will show that such a region contradicts Euler's formula $n - m + f = 2$ for the sphere. To see this, let $H$ be the finite plane subgraph of $G$ spanned by $K$ and all vertices in $B$. Note that for a cubic finite graph $J$, using the fact that $|E(J)| = \frac{3}{2} |V(J)|$ and that every edge lies in precisely two faces, Euler's formula can be rewritten as 
\labtequ{euler}{(Euler's formula for a cubic graph) \hspace*{1cm} $\sum_{k\ge3} c_k |F_k| = 12$,}
where $|F_k|$ is the number of $k$-gonal faces of $J$, and $ c_k:= 6-k$ is the \defi{curvature} of each $k$-gonal face. This means in particular that a cubic plane graph must have some faces of size $k$ less that 6. 

Our graph $H$ almost contradicts \eqref{euler} since all faces of \g have size greater that 6, except that it has some vertices of degree two on its boundary $K$. To amend these degrees, consider the graph $H'$ obtained from two copies of $H$ by joining corresponding vertices of degree two by an edge, and note that $H'$ is cubic. Consider an embedding of $H$ in the sphere \st\ the two copies of $H$ occupy two disjoint discs $D_1,D_2$, and the newly added edges and their incident faces lie in an annulus $Z$ that joins these discs. Now note that all faces within these discs still have size greater than 6, contributing a negative curvature to \eqref{euler}, but $Z$ can contain 4-gons. Still, we will show that the number of 4-gons is not enough to balance the deficit in curvature.

To begin with, note that every face in $Z$ has even size. Moreover, since $a$ reverses spin and each of $P ,P'$ in the construction of $H$ was $a,b$ alternating by \eqref{LIb3}, it is easy to see that an $a$-edge at the boundary of  $D_1$ or $D_2$ cannot be  incident with a 4-gon in $Z$ unless it was one of the four end-edges of $P$ and $P'$ (\fig{feuler}). 
\showFig{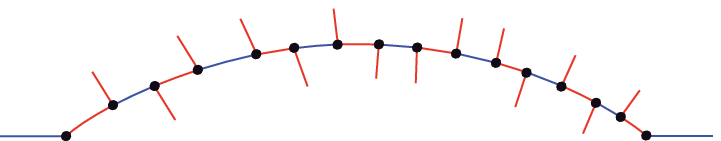}{}
On the other hand, a $b$ edge at the boundary of  $D_1$ or $D_2$ can be incident with a 4-gon in $Z$, however, the fact that $a$ reverses spin easily implies that if $e,f$ are two $b$ edges incident with a 4-gon and both lying on $P$, say,  then there must be a $b$ edge between them on $P$ that is incident with an 8-gon in $Z$, and the same holds for $P'$. These two observations together imply that the number $|F_4|$ of 4-gons in $Z$ is bounded from above by $|F_k|\leq 4 + 2 + |F_8| = 6  + |F_8|$. But as $c_4 + c_8 = 0$, this means that the total curvature contributed to \eqref{euler} by the faces in $Z$ is at most $6 f_4 = 12$, and as $D_1,D_2$ only contain faces of size greater than 6 each of which contributes a negative curvature, we obtain a contradiction to  \eqref{euler}. This proves \eqref{LIb6}. 

\comment{
		Note that the only property of $H$ we used in the above argument was that the edges on the boundary are alternating in two colours, at least one of which colours reverses spin. This allows us to formulate the following corollary, which we will use later in a similar case when we study the \Cg s generated by three involutions:
{perhaps not needed}
	\labtequ{corIb}{If $C,C'$ are two crossing \mpsc s of a plane cubic \Cg\ \g \st\ their edges alternate between two colours one of which reverses spin, then \g has a face of size at most 6.}
}

This argument also explains why we have to treat the case when every face is a hexagon separately. 
\medskip

We now return to the proof of \eqref{LIb7}. In fact, we are going to show something stronger: in the case where \g has  a face of size greater than 6, if a \mpsc\  is untidy then it must cross some other \mpsc, and so by \eqref{LIb6} every \mpsc\ is tidy.

\showFig{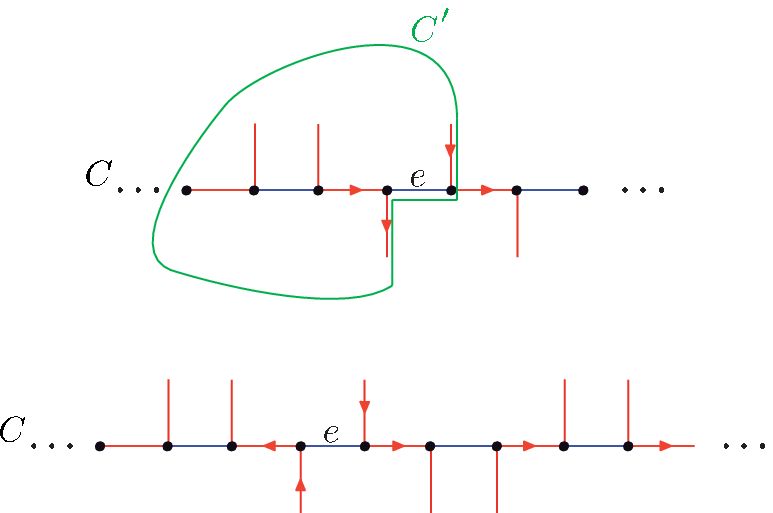}{Finding a crossing at an untidy edge $e$.}

For this, suppose there is a \mpsc\ $C$ with an untidy $b$ edge  $e$. If $b$ reverses spin then exchanging the endvertices of $e$ by a \auto\ translates $C$ to a cycle $C'$ that crosses $C$ and we are done (\fig{untidy}, top). If $b$ reverses spin, then note that the two $a$-edges of $C$ incident with $e$ both point towards $e$, or both point away from $e$ (\fig{untidy}, bottom). Now suppose we walk around $C$ starting at $e$. Every time we come to a tidy $b$ edge, its two incident $a$-edges in $C$ point in the same direction. Thus, as the two  $a$-edges in $C$ incident with $e$ point in opposite directions, before we arrive at $e$ again we  must visit a further untidy edge $e'$ the incident edges of which point away from $e'$ if the incident edges of $e$ point towards $e$ and the other way round. Now translating $e'$ to $e$ maps $C$  to a cycle  $C'$ that again crosses $C$. Thus in both cases applying \eqref{LIb6} we prove that
\labtequ{LIb9}{If \g has  a face of size greater than 6 then every \mpsc\ of \g is tidy.}
In particular, we proved \eqref{LIb7} in the case that \g has a face of size greater than 6.

\subsubsection*{The hexagonal grid case}

It remains to consider the case when every face of \g has size 6, since, easily,  no face can have a size smaller than 6 when $a$ reverses spin and has infinite order. In this case, it is easy to see that any two $a$-coloured double rays that are joined by a $b$ edge are joined by infinitely many $b$ edges that together with these double rays form an infinite strip of hexagons. Using this fact and \eqref{LIb3} it is easy to check that the subgroup of \Gam\ spanned by $a$ has finite index $n$, in other words, $V(G)$ is spanned by finitely many $a$-coloured double rays. 

Let us first consider the case when $b$ reverses spin, and so all the $a$ \dray s point in the same direction, see \fig{hex}. Now let $C$ be any cycle in \g \st\ the colours of the edges of $C$ alternate between $a$ and $b$, the directions of the $a$-edges being arbitrary. For instance, $C$ could be a \mpsc\ by \eqref{LIb3}. We claim that 
\labtequ{LIb8}{if $w$ is a word with letters $a,a^{-1},b$ that induces $C$ then the letter $a$ appears as often as the letter $a^{-1}$ in $w$.}
To see this, note first that translating $C$ by $a^2$ we obtain a cycle $C'$ whose union with $C$ bounds a strip of $|C|/2$ many hexagons (\fig{hex}). We call any of the translates of $C$ by  $a^{2i}, \iin$ a \defi{level}. Enumerate the levels by the integers so that neighbouring levels are assigned consecutive numbers, and these numbers increase whenever we apply $a^2$. 

\showFig{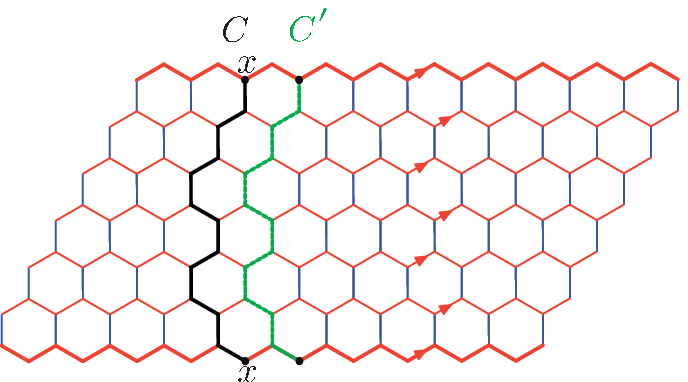}{The case when every face is a hexagon and $b$ reverses spin, and a \mpsc\ $C$. The top and bottom rays coincide, and so do the two vertices marked $x$.}

It is straightforward to check, using the spin behaviour of the edges and the structure of $C$, that for any side $A$ of $C$, all ($a$-coloured) edges that have precisely one endvertex in $C$ are directed the same way, that is, either they are all directed from $C$ into $A$ or the other way round. Moreover, if these edges are directed from $C$ say, the edges incident with $C$ on its other side are directed towards $C$. Now pick a vertex $z$ of $C$ and consider the path $P$ starting at $z$ and induced by $w$. We may assume \obda\ that the successor $y$ of $z$ on $P$ is not in $C$, for otherwise we could have started reading $w$ at $y$. Note that now any appearance of $a$ or $a^{-1}$ in $w$ forces a change of level, while any appearance of $b$ leaves us in the same level. Moreover, it follows from the aforementioned property of the direction of the edges incident with $C$ that any appearance of $a$ increases the level by one, while any appearance of $a^{-1}$ decreases it. But as $w$ induces a cycle, $P$ must return to its initial vertex, and so \eqref{LIb8} follows. 
\medskip

Now let $C$ be any \mpsc\ of \G. Using \eqref{LIb8} we will now modify $C$ into a tidy \mpsc. For this, note that we can replace any path of the form $aba^{-1}$ in $C$ by $a^{-1}ba$ (and the other way round), to obtain another \mpsc; indeed, $aba^{-1}a^{-1}ba$ is a relation, as it induces a boundary of a hexagonal face. Applying this operation several times, and using \eqref{LIb8}, we can reshuffle the letters in a word $w$ inducing $C$ to obtain a new word $w'$ that still induces a \mpsc\ and has the form $w'= (aba^{-1}b)^n$. By \eqref{LIb5} any cycle induced by $w'$ is tidy, and so we achieved our aim to show \eqref{LIb7} in case $b$ reverses spin.

If $b$ preserves spin instead, then the same arguments still apply with the following slight modification. In this case, adjacent $a$ \dray s point in different directions. Partition their set into two equal subsets none of which contains two adjacent \dray s, and call the elements of one the two subsets the \defi{even} \dray s, and call the remaining ones \defi{odd}. Now colour every edge that lies in an odd \dray\ with a new colour $o$ and reverse its direction. Pretending that we cannot distinguish the colour $a$ from $o$, and using the fact that every fourth edge on any \mpsc\ must be coloured $a$ and every fourth edge must be coloured $o$, we can apply the above arguments to obtain a \mpsc\ $C$ of the form $(abo^{-1}b)^n$. As $o^{-1}$ is an alias for $a$, 
\eqref{LIb5} implies again that $C$ is tidy, so we have proved  \eqref{LIb7} in this case too. 

Thus  \eqref{LIb7} holds in all cases.

\subsubsection{Structure and presentations} \label{Ib1pres}
%(in this case \g is a 2-ended hexagonal grid). 

It now follows from \eqref{LIb5} that every $a$-edge of \g\ lies in a unique tidy \mpsc. Similarly, every $b$ edge of \g\ lies in precisely two tidy \mpsc s. In other words, if we represent the involution $b$ by two parallel edges, then the tidy \mpsc s form a decomposition of $E(G)$ into edge-disjoint cycles.

Moreover, 
\labtequ{Ibinout}{the tidy \mpsc s incident with (the $b$ edges of) a given tidy \mpsc\ $C$ lie alternately in its inside and outside.}
This allows us to obtain a presentation, as well as a precise description of the embedding of \g by methods similar to those of \Sr{secIa}, where we had \ps\ $a$-coloured cycles and the $b$ edges incident with any of them lied alternately in its inside and outside.

So similarly to what we did there, we now define the subgroup $\Gam_2$ of \Gam\ to be the subgroup spanned by  $b$, $a^{-1}ba$ and $aba^{-1}$ if $b$ reverses spin, or the subgroup spanned by $b$ and $aba$ (and the inverse $a^{-1}ba^{-1}$ of $aba$) if $b$ preserves spin. We claim that  $\Gam_2$ is a proper subgroup of \Gam, and will prove this by showing that if $C$ is a tidy \mpsc\ containing the identity then only one of the sides of $C$ meets $\Gam_2$.
To see this, note that by \eqref{Ibinout} any path in \g composed as a concatenation of subpaths induced by the words $b, aba$ or $b,aba^{-1}$ and their inverses generating $\Gam_2$ can never cross from the inside of a tidy cycle to its outside. Thus, only vertices that lie in one of the sides of $C$, or any other tidy \mpsc, meet $\Gam_2$.

Given an embedding \sig\ of \G\ we can, still similarly to what we did in \Sr{secIa}, modify \g and \sig\ to obtain a \Cg\ $G_2$ of $\Gam_2$, \wrt\ the above generating set and an embedding $\sig_2$ of $G_2$ as follows. For every tidy \mpsc\ $C$ of \g that contains a vertex in $\Gam_2$, delete all vertices and edges in the side of $C$ that does not meet $\Gam_2$; such a side exists by the above argument. Let $G'_2$ be the graph obtained after doing so for every such cycle. Then, suppress all vertices of $G'_2$ that now have degree two to obtain $G_2$. 

It is easy to prove that
%
%\labtequ{g2con}{
$G_2$ is \iicon.
Indeed, \fe\ vertex $x\in V(G_2)$, the two tidy \mpsc s of \g incident with $x$ form a subgraph of \g that contains the neighbourhood $N(x)$ of $x$ and is connected even after removing $x$, which means that $x$ cannot disconnect $G_2$. With a little bit more effort we can even prove that 
\labtequ{g2tcon}{$G_2$ is \tcon\ unless $V(G_2)$ is contained in a tidy \mpsc\ of \G.}
For this, suppose that $G_2 - \{x,y\}$ is disconnected. If the vertices $x,y$ lie in no common tidy \mpsc\ of \G, then by the above argument their neighbourhoods are connected, a contradiction. So let $C$ be an \mpsc\ of \g containing both vertices. 

It could happen that $V(C)= V(G_2)$, which is the case when \g has hexagonal faces (for example, when it is the graph of \fig{hex}). In this case $G_2$ is not \tcon: it is a finite cycle with parallel edges. Now assume that this is not the case, which means that some $b$ edge $e$ of $C\cap G_2$ is incident with a further tidy \mpsc\ $C_e\neq C$ that contains vertices that lie in $G_2 \sm C$. But then, given two $b$-edges $e,e' \in C \cap G_2$ it is straightforward to check that there is a \auto\ $g$ of \g that rotates $C$, fixing it set-wise, and maps $e$ to $e'$, and so every $b$ edge  in $C \cap G_2$ has this property. 

We claim that if $e,e'$ are consecutive $b$ edges of $C\cap G_2$ then there is a path in $G - C$ connecting their incident cycles $C_e, C_{e'}$. To see this, note that as \g is \tcon, removing $e$ and its endvertices $p,q$ does not disconnect \G, and so there is a \pth{C_e}{C}\ $P$ in $G - \{p,q\}$. Let $f$ be the $b$ edge in $C$ incident with the endpoint of $P$. If $f=e'$ then $P$ is the path we were looking for and we are done. If $f\neq e'$, then consider the \auto\ $g$ of \g that rotates $C$ and maps $e$ to $e'$, and let $P':= gP$ be the image of $P$ under $g$. Since the endvertices of both $e$ and $e'$ lie in $G_2$, it follows by the construction of $G_2$ that $P$ and $P'$ lie in the same side of $C$. Thus, as $g$ rotates $C$, $P'$ must meet $P$ at some vertex $z$, say. Now combining the subpaths of $P,P'$ from their starting point up to $z$ we obtain the desired path joining $C_e$ to $C_{e'}$.

In fact, we can now see that $e$ and $e'$ do not have to be consecutive in the above assertion: For if $e=e_1, e_2, \ldots, e_k=e'$ is a sequence of consecutive $b$ edges of $C \cap G_2$, then combining the \pths{C_{e_i}}{C_{e_{i+1}}}\ we just constructed with subpaths of the cycles $C_{e_i}$ we can construct a \pth{C_{e}}{C_{e'}}\ in $G - C$.

We will now use this kind of path to complete our proof of \eqref{g2tcon}. For this, let $C_x\neq C$ be the other tidy \mpsc\ of \g containing $x$, and define $C_y$ similarly for $y$. Note that %$C_x,C_y$ do not meet distinct sides of $C$ by the previous discussion. Moreover, 
each of $C_x \cap G_2,C_y \cap G_2$ is connected to $C$ in $G_2 - \{x,y\}$. Thus, if $\{x,y\}$ separates $G_2$ then it has to separate $C$ into two subarcs $P_1, P_2$ that lie in distinct components of $G_2 - \{x,y\}$. Pick vertices $v\in P_1 \cap G_2$ and $w\in P_1 \cap G_2$, and let $e_v,e_w$ be the $b$ edges containing $v,w$ respectively. By our last observation there is a  \pth{C_{e_v}}{C_{e_w}}\ $P$ in $G - C$. If we could transform $P$ into a \pth{C_{e_v}}{C_{e_w}}\ $P$ in $G_2 - C$ we would be done, since such a path would contradict the fact that $x,y$ separate $P_1$ from $P_2$. But this is easy to do: for every tidy \mpsc\ $C'\neq C_{e_v},C_{e_w}$ of \g visited by $P$, if $P$ enters the side $A$ of $C'$ that does not meet $G_2$ then it has to exit that side again revisiting $C'$, and we can replace the subpath of $P$ that lies in $A$ by a subarc of $C'$ with the same endpoints. Doing so for every such cycle $C'$, we transform $P$ into a path $P'$ in the auxiliary graph $G'_2$ (see the definition of $G_2$), and it is straightforward to transform $P'$ into a \pth{v}{w}\ in $G_2 - C$. This completes the proof of \eqref{g2tcon}.
\medskip

We now consider separately the cases when $b$ preserves or reverses spin in \G.

If $b$ reverses spin, then recall that $\Gam_2$ was spanned by $b$, $c:=a^{-1}ba$ and $d:=aba^{-1}$, and note that all these generators are involutions. Note moreover that, by construction, $\sig_2$ is a \prem, and that all edges reverse spin in $\sig_2$. Graphs of this kind are characterised in \Sr{secIIa1} below, and it turns out (\Cr{CIIa1}) that if $G_2$ is \tcon\ then it is finite or 1-ended, so we can use our characterisation of those graphs from \Sr{secfin}. We obtain that in this case $G_2 \isom Cay\left<b,c,d\mid b^2, c^2, d^2, (bc)^n, (cd)^m, (db)^p \right>$, $n,m,p \geq 2$, i.e.\ possibility \ref{ziv} of \Tr{LG5}.

If $\Gam_2$ is not \tcon, then by \eqref{g2tcon} we have $c=d$ and $G_2 \isom Cay\left<b,c\mid b^2, c^2, (bc)^n\right>$, where $n\geq 2$ for if $n=1$ then the endvertices of any $b$ edge would separate \G.

\medskip

If $b$ preserves spin, then recall that $\Gam_2$ was spanned by $b$ and $a^*:=aba$. Note that $a^*$ is not necessarily an involution, and that it preserves spin in $\sig_2$ by construction. We distinguish two cases, according to whether the order $n$ of $a^*$ is finite or infinite. 

If $n$ is finite, then $G_2$ is one of the graphs we have already handled: by \Lr{apres}, $G_2$ has at most one end, and so it belongs to type \ref{oi} or \ref{ov} of \Tr{LG2}. Thus, if $a^{*2} = 1$ (in which case \g has hexagonal faces) then $G_2 \isom Cay \left<a^*,b\mid b^2, a^{*2}, (a^* b)^m\right>$, and if $a^{*2}\neq 1$ then  $G_2 \isom Cay \left<a^*,b\mid b^2, a^{*n}, (a^* b)^m\right>$, $n\geq 3$, $m\geq 2$.

If $n$ is infinite, then in particular $a^{*2}\neq 1$ and so $G_2$ is \tcon\ by \eqref{g2tcon}. By \Lr{apres2} $G_2$ has at most one end, so we can apply \Tr{LG2} again, but this time the conclusion is that no such graph $G_2$ exists. Thus $n$ must be finite.

\medskip

In all cases, we have succeeded in finding a \plpr\ of $G_2$. Similarly to what we did in \Sr{secIa}, we will now use this presentations and apply \Tr{ccfdec} to obtain a \plpr\ of \G.

In order to apply \Tr{ccfdec} we will, similarly to the proof of \Tr{TIa2}, let \cx\ be the set of all images of the graph $G'_2$ defined above by \auto s of \G. Again, we define an auxiliary tree $T$ with vertex set $\cx$, this time joining two vertices with an edge whenever they share a tidy cycle. Note that \ref{cciiip} of \Tr{ccfdec}  is satisfied, this time $F_i$ being a tidy cycle. We let again $\cf$ be the set of cycles of \g induced by the relators in the presentation of $\Gam_2$ obtained above after replacing the auxiliary letters $c,d$ and $a^*$ by the corresponding words. All these cycles are contained in $G'_2$. Moreover, \cf\ generates $\cc_f(G'_2)$ by \Lr{relcc}. Thus, defining $\cf_H$ to be the image of \cf\ under the automorphism of \g that maps $G'_2$ to its copy $H\in \cx$, we meet the requirement that $\cf_H$ generate $\cc_f(H)$, and \Tr{ccfdec} yields that $\bigcup_{H\in \cx} \cf_H$  generates $\ccfg$. Using the second sentence of \Lr{relcc} and the definition of \cf\ we thus obtain presentations of $\Gam$ as follows.

If $b$ reverses spin and $G_2$ is \tcon\ then the presentation\\ $G_2 \isom Cay\left<b,c,d\mid b^2, c^2, d^2, (bc)^n, (cd)^m, (db)^p \right>$ obtained above translates into $G \isom Cay\left<a,b\mid b^2, (baba^{-1})^n, (a^2 ba^{-2} b)^m, (baba^{-1})^p  \right>$. Note however that the second and fourth relations both induce the tidy \mpsc s of \G, and so $n=p$ and the two relations coincide, so one of them can be dropped. If $b$ reverses spin and $G_2$ is not \tcon\ then similarly we obtain $G \isom Cay\left<a,b\mid b^2, a^2ba^{-2}b,  (baba^{-1})^n \right>$, where the relation $a^2ba^{-2}b$ is tantamount to $c=d$, and the relation $c^2$ was dropped as it translates into the trivial $a^{-1} b a a^{-1} b a=1$.

If $b$ preserves spin, then the presentations $G_2 \isom Cay \left<a^*,b\mid b^2, a^{*2}, (a^* b)^m\right>$ and $G_2 \isom Cay \left<a^*,b\mid b^2, a^{*n}, (a^* b)^m\right>$ obtained above easily translate into $G \isom Cay \left<a,b\mid b^2, (a^2b)^2, (aba b)^m \right>$ and $G \isom Cay \left<a,b\mid b^2, (a^2b)^{n}, (aba b)^m \right>$ respectively (note that  $(aba)^n= (a^2b)^{n}$). We thus have

\begin{theorem} \label{TIbx}
Let $G= Cay\left<a,b\mid b^2, \ldots \right>$ be a \tcon\ planar \Cg\ with more than one end % and a \prem, 
and suppose that $a$ has infinite order. Then $a$ reverses spin. If $b$ reverses spin then either 
\begin{enumerate}\addtolength{\itemsep}{-0.5\baselineskip}
\item \label{bi} $G \isom Cay\left<a,b\mid b^2, a^2ba^{-2}b;  (baba^{-1})^n \right>$, $n\geq 2$, which is the case when \g has hexagonal faces (and $G_2$ is not \tcon), or
\item \label{bii} $G \isom Cay\left<a,b\mid b^2, (a^2 ba^{-2} b)^m; (baba^{-1})^n \right>$, $n,m,p \geq 2$.

\medskip
\hspace*{-\leftmargin}
 If $b$ preserves spin then either
\item \label{biii} $G \isom Cay \left<a,b\mid b^2, (a^2b)^2; (ab)^{2m} \right>$, $m\geq 2$, which is the case when \g has hexagonal faces, or
\item \label{biv} $G \isom Cay \left<a,b\mid b^2, (a^2b)^{n}; (ab)^{2m} \right>$, $n\geq 3$, $m\geq 2$.

\end{enumerate}
In all cases, the presentation is planar. %face boundaries of \g are the (finite) cycles induced by the last relator. 

Conversely, each of the above presentations, with parameters chosen in the specified domains, yields a planar, \tcon\ \Cg\ with more than one end.
\end{theorem}
\begin{proof}
The forward implication was proved in the above discussion.

For the converse implication we follow the approach the proof of \Tr{TIa2}. Given a presentation \cp\ as in \ref{bi}--\ref{biv}, we construct a \Cg\ \g as follows.

If \cp\ is of type \ref{bii}, then we begin by constructing the auxiliary \Cg\ $G_2 \isom Cay\left<b,c,d\mid b^2, c^2, d^2, (bc)^n, (cd)^m, (db)^n \right> $, where the parameters $n,m$ here coincide with those in \cp. By \Tr{LG5} \ref{ziv}, $G_2$ has an embedding \sig\ in which, in particular, the 2-coloured cycles induced by $(bc)^n$ bound faces. Let \g be the graph obtained from the \twsqam\ (as defined in the Introduction, recall \fig{fidih}) of $G_2$ \wrt\ those cycles. 

Similarly, if \cp\ is of type \ref{biv}, we let $G_2 \isom Cay\left< a^*,b\mid b^2, a^{*n}, (a^* b)^m\right>$, and apply \Tr{LG2} \ref{oi} to obtain an embedding \sig\ of $G_2$ in which the 2-coloured cycles induced by $(a^* b)^n$ bound faces. Let \g be the graph obtained from the \twsqam\ of $G_2$ \wrt\ those cycles. It follows easily from Sabidussi's \Tr{sab} that \g is a \Cg; see \cite{am} for details.

By \Lr{lkcon}  \g is \tcon\ in both above cases; this can be shown by arguments similar to those of the proof of \Tr{TIa2}.

Next we claim that $a$ has infinite order in \G. To see this note that, by construction, any $a$-coloured component of \g meets infinitely many of the cycles along which the \twsqam\ took place. 
%so for any copy $H$ of $G_2$
%Suppose to the contrary that a pair $\{x,y\}$ of vertices separates \G. If $\{x,y\}$ do not lie on a common basic\sss cycle, then all neighbours of $x$ lie in the same component of $G - y$, namely the component containing the two basic cycles containing $x$. This means however that $x$ is redundant in the separator  $\{x,y\}$, and so \g is not even \iicon. But this easily contradicts \Lr{Liicon}. So let us suppose that $\{x,y\}$ lie on a common basic\sss cycle $C$.

If \cp\ is of type  \ref{bi} or  \ref{biii} instead, then our task is easier. Although we could again follow the same approach, starting with a finite graph $G_2$, it is simpler to construct \g directly as in \fig{hex}: choose the number of parallel monochromatic double rays to be twice the parameter $n$ or $m$ in \cp, and direct all edges of every second monochromatic double ray the other way if \cp\ is of type  \ref{biii}.

This completes the construction of \g in all cases. The fact that \g has indeed the desired presentation \cp\ now follows from the forward implication which we have already proved.
\end{proof}

It follows easily from the proof of \Tr{TIbx} that \g has precisely two ends if it is of type \ref{bi} or  \ref{biii}, and it has infinitely many ends if it is of one of the other two types.

We have now completed our analysis of the case where \g has 2 generators $a,b$. Let us remark the following, which is perhaps interesting in view of our discussion in \Sr{intBW}. 

\begin{corollary} 
Let $G= Cay\left<a,b\mid b^2, \ldots \right>$ be a \tcon\ planar \Cg. Then every face of \g has a finite boundary.
\end{corollary}
\begin{proof}
If \g has only one end then this follows from our results of \Sr{secfin} (in fact, this is known and holds no matter what the vertex degree is, see \cite{KroInf}). 

If \g has more than one end, then recall that all our presentations contained a relator inducing a face-boundary. As $a$ always reverses spin in this case (see \eqref{arevI} and \Lr{apres2}), and our embeddings are consistent, it follows easily that any two face-boundaries of \g can be mapped to each other by a \auto. Thus all face-boundaries are induced by that relator, and so they are finite.
\end{proof}

\section{The planar multi-ended \Cg s generated by 3 involutions} \label{secII}

Having already described all planar cubic \Cg s on two generators, we proceed to the planar cubic \Cg s on three generators. Recall that it only remains to describe those that are \tcon\ and multi-ended. We divide them into two subclasses, those that have 2-coloured cycles and those that do not, since different arguments are needed in these two cases.

\subsection{Graphs with 2-coloured cycles} \label{secIIa}
 
We have to distinguish three further subcases, according to how the  2-coloured cycles behave \wrt\ spin.

\subsubsection{2-coloured cycles in two spin-reversing colours} \label{secIIa1}

We start this section by pointing out that certain choices of spin behaviour cannot give rise to \Cg s of the type we are studying. We will make use of these results in subsequent subsections, where we characterise the choices that do give rise to \Cg s.

\begin{lemma}\label{LIIaPsc}	
Let $G\isom Cay\left<b,c,d\mid b^2, c^2, d^2, \ldots \right>$ be planar, multi-ended and \tcon. If $(bc)^n=1$ for some $\nin$, and both $b,c$ reverse spin, then \g has a \psc. 
\end{lemma}
% *** ---- *** 
\begin{proof}
Note that under these assumptions, $b,c$ span finite cycles that bound faces of \G, see \fig{fIIa1}. If the remaining faces are also finite then we are done by \Lr{LG3}.

If there is an infinite face $F$, then it must be incident with one of the $bc$ cycles $C$, so let $e=zy$ be an edge of $C$ incident with $F$, and assume \obda\ that $e$ is coloured $c$, see \fig{fIIa1}. Let $v=zb$ be  the vertex of $C - y$  adjacent with $z$. Since \g is \tcon, there is a path $P$ in $G - \{z,v\}$ from $y$ to the neighbour $x$ of $z$ outside $C$. Then $yPxzy$ is a cycle $D$. Note that there is a side $A$ of $D$ containing $F$, and so $A$ is infinite. 
\showFig{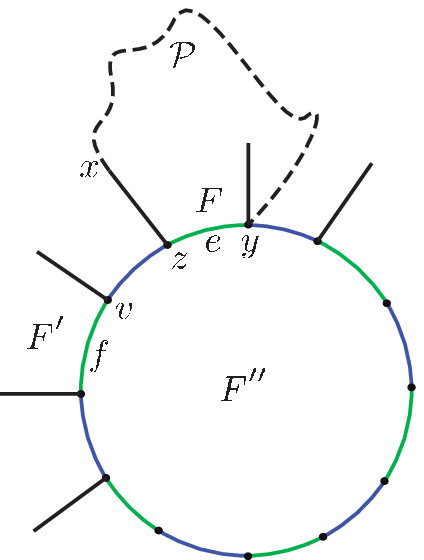}{Finding a \psc\ in the proof of \Lr{LIIaPsc}.}

We would like to show that $G - A$ is also infinite, which would mean that $D$ is \ps. To show this, consider the \auto\ $g$ of \g mapping $z$ to $v$. Then $g$ maps $e$ to some other $c$ edge $f$ of $C$, and it maps $F$ to some infinite face $F'$ incident with $f$. Note that $D$ cannot contain $f$ by its construction. Thus $F'$ and the finite face $F''$ bounded by $C$ lie in the same side of $D$. But that side cannot coincide with $A$, because as $e\in E(C)$, $C$ separates $F''$ from $F\subseteq A$. Thus both sides of $D$ are infinite, since they contain the distinct infinite faces $F,F'$.

\end{proof}

%\subsubsection{?} \label{secIIa1}

%???In this case the situation is very similar with the case when \g is spanned by two generators one of which $a$ preserves spin and spans finite cycles. And indeed, similarly to \Lr{apres} we will see that this kind of graph cannot have more than one end:

\begin{lemma} \label{LIIab1}
Let $G\isom Cay\left<b,c,d\mid b^2, c^2, d^2, \ldots \right>$ be planar, multi-ended and \tcon.  
%WRONG: {\tcon ness needed only in the no-2-coloured-cycle case}
Then at least one of the colours, $d$ say, preserves spin. Moreover, if both $b,c$ reverse spin then either $(bd)^n=1$ or $(cd)^n=1$ holds for some $n \geq 2$.
\end{lemma}
\begin{proof}
If at least two of the colours preserve spin then there is nothing to show, so suppose \g is such a graph in which two colours reverse spin. We distinguish two cases. If two of the colours, $b$ and $c$ say, span finite cycles, that is, if  $(bc)^n=1$ holds, then we can assume that  both $b,c$  reverse spin for otherwise our claim is already proved. We can then apply \Lr{LIIaPsc}, which yields that \g has a \psc. If no two of the colours span finite cycles, then we can imitate \Lr{apres2} (see \fig{fIb2}) to prove that \g has a \psc. Thus \g has a \psc\ in both cases. Moreover, we can assume in both cases that  $b$ and $c$ reverse spin.

Choose a \mpsc\ $C$ with a facial subpath $F$ of maximum length among all \mpsc s of \G. We distinguish three cases according to how $F$ ends, which are very similar to the cases in \Lr{apres}: either $F$ ends with an edge coloured $b$ or $c$, or both end-edges of $F$ are coloured $d$ and $d$ reverses spin, or both end-edges of $F$ are coloured $d$ and $d$ preserves spin. %(if \g has no 2-coloured cycles we do not have to make this distinction).
In the first two cases it is easy to obtain a \mpsc\ $C'$ that crosses $C$ similarly to the first two cases of  \Lr{apres} (see \fig{fIa12}): consider the \auto\ of \g exchanging the endvertices of the last edge $e$ of $F$. But this crossing contradicts \Cr{CII3p}. Thus the third case must occur; in particular, we have shown the first part of our claim, asserting that $d$ must preserve spin.

Now suppose that the second part of our claim is false, that is, $d$-edges are in no 2-coloured cycles. Then $C$ must be 3-coloured. Now let $e$ be the last edge of $F$, and recall that $e$ is coloured $d$. Assume \obda\ that the colour of the edge $f$ of $F$ preceding $e$ is $c$. Note that the edge $h$ following $e$ on $C$ must be coloured $c$ too, since $e$ is the last edge of a facial path and $e$ preserves spin (\fig{fIIab1}). We can use this fact to show that $C$ has no subpath $P$ of the form $db$: for then we could map the $d$-edge of $P$ to $e$ by a \auto\ of \g to obtain a translate $C'$ of $C$ that crosses $C$. Indeed, the image of the $b$ edge $g$ of $P$ would then lie in the side of $C$ containing the face \cf\ incident with $F$. Moreover, $C'$ would have to leave $F$ with an edge that lies in the other side of $C$, because $C'$ is not allowed to have a path longer than $F$ incident with \cf\ and $g\in E(C')$ is incident with $F$ (\fig{fIIab1}). But this kind of crossing contradicts  \Cr{CII3p}, which proves our claim that $C$ has no subpath of the form $db$.

\showFig{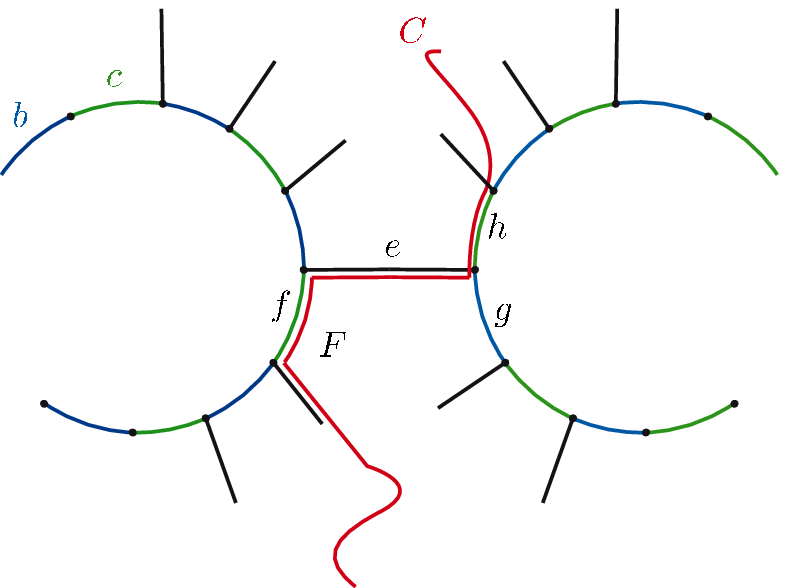}{The facial path $F$ in the proof of \Lr{LIIab1}.}

As $C$ must be 3-coloured, this implies that $C$ has a subpath of the form $dc(bc)^kd$ for some $k\geq 1$. Now exchanging the endvertices of the first $c$ edge of such a subpath by a \auto\ of \g yields again a crossing. We can now use an argument similar to the the proof of \eqref{LIb1} for $p$ even (\fig{crossIb}) to obtain a contradiction: choosing $C$ so as to maximise the length of a subpath $P\subseteq C$ of the form $dc(bc)^kd$, and considering a crossing as above, \eqref{CII3m} implies the existence of a \mpsc\ with a longer path of this kind. This completes the proof of the second part of our claim.
\end{proof}

With \Lr{LIIab1} we immediately obtain

\begin{corollary} \label{CIIa1}
There is no planar \tcon\ multi-ended \Cg\ of the form $Cay\left<b,c,d\mid b^2, c^2, d^2, \ldots \right>$ %\st\ $(bc)^n=1$ holds for some $\nin$ and all
in which all edges reverse spin.
\end{corollary}

%Note that in the above proof we did not really make use of the fact that \g has a two coloured cycle, and so we can use the same argument also in that case. The existence of a \psc, which we do need, does not follow from \Lr{LIIaPsc} in this case, but can be proved imitating \Lr{apres2}. Thus we have

\comment{
	The graphs in the following assertion do not belong to this section, but we state it here because it follows from the same arguments

	\begin{corollary}\label{CIIb1}
	There is no planar \tcon\ multi-ended \Cg\ of the form $Cay\left<b,c,d\mid b^2, c^2, d^2, \ldots \right>$ that has no 2-coloured cycle and at least two colours reversing spin.
\end{corollary}
\begin{proof}
	Repeat the proof of \Lr{LIIab1} letting, \obda, the colours $b,c$ be spin-reversing.
	\end{proof}
}

\subsubsection{2-coloured cycles in two spin-preserving colours} \label{secIIa2}

In this section we study the planar \tcon\ \Cg s of the form $G =Cay\left<b,c,d\mid b^2, c^2, d^2, \ldots \right>$ that have a cycle induced by a relation of the form $(bc)^n$ and both $b$ and $c$ preserve spin.

It is not hard to check that a graph of this kind must be infinite, because every such cycle has a translate of itself in each of its sides, and, by the same argument, every such cycle is \ps\ (\fig{fIIa2}); in particular, \g is multi-ended.
\showFig{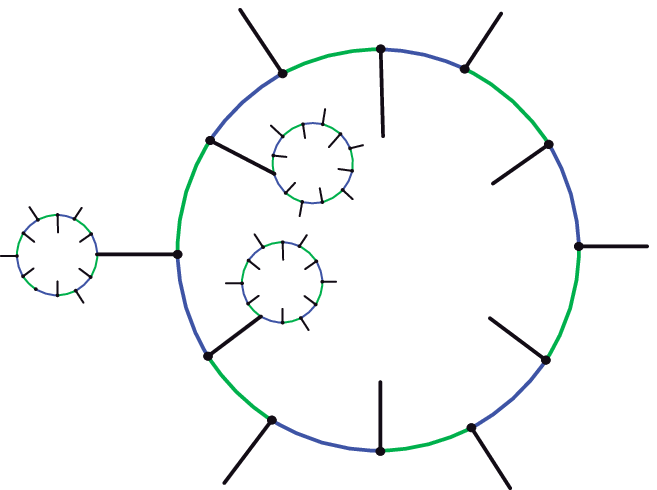}{A 2-coloured cycle with spin-preserving edges, and some of its translates.}

We will again consider a subgroup $\Gam_2$ of $\Gam(G)$ as we did in Sections \ref{secIa} and \ref{Ib1pres}: this time we let  $\Gam_2$ be the subgroup of $\Gam(G)$ generated by $bc, (bc)^{-1}, d$. Let $a:= bc$, and note that $a^n=1$. Again, we will define the auxiliary subgraph  $G'_2\subseteq G$ and use it to obtain a \Cg\ $G_2$ of $\Gam_2$ with an embedding induced by that of \G\ using a construction and arguments very similar to those of \Sr{secIa}. 

To begin with, note that if $x,y$ are two elements of $\Gam_2$, then there is an \pth{x}{y}\ $P$ in \g the $b$ and $c$ edges of which can be decomposed into incident pairs. Thus, since $b$ and $c$ preserve spin, whenever such a path $P$ meets a $bc$ cycle $C$ of \G, the two edges of $P$ incident with $C$ lie in the same side of $C$ (\fig{fIIa2}). In other words, $P$ cannot cross any $bc$ cycle $C$ of \G. Now given the embedding \sig\ of \G, we can modify \g and \sig\ to obtain a \Cg\ $G_2$ of $\Gam_2$, \wrt\ to the generating set $\{a, d\}$, and an embedding $\sig_2$ of $G_2$ as follows. For every $bc$ cycle $C$ of \g that contains a vertex in $\Gam_2$, delete all vertices and edges in the side of $C$ that does not meet $\Gam_2$. Let $G'_2$ be the graph obtained after doing so for every such cycle. Then, suppress all vertices of $G'_2$ that now have degree two; that is, replace any $bc$ path $xPy$ of length two whose middle vertex now has no incident $d$-edge by a single $x$-$y$~edge, directed the same way as $P$ and bearing the colour $a$, to obtain the \Cg\ $G_2$ of $\Gam_2$.  

Using \Lr{Liicon} it \ises\ that $G_2$ is \iicon\ since \g was.

We are now in the fortunate situation of having obtained a \Cg\ of a type that we have already handled: $G_2$ is generated by two elements, and has monochromatic cycles induced by the relation $a^n$. Moreover, by the construction of the embedding $\sig_2$, the $a$-edges preserve spin. Thus we can apply \Lr{apres}, which yields that $G_2$ has at most one end. So $G_2$ is one of the graphs in \Tr{LG2}, and as $a$ preserves spin cases \ref{oiii} and \ref{oiv} can be eliminated. The degenerate case \ref{ov} of \Tr{LG2} cannot occur, because it would imply that \g is not \tcon\ as opposite vertices of a $bc$ cycle would separate in that case.

Having obtained a \plpr\ of $\Gam_2$, we can now use the same method as in \ref{secIa}, namely to apply \Tr{ccfdec}, to yield a \plpr\ of $\Gam$. 

\begin{theorem} \label{TIIa2}
Let $G= Cay\left<b,c,d\mid b^2, c^2, d^2, \ldots \right>$ be a planar \tcon\ \Cg\ with more than one end, and suppose that $bc$ has a finite order $n$ and both $b,c$ preserve spin. 
If $d$ preserves spin then

$G \isom Cay \left<b,c,d\mid b^2,c^2,d^2, (bcd)^m; (bc)^n \right>$, $n\geq 3$, $m\geq 2$.\\ 
If $d$ reverses spin then

$G \isom Cay \left<b,c,d\mid b^2,c^2,d^2, (bcdcbd)^m; (bc)^n\right>$, $n\geq 3, m\geq 1$.\\
In both cases, the presentation is planar. %face boundaries of \g are the (finite) cycles induced by the last relator. 
%\g is \tcon\ \iff\ $n>2$; otherwise it is \iicon.\\

Moreover, \g is the Mohar amalgamation of $G_2 \isom Cay \left<a,b\mid b^2, a^n, (ab)^m\right>$ or
 $G_2 \isom Cay \left<a,b\mid b^2, a^n, (aba^{-1}b)^m\right>$ with itself along the $a$ coloured cycles.

Conversely, each of these presentations, with parameters chosen in the specified domains, yields a \Cg\ as above.
\end{theorem}
\begin{proof}
For the forward implication we apply \Tr{ccfdec} as in the previous sections. This time the common cycles giving rise to the edges of the auxiliary tree $T$ on the copies of $G'_2$ are the 2-coloured cycles induced by $(bc)^n$. Recall that we obtained a presentation of \gt\ in the above discussion from \Tr{LG2}: we have $G_2 \isom Cay \left<a,d\mid d^2, a^n, (ad)^m\right>$, $n\geq 3$, $m\geq 2$ if $d$ preserves spin and
$G_2 \isom Cay \left<a,d\mid d^2, a^n, (ada^{-1}d)^m\right>$, $n\geq 3, m\geq 1$ 
if $d$ reverses spin, with $a=bc$ in both cases. Applying \Tr{ccfdec}, and replacing $a$ back yields the desired \plpr s.

To prove the converse implication, given one of these presentations of the first type we construct the  Mohar amalgamation $G$ of $G_2 := Cay \left<a,b\mid b^2, a^n, (ab)^m\right>$ with itself along the $a$ coloured cycles ---see Introduction. By Sabidussi's theorem \g is a \Cg; see \cite{am} for details. \Lr{lkcon} yields that \g is \tcon\ since, by \Tr{endtcon}, $G_2$ is \tcon. We can thus apply the forward implication to prove that the presentation we started with is indeed a presentation of \G. If we are given a  presentations of the second type instead, then we proceed similarly, except that we now let 
 $G_2 := Cay \left<a,b\mid b^2, a^n, (aba^{-1}b)^m\right>$. 

%Note that if $n=2$, in which case $\kappa(G)=2$, then the two presentations of \Tr{TIIa2} become equivalent, and indeed the resulting graph has both a \prem\ in which $b$ preserves spin and a \prem\ in which $b$ reverses spin. An embedding of the first type can be modified into an embedding of the second type (and vice versa) by recursively ``fliping'' one of the components of separating pairs of $d$-edges (incident with opposite vertices of a $bc$ cycle. See for example the two embeddings of the graph of Figure 5 (target)\sss in \cite{cay2con}.
\end{proof}

%Every face is finite \sss.

\subsubsection{2-coloured cycles with mixed spin behaviour} \label{secIId1}

In this section we study the planar, infinite, \tcon\ \Cg s of the form $G =Cay\left<b,c,d\mid b^2, c^2, d^2, \ldots \right>$ that have a cycle of the form $(bc)^r$ and precisely one of $b,c$ preserves spin. Let us assume that $b$ preserves spin while $c$ reverses spin (\fig{fIId1}). %Recall that the existence of a \prem\ for graphs of this type with connectivity 2 was proved in \cite{cay2con}. We have to consider also the graphs of connectivity 2 in this case because they are not handled in \cite{cay2con}.
\showFig{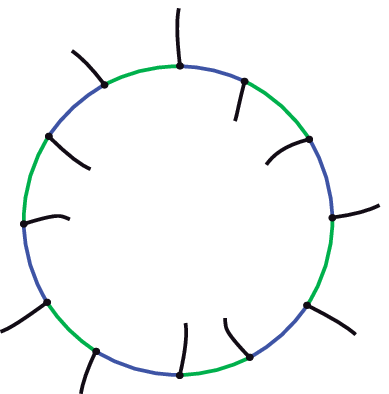}{A 2-coloured cycle with mixed spin behaviour.}

%It is easy to see 

We proceed similarly to \Sr{secIIa2}: let  $\Gam_2$ be the subgroup of $\Gam(G)$ generated by $bcb, c, d$. Let $b^*:= bcb$, and note that $b^{*2}=1$ and $(b^*c)^{n}=1$ where $n = r/2$ ($r$ must be even). Define the auxiliary subgraph  $G'_2\subseteq G$ and use it to obtain a \Cg\ $G_2$ of $\Gam_2$ with an embedding $\sig_2$ induced by \sig\ similarly to what we did in \Sr{secIIa2} and \Sr{secIa}: for every $bc$ cycle meeting $\Gam_2$ delete all vertices in its side not meeting $\Gam_2$, then suppress vertices of degree 2. Note that in this case we replace paths of length 3 by edges when suppressing, while in \Sr{secIIa2} the corresponding paths had length 2.

By the construction of $\sig_2$ the new edges, coloured $b^*$, reverse spin while the $c$ and $d$-edges retain their spin behaviour of \sig. This means that the new $b^*c$ cycles have all their edges reversing spin. Thus $G_2$ is the kind of graph we studied in \Sr{secIIa1} or \Sr{secfin} if it is \tcon. Let us check that this is indeed the case. % when \g is \tcon:

%, or $\kappa(G)=2$ and we obtain a \plpr\ for $G_2$ from \cite{cay2con}: either $G_2 \isom Cay \left<{b^*},c,d\mid {b^*}^2, c^2,d^2, ({b^*}c)^n, ({b^*}d)^m\right>$, $n,m\geq 2$ or $G_2 \isom Cay \left<{b^*},c,d\mid {b^*}^2, c^2,d^2, ({b^*}c)^n, (cd)^m\right>$, $n,m\geq 2$. 

\begin{proposition} \label{kG2}
%If \g is \tcon\ then so is $G_2$.
%$\kappa(G_2) = \kappa(G)$.
$G_2$ is \tcon.
\end{proposition}
\begin{proof}

Using \Lr{Liicon} and the fact that $(b^*c)^{n} $ is a relation in $\Gam_2$ easily implies that $G_2$ is \iicon\ since \g is.
%Let us first prove that if \g is \tcon\ then so is $G_2$. Indeed, 
Recall that $G_2$ has 2-coloured cycles of the form $({b^*}c)^n$, both ${b^*},c$ reverse spin in $\sig_2$, so that every such cycle is a face boundary in $\sig_2$. But by the results of \cite{cay2con}, if $\kap(G_2)=2$ then  $G_2$ belongs to one of the types \ref{iv}, \ref{v}, \ref{vi} or \ref{ix} of \Tr{main2}, and of those types, only \ref{vi} and \ref{ix} can have a finite 2-coloured face boundary; see \citeObsfinfac. 

If $G_2$ belongs to type \ref{ix}, which means that it is a finite cycle with some additional parallel edges, then \ta\ three cases to be considered. If $n=1$, which means that ${b^*}=c$ and ${b^*},c$ span 2-cycles, then any two vertices separating $G_2$ also separate \G, contrary to our assumption that the latter is \tcon. If ${b^*}=d$ or  $c=d$ instead, then it is straightforward to check that $G$ must be finite, which we are also assuming is not the case
 
Thus $G_2$ belongs to type \ref{vi}, and so either $G_2 \isom Cay \left<{b^*},c,d\mid {b^*}^2, c^2,d^2, ({b^*}c)^n, ({b^*}d)^m\right>$, $n,m\geq 2$ or $G_2 \isom Cay \left<{b^*},c,d\mid {b^*}^2, c^2,d^2, ({b^*}c)^n, (cd)^m\right>$, $n,m\geq 2$. In the former case the $b^*$ edges are \sepe s and in the latter the $c$ edges are \sepe s. We claim that the endvertices of such a \sepe\ also separate \G. 

For this, let $\{x,y\}$ be the endpoints of a \sepe\ of $G_2$, let $C$ be the ${b^*}c$ cycle of $G_2$ containing $x,y$, and let $K$ be the component of $G_2 - \{x,y\}$ that does not meet $C$; such a component exists because $\{x,y\}$ cannot separate $C$ as $xy$ is an edge of $C$. If $G - \{x,y\}$ has a \pth{K}{C} $P$, then $P$ can be modified into a \pth{K}{C} in $G_2 - \{x,y\}$ as follows: for every $bc$ cycle $D$ met by $P$ that is disjoint from $C$, note that both endpoints of $P$ lie in the same side of $D$, for $G_2$ cannot meet both sides of any such cycle by the construction of $G_2$. This means that if $P$ enters the side $A$ of $D$ not meeting $G_2$, then it must exit that side again. Thus, we can replace a subpath of $P$ that has endvertices $v,w$ on $D$ and whose interior lies in $A$ by a $v$-$w$~subarc of $D$ to obtain a path that does not meet the `wrong' side $A$ of $D$. Similarly, if $P$ meets the side of $C$ not containing $K$, then we replace the part of $P$ in that side by a subarc of $C$, this time being careful enough to pick that subarc that does not contain $x$ and $y$. Performing such a modification recursively as long as $P$ meets both sides of a $bc$ cycle $D$, we modify $P$ into a path $P'$ with the same endvertices that meets at most one of the sides of any $bc$ cycle. It follows easily that $P'$ is a path in $G'_2$. Moreover, $P'$ does not meet $x,y$ by construction. But then $P'$ contradicts the fact that $\{x,y\}$ separates $C$ from $K$ in $G_2$. This completes the proof of our claim that \gt\ is \tcon.

%For converse see \Tr{TIId1con}
\end{proof}

Now $G_2$ might be finite or 1-ended, in which case we can use our classification of \Sr{secfin}, or multi-ended, in which case we can apply  \Lr{LIIab1}, which yields that $d$ must preserve spin. Thus, in the case where $d$ reverses spin, we obtain the following classification.

\begin{theorem} \label{TIId11}
Let $G= Cay\left<b,c,d\mid b^2, c^2, d^2, \ldots \right>$ be an infinite planar \tcon\ \Cg\ 
\note{redundant: with more than one end} % and a \prem\ \sig, 
such that $bc$ has a finite order $n$ and precisely one of $b,c$ preserves spin ($b$ say) and $d$ reverses spin. Then $G \isom Cay\left<b,c,d\mid b^2, c^2, d^2, (cd)^m, (dbcb)^p; (bc)^{2n} \right>$, $n,m,p \geq 2$. %G2 <=1-ended
\comment{Then precisely one of the following is the case:
		\begin{enumerate}\addtolength{\itemsep}{-0.5\baselineskip}
	\item \label{di} $G \isom Cay \left<b,c,d\mid b^2, c^2,d^2, (bc)^{2n}, (bcbd)^m\right>$, $n\geq 1,m\geq 2$  %G2 2-conn
	\item \label{dii} $G \isom Cay \left<b,c,d\mid b^2, c^2,d^2, (bc)^{2n}, (cd)^m\right>$, $n\geq 1,m\geq 2$ (dummy case: can be drawn with all edges reversing spin as $c$ edges are separating (\vapf\ graph)). %G2 2-conn
	\item \label{diii} $G \isom Cay\left<b,c,d\mid b^2, c^2, d^2, (bc)^{2n}, (cd)^m, (dbcb)^p \right>$, $n\geq 1,m,p \geq 2$ %G2 <=1-ended
	\end{enumerate}
}
This presentation is planar. %face boundaries of \g are the (finite) cycles induced by the last relator. 
%\g is \tcon\ \iff\ it is of type \ref{diii}; otherwise it is \iicon.\\

Moreover, \g is a \twsqam\ of $Cay\left<b^*,c,d\mid b^{*2}, c^2, d^2, (b^*c)^n, (cd)^m, (db^*)^p \right>$ with itself.

Conversely, each of these presentations, with parameters chosen in the specified domains, yields a planar \tcon\ multi-ended \Cg.
\end{theorem}
\begin{proof}
By the above discussion, $G_2$ cannot be multi-ended since we are assuming that  $d$ reverses spin. Since all edges of  $G_2$ reverse spin, we obtain a presentation of $G_2$ from \Tr{LG5}~\ref{ziv} or \ref{zv}, and by \Prr{kG2} we can exclude \ref{zv}. Thus we are left with
$G_2 \isom Cay\left<b^*,c,d\mid b^{*2}, c^2, d^2, (b^*c)^n, (cd)^m, (db^*)^p \right>$, $n,m,p \geq 2.$
%(we used the fact that \ref{ziv} reduces to \ref{zv} if any of $n,m,p$ equals 1). 

For the forward implication we apply \Tr{ccfdec} as in the last section. This yields that we can obtain a presentation of \g by replacing $b^*$ with $bcb$ in the above presentation of \gt, except that we replace the relation $b^{*2}$ by $b^2$. It \ises\ that the asserted presentation is planar.

For the converse implication we proceed as in the proof of \Tr{TIIa2}: let  $G_2 := Cay\left<b^*,c,d\mid b^{*2}, c^2, d^2, (b^*c)^n, (cd)^m, (db^*)^p \right>$ and let \g be the \twsqam\ of \gt\ with itself along the $b^*c$ cycles. Again \g is a \Cg\ by Sabidussi's theorem; see \cite{am} for details. \Lr{lkcon} yields that \g is \tcon\ since, by \Tr{endtcon}, $G_2$ is \tcon. We can thus apply the forward implication to prove that \G\ has the desired presentation.
\end{proof}

It remains to consider the case when $d$ preserves spin. Again we have to distinguish various cases. The most interesting case is when $G_2$ is \tcon\ and multi-ended, and so applying \Lr{LIIab1} to $G_2$ we obtain that $G_2$ has 2-coloured cycles containing $d$; we then have to distinguish two subcases according to which of ${b^*},c$ participates in those cycles. Let us first consider the case when this is $c$, and so $(cd)^p$ is a relation for some $p$. Interestingly, we can now apply \Tr{TIId11} to $G_2$ rather than \G: recall that $G_2$ is \tcon\ by \Prr{kG2}, and that $\sig_2$ has spin behaviour as in the requirements of that theorem, except that the roles of the letters are now interchanged. Thus, substituting $d$ by $b$ and ${b^*}$ by $d$ we can apply \Tr{TIId11} to obtain a \plpr\ of $G_2$, namely
\labtequ{bst1}{$G_2 \isom Cay\left<d,c,{b^*}\mid d^2, c^2, {b^*}^2, (dc)^{2n}, (c{b^*})^m, ({b^*}dcd)^p \right>$, $n,m,p \geq 2.$}

In the second subcase, when $(b^*d)^p$ is a relation rather than $(cd)^p$, we can repeat the same arguments to obtain a similar presentation but with $c$ and $b^*$ interchanged. Thus, in this case we have  
\labtequ{bst2}{$G_2 \isom Cay\left<d,c,{b^*}\mid d^2, c^2, {b^*}^2, (d{b^*})^{2n}, ({b^*}c)^m, (cd{b^*}d)^p \right>$, $n,m,p \geq 2.$}

Note that these two cases are distinct: if the order of $b^*d$ is finite then the order of $dc$ is infinite and vice-versa. Indeed suppose that the order of $b^*d$ is finite, and recall that $d$ preserves spin while $b^*,c$ reverse spin in $G_2$. We will show that an infinite $cdcdc\ldots$ walk $W$, starting at an arbitrary vertex, meets infinitely many of the finite $b^*d$-cycles. Indeed, suppose there is a last $b^*d$-cycle $C$ met by $W$. Then $W$ cannot have met a $b^*d$-cycle in each side of $C$, for $C$ separates its sides. But $W$ arrived at $C$ along a $c$ edge, traversed a $d$-edge of $C$, and left $C$ by another $c$ edge. Now as $d$ preserves spin, those two $c$ edges lie in different sides of $C$, and are incident with $b^*d$-cycles other than $C$. This contradiction proves our claim. The same argument proves the reverse claim.

Thus we have obtained a presentation of \gt\ in the case where $d$ preserves spin too, and can now use this to deduce a presentation of \G.

\begin{theorem} \label{TIId12}
Let $G= Cay\left<b,c,d\mid b^2, c^2, d^2, \ldots \right>$ be a planar \tcon\  \Cg\ 
\note{redundant: with more than one end} and suppose that $bc$ has a finite order $n$, precisely one of $b,c$ preserves spin ($b$ say), and $d$ preserves spin. Then precisely one of the following is the case:
\begin{enumerate}\addtolength{\itemsep}{-0.5\baselineskip}
\item \label{ddi} $G \isom Cay\left<b,c,d\mid {b^2}, c^2, d^2, (bcbdcd)^m; (bc)^{2n} \right>, n\geq 2, m\geq 1.$
\item \label{ddii} $G \isom Cay\left<b,c,d\mid b^2, c^2, d^2, (bcbdcd)^p ; (dc)^{2n}, (bc)^{2m}\right>$, $n,m,p \geq 2$;
\item \label{ddiii} $G \isom Cay\left<b,c,d\mid b^2, c^2, d^2, (bcbdcd)^p; (dbcb)^{2n}, (bc)^{2m} \right>$, $n,m,p \geq 2$. 
\end{enumerate}
These presentations are planar.

Conversely, each of the above presentations, with parameters chosen in the specified domains, gives rise to a planar \tcon\ \Cg\ as above.
\end{theorem}
\begin{proof}
The \Cg\ $G_2$ (as defined in the beginning of this section) is \tcon\ by \Prr{kG2}. We distinguish two cases. 

{\bf Case I:} $G_2$ has at most 1 end.

Recall that both ${b^*},c$ reverse spin while $d$ preserves spin in $G_2$. Thus we are in type \ref{ziii} of \Tr{LG5}, and so  $G_2 \isom Cay\left<{b^*},c,d\mid b^{*2}, c^2, d^2, ({b^*}c)^n, ({b^*}dcd)^m\right>$, $n\geq 2, m\geq 1$. Applying \Tr{ccfdec} again we obtain 
$$G \isom Cay\left<b,c,d\mid {b^2}, c^2, d^2, (bcbc)^n, (bcbdcd)^m\right>, n\geq 2, m\geq 1,$$
and rewriting $(bcbc)^n$ as $(bc)^{2n}$ we obtain possibility \ref{ddi} of the statement. Clearly, this is a \plpr.

{\bf Case II:} $G_2$ is multi-ended.

In this case we have already obtained a presentation for $G_2$ in the above discussion; see \eqref{bst1} and \eqref{bst2}. By the same technique as in the first case, and a little bit of rearranging (a relation of the form $(WZ)^n$ is equivalent to $(ZW)^n$), we obtain the claimed presentations~\ref{ddii} and~\ref{ddiii}. It \ises\ that these presentations are planar using the spin behaviour; see \fig{fIId1}.

\medskip
The converse implication can be established as in the proof of \Tr{TIId11}, by explicitly constructing \g as a \twsqam\ of the corresponding \gt. This time we have to apply the converse implication of \Tr{TIId11} to obtain the desired \gt. 
\end{proof}

We observe the following fact, which follows from \Cr{CIIa1} and Theorems \ref{TIIa2}, \ref{TIId11} and \ref{TIId12}, and is interesting in view of the forthcoming counterexamples to \Cnr{bowa}.

\begin{corollary}\label{corIId1}
Let $G= Cay\left<b,c,d\mid b^2, c^2, d^2, \ldots \right>$ be a planar \tcon\  \Cg\ containing a 2-coloured cycle. Then every face of \g is finite.
\end{corollary} 
% *** ---- *** \begin{proof} 	\end{proof}

\subsection{Graphs without 2-coloured cycles} \label{secIIb}
In this section we consider the cubic multi-ended planar \Cg s $G \isom Cay\left<b,c,d\mid b^2, c^2, d^2, \ldots \right>$ that have no 2-coloured cycles. %We are now not always assuming that \g is \tcon, because our analysis of the \iicon\ case in \cite{cay2con} needs some of the results of this section. We can however assume that \g is \iicon, and that \g has a \prem.
Similarly to our analysis of the graphs on three generators that do have 2-coloured cycles (\Sr{secIIa}), we will have to distinguish cases according to the spin behaviour of the generators.

Recall that our analysis of the graphs with 2-coloured cycles in \Sr{secIIa} was very intimately connected with those cycles: in all non-trivial cases, we used  such cycles to split \G\  by finding a subgraph $G'_2$ in which those cycles bound faces. In the current case, the absence of 2-coloured cycles makes our task harder. However, we will still be able to use similar methods. We will be able to find  a good substitute for the 2-coloured cycles: namely, the minimal \psc s.

\subsubsection{All edges preserve spin} \label{secIIc}

In this section we consider a \iicon\ multi-ended \Cg\ $G \isom Cay\left<b,c,d\mid b^2, c^2, d^2, \ldots \right>$ with no 2-coloured cycles, that has a \prem\ \sig\ in which all edges preserve spin. We do not demand that \g be \tcon\ here, because the results of this section are needed for the characterization of graphs of connectivity 2 in \cite{cay2con}. %We can assume though that even if $\kappa(G)=2$ no edge of \g is a \sepe. Recall that an edge $uv$ is called a \sepe\ if $\{u,v\}$ separates the graph. 

Before we can state our main result of this section we need to define the concept of a \ncp. Intuitively, a \ncp\ is a finite word $\cp$ in the letters $b,c,d$ such that whenever \cp\ is a relation of a planar cubic \Cg\ \g all vertices of which have the same spin, no cycle of \g induced by \cp\ bounds a face, and no two cycles of \g induced by \cp\ cross. Making a formal definition out of this intuitive idea is a bit tricky in the absence of a concrete \Cg \G. 

For this, let $H$ be a plane graph the edges of which are coloured with the colours $b,c,d$ in such a way that no vertex is incident with more than one edge of the same colour, and let $C \subseteq H$ be a cycle in $H$. We will say that $C$ \defi{complies} with \cp\ (in $H$), if one of the words obtained by reading the colours of the edges of $C$, as we cycle once along $C$ once in a straight manner, is \cp\ and moreover all vertices of $C$ that have degree 3 in $H$ have the same spin.
Given two cycles $C,R$ in $H$ both complying with \cp, we will say that $R$ is a \defi{\rota} of $C$ if $R\cap C$ is a (possibly closed) path. The intuition of this definition is derived from the fact that $H$ can be though of as a \Cg\ in which $C,R$ are induced, starting at the same vertex, by relations that are obtained from each other by rotating the letters. 

We can now give the formal definition of a \ncp.
\begin{definition} \label{dncp}
A non-empty word $\cp$ in the letters $b,c,d$ is called a \defi{\ncp} if it satisfies the following conditions:
\begin{enumerate}\addtolength{\itemsep}{-0.5\baselineskip}
\item \label{ncp0} \cp\ contains all three letters $b,c,d$;
\item \label{ncp00} \cp contains no consecutive identical letters; 
\item \label{ncpi} \cp\ is not of the form $(bcd)^n$ up to rotation and inversion, and
\item \label{ncpii} if $C$ is a cycle complying with \cp\ then no \rota\ of $C$ crosses $D$.
\end{enumerate}
\end{definition}

This definition might look somewhat abstract at first sight, but in fact there is an easy algorithm that recognises \ncp s. \sss 

It will be easier to understand the necessity of the requirements of Definition \ref{dncp} if one considers the usage of \ncp s in the following theorem: we impose \ref{ncp0} because $b,c,d$ are involutions. With \ref{ncp00} we prevent 2-coloured cycles, which we have handled in earlier sections. We require \ref{ncpi} to prevent \cp\ from being a face boundary. Finally, \ref{ncpii} is the important property of \cp\ from which our results yield their strength.

We can now state the main result of this section, yielding a complete description the corresponding \Cg s.
\begin{theorem} \label{TIIc}
Let $G= Cay\left<b,c,d\mid b^2, c^2, d^2, \ldots \right>$ be a \iicon\ multi-ended \Cg. 
%\ with no \sepe. 
Suppose \g has a \prem\ in which all edges preserve spin, and that each of the elements $bc, cd,db$ has infinite order. Then precisely one of the following is the case:
\begin{enumerate}\addtolength{\itemsep}{-0.5\baselineskip}
\item \label{IIci} $G \isom Cay\left<b,c,d\mid {b^2}, c^2, d^2, (bcd)^2; (bcdc)^n \right>, n\geq 2$ (faces of size $6$);
\item \label{IIcii} $G \isom Cay\left<b,c,d\mid {b^2}, c^2, d^2, (bcd)^k; \cp \right>$,  $k\geq 3$ (faces of size $3k\geq 9$);
\item \label{IIciii} $G \isom Cay\left<b,c,d\mid {b^2}, c^2, d^2; \cp \right>$, (no finite faces),
\end{enumerate}
where $\cp$ is a \ncp.

Conversely, \fe\ $n$ or $k$ in the specified domains, and every \ncp\ $\cp$, the above presentations yield a planar \Cg\ as above.

In the first 2 cases \g is always \tcon. In case \ref{IIciii} $\kappa(G)=2$ if $\cp$ is regular and $\kappa(G)=3$ otherwise. If $\kappa(G)=2$ then \g has a \sepe\ \iff\ $\cp$ is strongly regular.
\end{theorem}

The above presentations are less explicit than the presentations we have obtained so far because of the presence of \cp. This is not a shortcoming of our analysis: there is no word with arithmetic parameters capturing all \ncp s, but all \ncp s are needed to make \Tr{TIIc} true. However, as we will see in the forthcoming proof, \ncp s have a rather simple structure and they are similar to each other. Since there is an algorithm that recognises them, the set of \ncp s, and thus the set of \Cg s described in \Tr{TIIc}, can be effectively enumerated.

The rest of this section is devoted to the proof of \Tr{TIIc}, which is completed in page \pageref{secIIdum}. The reader who does not wish to see the details yet could skip the rest of this section, as well as \Sr{secIId} which is similar, and continue with \Sr{afterIId} in page \pageref{afterIId}.

\subsubsection*{Faces and \psc s}

Let $G$ be a graph as in \Tr{TIIc} fixed throughout this section and let $\Gam$ be its group. Note that as we are assuming that all colours preserve spin, 
\labtequ{bcdbcd}{every facial walk is of the form $\ldots bcd bcd \ldots$ or the inverse (\fig{fIIc}).}
\showFig{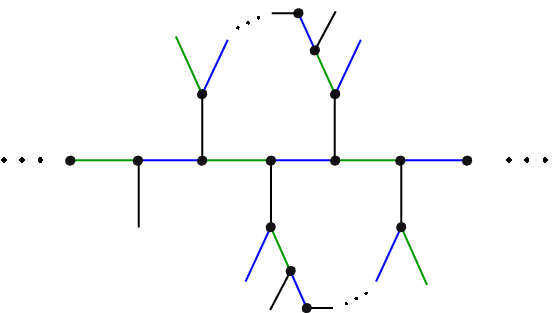}{The local situation around a vertex in the case that all edges preserve spin and no 2-coloured cycles exist.}

This means that any two face boundaries look locally the same, but in fact more is true: it is easy to check that
\labtequ{samef}{Any two face boundaries of \sig\ can be mapped to each other by a \auto\ of \G.}
Indeed, recall that \sig\ is a \prem\ in which all edges preserve spin. Thus, \fe\ vertex $x$, the three \auto s of \g exchanging $x$ with its neighbours can be used to map any of the three face boundaries incident with $x$ to each other.

In particular, all faces in \sig\ have the same size, which by \eqref{bcdbcd} is a multiple of 3. It cannot be equal to 3 though, for this would mean that $G \isom  Cay\left<b,c,d\mid b^2, c^2, d^2, bcd \right>$ and this graph has \sepe s, which we are assuming is not the case for \G. Thus we have proved that
\labtequ{fsize}{For some $\nin \cup \sgl{\infty}$, every face of \sig\ has size $N=3n+6$.}

As already mentioned, our analysis of the graphs without 2-coloured cycles will be based on their \mpsc s. We begin by showing that they do exist:
\labtequ{pscex}{\g has a \psc.}
Indeed, if all faces in \sig\ are finite, then this follows immediately from \Lr{LG3}. If there is an infinite face in \sig, then since \sig\ is consistent, every vertex is incident with an infinite face by \Lr{imrcb}. Let $F$ be an infinite face of \G, let $v e x e' w$ be a facial walk incident with $F$ comprising two edges $e,e'$, and let $f$ be the third edge incident with $x$. Since we are assuming that \g has no \sepe, there is a \pth{w}{v}\ $P$ in $G$ that avoids both endvertices $x,x'$ of $f$. We claim that the cycle $C = vxwPv$ is \ps. Indeed, by its construction $C$ separates $F$ from $x'$, and as $F$ is infinite and $x'$ must be incident with an infinite face too by the above remark, both sides of $C$ contain infinitely many vertices. This completes the proof of \eqref{pscex}.

So let $C$ be a \mpsc\ of \G. Our next claim is that 
\labtequ{LII2}{any maximal 2-coloured subpath of $C$ containing at least 3 edges contains an odd number of edges.}
This can be proved by an argument very similar to the one we used in the proof of \eqref{LIb1} for the case when $p$ is even.

Using this we are now going to prove that 
\labtequ{LII4}{$C$ has no  facial subpath containing more than 3 edges.}
To see this, let $F$ be a longest facial subpath of $C$, and suppose to the contrary that $||F||>3$. We may assume \obda\ that no other \mpsc\ $C'$ has a facial subpath longer than $||F||$, for otherwise we could have chosen $C'$ instead of $C$.

Recall that, by \eqref{bcdbcd}, every facial walk is of the form $\ldots bcd bcd \ldots$ or the inverse (\fig{fIIc}). Since all colours behave the same way in this case, we may assume \obda\ that $F$ starts with a subpath $F_4$ of the form $bcdb$. Let $\cf$ be the face  incident with $F$. Consider the \auto\ $g$ of \g mapping the fourth vertex $y$ of $F_4$ to its first vertex $x$, and let $C':= gC$; see \fig{fIIcCross}. It is easy to check that $C'$ crosses $C$: indeed, note that $C'$ contains an edge $e$ incident with both $x$ and $\cf$, and so $C'$ must leave $\cf$ before $C$ as no \mpsc\ can have a longer subpath incident with $\cf$ than $F$. This crossing however contradicts \Cr{CII3p}. This contradiction proves \eqref{LII4}.
\medskip

\showFig{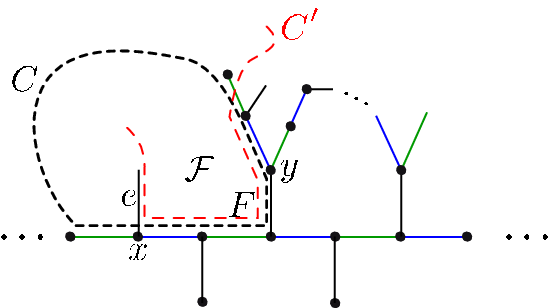}{}

\subsubsection*{The dominant colour.}

Our next assertion shows that even though in this case \g has no 2-coloured cycle, it must have cycles that are not far from being 2-coloured:
\labtequ{LIIA}{There is a colour $a\in \{bcd\}$ \st\ \fe\ \mpsc\ $C$ of \G,  every other edge of $C$ is coloured $a$. In particular, $|C|$ is even.}
Let $P$ be a maximal 2-coloured subpath of $C$. Obviously, $P$ contains at least two edges. Moreover, $P$ cannot consist of precisely two edges, because then the subpath $P'$ of $C$ comprising $P$ and its two incident edges would be  facial by \eqref{bcdbcd}, and this would contradict \eqref{LII4} since $||P'||=4$. Thus, $||P||$ is at least 3, and it is odd by \eqref{LII2}. This means that the first edge of $P$ has the same colour $a$ as its last edge. Assume \obda\ that $a=c$.

The path $P$ is a good starting point in our attempt to prove \eqref{LIIA}: every other edge of $P$ is coloured $c$ since it is 2-coloured. And indeed, we will be able to extend it by adding further 2-coloured subpaths of $C$, retaining the property that every other edge is coloured $c$, until exhausting all of $C$. For this, let $e',e$  be the last two edges of $P$ and let $f,f'$ be the two edges of $C$ succeeding $P$, appearing in $C$ in that order (\fig{fIIcA}). Note that as $P$ was chosen to be maximally 2-coloured, $e'$ and $f$ have different colours. This, combined with  \eqref{LII4}, implies that $f'$ must be coloured $c$, for otherwise the subpath of $C$ spanned by $e',e,f,f'$ is facial. 
\showFig{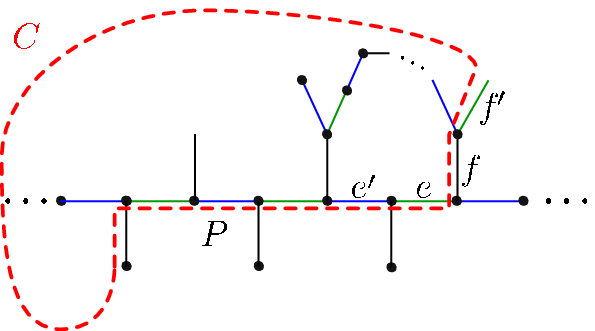}{The situation around $P$ in the proof of \eqref{LIIA}.}

Thus $e,f,f'$ span a 2-coloured subpath of $C$. Let $P'$ be the maximal 2-coloured subpath of $C$ containing these three edges. By  \eqref{LII2} $||P'||$ is odd. Moreover, it does not contain $e'$ as $f$ and $e'$ have different colours. Thus  $P'$ starts and ends with a $c$-edge. Consider the path $P \cup P'$, and note that every other edge of this path is coloured $c$. Now starting with this path instead of $P$ and repeating the above arguments, we find a longer odd subpath of $C$ every other edge of which is coloured $c$. Continuing like this we prove that every other edge of $C$ bears the same colour. Let us call this colour the \defi{dominant} colour of $C$. Note that the dominant colour of a cycle is always unique in the current case, since no 2-coloured cycles exist. 

We just proved that every \mpsc\ of \G\ has a dominant colour. It remains to prove that they all have the same dominant colour. So suppose that $C,D$ are \mpsc s of \g with distinct dominant colours $c,d$ respectively, say. Since all cycles are 3-coloured, both $C,D$ contain a $b$ edge, and we may assume that this $b$ edge is the same edge $e$ in both cases, for otherwise we could have considered translates of $C,D$ through $e$. Now $e$ is surrounded by two $c$ edges in $C$, and it is surrounded by two $d$-edges in $D$. As $e$ preserves spin, this means that $C$ and $D$ cross each other at $e$. By the remark preceding \Cr{CII3p}, we can obtain a new \mpsc\ $C'$ by combining $C$ and $D$, following one of them up to $e$ and then switching to the other. But $C'$ then contains a $cbd$ subpath (with the edge $e$ in the middle), and can be chosen so that it also contains a $cdc$ subpath of $C$, contradicting the fact that every other edge of $C'$ must bear the same colour. This contradiction proves that the dominant colour is the same for every \mpsc\ of \g indeed, and completes the proof of \eqref{LIIA}.
\medskip

From now on we assume that the dominant colour of the \mpsc s of \g is $c$.

Similarly to the case when \g has two generators, we will have to consider the case when \g has hexagonal faces separately (see \Sr{secTidy}). Recall that by \eqref{fsize} every face of \sig\ has size at least 6.
\labtequ{LII5}{If \g has no face of size 6 then no two \mpsc s of \g cross.}
%  IF YOU CHANGE THIS PROOF ADAPT THE PROOF OF {Z}
Indeed, by our discussion in \Sr{secPsc}, in particular by \eqref{BF}, such a crossing gives rise to a finite region bounded by a cycle in $C \cup D$. As this region contains only faces of size larger than 6, we obtain a contradiction to Euler's formula \eqref{euler} in a way similar to the proof of \eqref{LIb6}. This proves \eqref{LII5}.
%\medskip

\subsubsection*{The hexagonal grid case.}

Let us now consider the case when one, and thus by \eqref{fsize} all, of the faces in \sig\ have size 6. We proceed as in \Sr{secTidy} to prove that
\labtequ{hexmpsc}{If \g has a face of size 6 then it has a \mpsc\ induced by $(bcdc)^n$.}
Indeed, similarly to \eqref{LIb8}, we can prove that every \mpsc\ $C$ contains the same amount of edges from each non-dominant colour $b,d$. And again, rerouteing  $C$ around some of its incident hexagons if needed, in other words, replacing subarcs of the form $bcd$ by $dcb$, which we are allowed to do since $bcd (dcb)^{-1}$ is a relation (inducing a hexagonal face), we modify $C$ into a cycle of the same length that is induced by a word of the form $(bcdc)^n$ and is still \ps. This completes the proof of \eqref{hexmpsc}.
\medskip

Note that translates of such a cycle cannot cross, and so we have obtained something similar to \eqref{LII5} for graphs with hexagonal faces.

We already have enough information to finish off the case when the faces of \g are hexagonal. It is now not hard to check that the relation $(bcdc)^n$ we just obtained combined with the one inducing the face boundaries, and of course the involution relations for the generators, yield a \plpr\ of \Gam:
$$G \isom Cay\left<b,c,d\mid {b^2}, c^2, d^2, (bcd)^2; (bcdc)^n \right>, n\geq 2.$$
This can be proved for example by showing that the underlying graph is isomorphic to one of the graphs in case \ref{bi} or \ref{biii} of \Tr{TIbx}; to see this, look at \g through a lens that identifies colours $b$ and $d$. The details are left to the reader.

Note that if we let $n=1$ in the above presentation we would obtain a graph in which $c$ edges are \sepe s, contradicting our assumptions on \G.

Conversely, for every $n\geq 2$ one can show that the above presentation corresponds to a planar \tcon\ \Cg\ with hexagonal faces by explicitly constructing such a graph: it consists of $2n$ `parallel' \dray s coloured $b,d$ in an alternating fashion, joined by edges coloured $c$.

\subsubsection*{Back to the main case: no hexagons.}
For the rest of this section we will be assuming that $N>6$.

\subsubsection*{The weak colour.}

We now turn our attention to the behaviour of the non-dominant colours in the \mpsc s of \G: it turns out that, in general, one of them is more `dominant' than the other. More precisely:
\labtequ{LII6}{Suppose that some \mpsc\ of \g has a $bcb$ subpath. Then no \mpsc\ of \g has a $dcd$ subpath.}%If \g has a face of size 6 then it has at least one \mpsc\ with no $dcd$ subpath.}
To see this, suppose that some \mpsc\ $C$ of \g has a $bcb$ subpath and some \mpsc\ $D$ of \g has a $dcd$ subpath. We may assume \obda\ that these two subpaths traverse the same $c$-edge $e$, for otherwise we can consider a translate of one of the two. As $e$ preserves spin, $C$ and $D$ cross each other at $e$, contradicting \eqref{LII5}. 
\medskip

Proposition \eqref{LII6} implies that at least one of the non-dominant colours $a\in \{b,d\}$ cannot appear too often in any \mpsc: between any two $a$-coloured edges in any \mpsc, there are edges of both other colours. We call this colour the \defi{weak} colour of \G, and we call the other non-dominant colour the \defi{\sedo} colour of \G, unless both non-dominant colours appear with the same frequency, in an alternating fashion, in which case both non-dominant colours are called \defi{weak}. Note that if this is the case, then every \mpsc\ is induced by a word of the form $(bcdc)^n$. 

The argument of the proof of \eqref{LII6} can be repeated to  prove something stronger: 
\labtequ{LII7}{Suppose that some \mpsc\ of \g has a $bWb$ subpath, where $W$ is any word in the letters $b,c,d$. Then no \mpsc\ of \g has a $dWd$ subpath.}
Indeed, such a word $W$ must start and end with the dominant colour $c$, and it is straightforward to check, using the fact that $c$ preserves spin, that if some \mpsc\ $C$ had a $bWb$ subpath and some \mpsc\ $D$ had a $dWd$ subpath then some translate of $D$ would cross $C$, contradicting \eqref{LII5}.

\subsubsection*{Maximal common subpaths and the word $Z$.}
We have already seen that  two \mpsc s cannot cross each other. However, they may be tangent, that is, they may have a common subpath provided one is contained in the other. Our next assertion restricts the possible common subpaths.
\labtequ{prefix}{Let $P$ be a maximal common subpath of two distinct \mpsc s $C,D$ of \G, and let $P'$ be another maximal common subpath of two distinct \mpsc s $C',D'$. Let $W_P, W'_P$ be words inducing $P,P'$ respectively. Then one of these words is a \defi{prefix}, i.e.\ an initial subword, of the other.}
The word `maximal' here is meant \wrt\ inclusion: $P$ is a maximal common subpath of $C,D$ if $P\subset C,D$ and \fe\ path $Q$ properly containing $P$ either $Q\not \subset C$ or $Q\not \subset D$ holds.

This proposition implies that there is a unique word $Z=Z(G)$, namely the maximal word for which there are two distinct \mpsc s of \G\ and a common subpath of theirs induced by this word, that governs all the possible intersections of \mpsc s: every maximal common subpath of two \mpsc s is induced by a prefix of $Z$ (or $Z$ itself).

To prove \eqref{prefix}, let $P,C,D,P',C',D'$ be as in the assertion, and suppose that none of $W_P, W'_P$ is a prefix of the other. We may assume \obda\ that $P$ and $P'$ have the same initial vertex $x$, for otherwise we can translate $C',D'$ by a \auto\ of \g to achieve this (\fig{fpref}).
\showFig{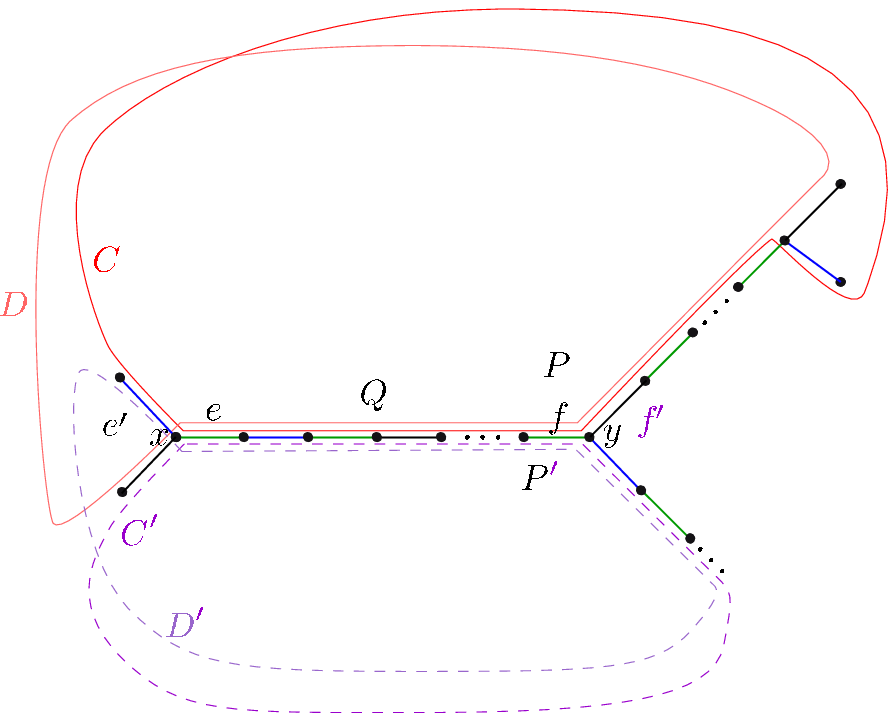}{The hypothetical situation when two maximal common subpaths of \mpsc s are not a subpath of each other.}

Easily, both $P$ and $P'$ must begin (and end) with an edge $e$, incident with $x$,  of the dominant colour $c$. Moreover, one member of each of the pairs $C,D$ and  $C',D'$ contains the $b$ edge incident with $x$ and the other member contains the $d$-edge, for $P$ and $P'$ are maximal common subpaths. Now since we are assuming that neither of $W_P, W'_P$ is a prefix of the other, there is a  common  vertex $y\in P\cap P'$ after which the two paths split for the first time. Again, the common edge $f$ leading into $y$  must bear the dominant colour $c$, and so one of the paths $P,P'$ follows the $b$ edge and the other follows the $d$-edge incident with $y$ (\fig{fpref}). Consider the common subpath $Q:= xPy = xP'y$, and note that $Q$ is a subpath of all four cycles $C,D,C',D'$. Moreover, for each of the four ways to choose a non-dominant edge $e'$ incident with $x$ and a non-dominant edge $f'$ incident with $y$, one of these cycles contains $e' Q f'$ as a subpath. In particular, if $W_Q$ is a word inducing $Q$, then both the words $b W_Q b$ and $d W_Q d$ are induce subpaths of \mpsc s. But this contradicts \eqref{LII7}, and so \eqref{prefix} is proved.
\medskip

Proposition \eqref{prefix} implies that $Z(G)$ is symmetric. Indeed, $Z^{-1}$ also induces a common subpath  of \mpsc s, so it must be a prefix of $Z$. This implies that the element of \Gam\ described by $Z$ is an involution.

Next, we claim that
\labtequ{Z}{\Fe\ path $P$ in \g induced by $Z$ \ta\ precisely two \mpsc s $C,D$ containing $P$. Moreover, $C \cap D=P$, that is, $C$ and $D$ have no common edge outside $P$.} %Moreover, $C$ separates $F$ from the edges incident with the interior of $P$.}
Indeed, by the definition of $Z$ and $P$ \ta\ at least two \mpsc s $C,D$ containing $P$. If there was a third one $C'$ then, as \g is cubic, $C'$ would have an edge $f$ incident with $P$ in common with one of $C,D$. But this would contradict the maximality of $Z$ as $f\cup P$ would then be a common subpath.

If $C$ and $D$ have a further common edge $e$ outside $P$, then by the above argument $e$ is not incident with $P$. This means that we can choose subpaths $C',D'$ of $C,D$ such that $C' \cup D'$ is a cycle shorter that $C$. Thus this cycle cannot be \ps, and so one of its sides contains only finitely many vertices. This leads to a contradiction as in the proof of \eqref{LII5}.

\subsubsection*{Face incidences and the word $A$.}

Having analysed the ways that \mpsc s can intersect each other let us now see how a \mpsc\ can intersect a face boundary. It turns out that there are \mpsc s nicely arranged around each face boundary $F$, each of them using precisely three consecutive edges of $F$. Recall that by \eqref{bcdbcd} every facial walk is of the form $\ldots bcdbcd \ldots$ or the inverse. 
\labtequ{cycfac}{For every face boundary $F$ and every $bcd$ subpath $P$ of $F$ \ti\ a \mpsc\ $C$ containing $P$ and no other edge of $F$.} %Moreover, $C$ separates $F$ from the edges incident with the interior of $P$.}
Indeed, the fact that every cycle is 3-coloured and \eqref{LIIA} imply together that every \mpsc\ contains a $bcd$ subpath. Translating this subpath to $P$ we thus obtain a \mpsc\ $C$ containing $P$. By \eqref{LII4} $C$ contains none of the two edges of $F$ incident with $P$. It remains to check that $C$ can also not contain an edge $e$ of $F$ not incident with $P$. But if this was the case then the \auto\ $g$ of \g mapping the first vertex of $P$ to its last vertex would translate $C$ to a cycle $C'=gC$ that crosses $C$: indeed, note that $g$ fixes $F$ as $F$ is a concatenation of translates of $P$, and so $C'$ must also contain a further edge $e'=ge$ on $F$. An easy topological argument now shows that $C$ and $C'$ cross indeed, contradicting \eqref{LII5}. This proves \eqref{cycfac}.
\medskip 

This motivates us to define the word $A:= bcd$, which will play an important role in the sequel.

\subsubsection*{The word $(AZ)^n$ inducing the \mpsc s.}

It turns out that there is a word inducing all the \mpsc s of \G:
\labtequ{azaz}{For some $n\geq 2$, every \mpsc\ of \g is induced by the word $(AZ)^n$. In particular, any two \mpsc s can be mapped to each other by a \auto\ of \G.}
To prove this, pick a pair $C,C'$ of \mpsc s that have a common path $P$ induced by $Z$, and let $x$ be an endvertex of $P$. Let $F$ be the face-boundary containing the $b$ and $d$-edge incident with $x$. We claim that each of $C,C'$ contains a $bcd$ (or $dcb$) subpath of $F$ incident with $x$; in other words, $C$ and $C'$ are as in \eqref{cycfac} (\fig{fazaz}).

\showFig{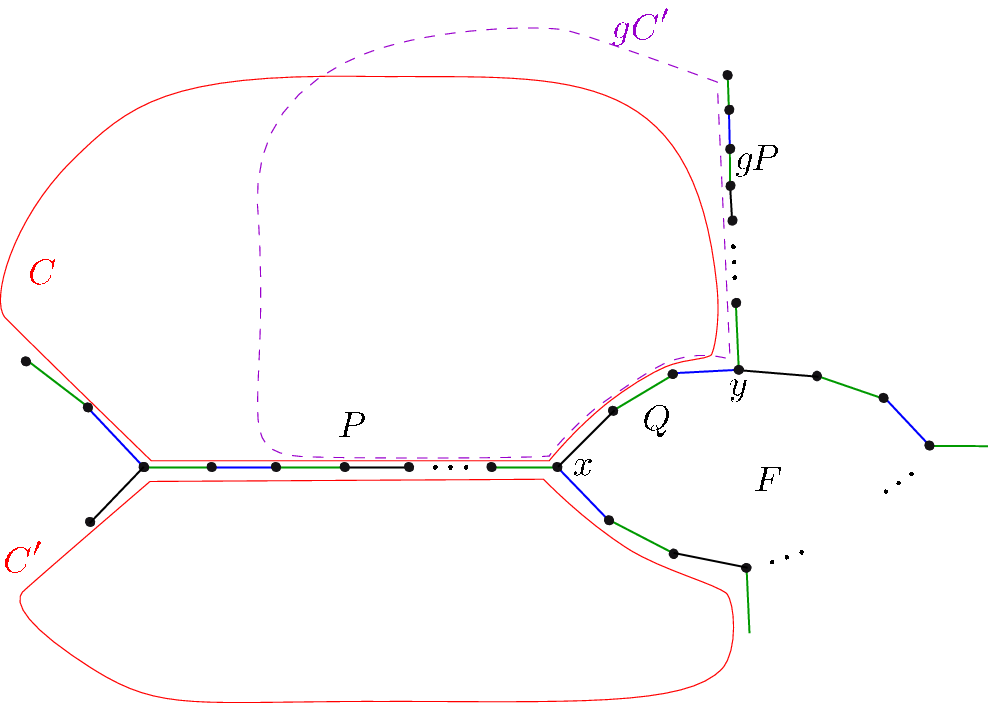}{The situation arising in the proof of \eqref{azaz}.}

To begin with, note that as $x$ was by definition a last common vertex of $C$ and $C'$, one of them, $C$ say, contains the $b$ edge incident with $x$ and the other contains the $d$-edge. By \eqref{LII4} $C$ cannot have a subpath on $F$ containing more than three edges. So it suffices to show that $C$ does not leave $F$ after having traversed less than three edges. If this is the case though, then the \mpsc\ $D$ containing the $bcd$ subpath $Q$ of $F$ starting at $x$, which cycle is provided by  \eqref{azaz}, is distinct from $C$. Thus, $D$ cannot contain all of $P$ because then it would have a common subpath with $C$ properly containing $P$ and this would contradict the choice of $P$. So $D$ leaves $P$ at some of its interior vertices, which means that $D$ enters a side of one of $C,C'$ not containing $F$. On the other hand, $D$ also meets the side of each of $C,C'$ that does contain $F$; see (\fig{fazaz}). This means that $D$ crosses one of $C,C'$, contradicting \eqref{LII5}.

This proves that, as claimed, $C$ contains the $bcd$ subpath of $F$ incident with $x$. By the same argument we can prove that $C'$ contains the $dcb$ subpath of $F$ incident with $x$, but we will not need this. This is a good start for the proof of \eqref{azaz}: we just proved that $C$ contains an $AZ$ subpath. 

Now consider the \auto\ $g$ of \g mapping $x$ to the other endvertex $y=xbcd$ of $Q$. Repeating the above arguments we see that one of $gC,gC'$, in fact it must be $gC'$, also contains $Q$. We claim that one of $C, gC'$ must contain all three paths $P, Q, gP$. For suppose that $C$ misses part of $gP$ and $gC'$ misses part of $P$. Then, as $C$ is not allowed to cross $gC$ by \eqref{LII5}, it leaves $gP$ entering the side of $gC'$ not containing $F$. Similarly, $gC'$ leaves $P$ entering the side of $C$ not containing $F$. But if each of $C,gC'$ meets the side of the other not containing a face then the two cycles must cross each other contradicting \eqref{LII5}. This proves our claim that one of $C, gC'$ must contain all three paths $P, Q, gP$. Note that by the definition of $Z$ and the fact that both these cycles contain $Q$, this immediately implies that $C = gC'$.

We just proved that $C$ has a $ZAZ^{-1} = ZAZ$ subpath. Now repeating the previous arguments at the other end of $P$ we prove that $C$ has a $AZAZ$ or $A^{-1}ZAZ$ subpath. The latter possibility can however not occur, for it would mean that $C$ crosses $C'$ when leaving $P$; this can be seen by observing the spin of the end-vertices of $P$. Thus $C$ has a $(AZ)^2$ subpath. Continuing like this, we prove that $C$ is induced by $(AZ)^n$. Moreover, $n\neq 1$ for otherwise we would easily obtain, with the above arguments, that the faces have size 6 which we are assuming is not the case.

By the same arguments, we can prove that $C'$ is induced by $(A^{-1}Z)^n$, but as $Z=Z^{-1}$ this means that $C'$ is induced by $(AZ)^n$ too. 

We started with  $C,C'$ being arbitrary \mpsc s having a subpath induced by $Z$. Thus if we could prove that every \mpsc\ of \g has such a subpath this would complete the proof of \eqref{azaz}. This is indeed the case. For let $D$ be a \mpsc\ and let $R$ be a maximum-length subpath of $D$ shared with another \mpsc\ $D'\neq D$. Easily, $R$ has at least one edge. By \eqref{prefix}, $R$ is an initial subpath of some path $P$ induced by $Z$ which is a maximal common subpath of two \mpsc s $C,C'$. Now if $|R|\neq |P|$ then $D\neq C,C'$. But $D$ shares an edge $e$ not contained in but incident with $P$ with one of $C,C'$, for these two cycles together use all edges incident with the first vertex of $P$. This means that a common subpath of $D$ and one of $C,C'$ is $e R$, contradicting the maximality of $R$. This proves that every \mpsc\ of \g has a subpath induced by $Z$ and  completes the proof of \eqref{azaz}.
%\medskip

\subsubsection*{The subgroup $\Gam_2$, \socs, and the subgrpaph $G'_2$.}

Let $a,z$ be the elements of \Gam\ corresponding to the words $A,Z$ respectively, and let $\Gam_2$ be the subgroup of \Gam\ spanned by $a,b$. This subgroup, and the above results relating $A$ and $Z$ to \mpsc s, will allow us to follow an approach similar to that of the previous sections, the \mpsc s now playing the role of the monochromatic cycles of \Sr{secIa} or the 2-coloured cycles of \Sr{secIIa}. 

As in those sections, we are going to show that \g is a union of subdivisions of isomorphic copies of the \Cg\ $G_2$ of $\Gam_2$ \wrt\ the generating set $\{a,z\}$. We would like to define this kind of subdivision $G'_2$ of $G_2$ similarly to previous sections, by deleting for each \mpsc\ $C$ meeting  $\Gam_2$ all vertices in one of the sides of $C$, and then suppress vertices of degree 2 to obtain $G_2$ from $G'_2$. However, things are more complicated now and we need some preparatory work before we can show that this operation yields $G_2$ indeed. They reader may choose to skip this preparatory work and continue reading after \eqref{indep}, perhaps after having a look at the following vital definition.

\begin{definition} \label{defsoc}
Let $C$ be a \mpsc\ of \G, and let $x$ be a vertex of $C$ \st\ the word $(AZ)^n$ induces $C$ if the starting vertex is $x$. Then, we call the set of vertices of $C$ that can be reached from $x$ by subarc of $C$ induced by a prefix of $(AZ)^n$ a \defi{\soc} of $C$.
\end{definition}
Note that, by \eqref{azaz}, every \mpsc\ has at least one \soc. One of the major points of this paper, to be proved in the sequel, is that every \mpsc\ has precisely two \socs, one corresponding to each of its sides as indicated by our next claim:
\labtequ{socside}{For every \mpsc\ $C$ of \g and every \soc\ $S$ of $C$, one of the sides of $C$ contains all edges incident with $C$ at an element of $S$.}
This can be seen by observing the spin behaviour or by using \eqref{LII5}, see \fig{fazaz}.

%and any \soc\ contains precisely $2n$ vertices. 

Let us next check that \mpsc s do not separate cosets of $\Gam_2$ (compare this with the behaviour of 2-coloured cycles in earlier sections).
\labtequ{delta}{For every \mpsc\ $C$ of \g and every left coset $\Delta$ of $\Gam_2$ in \G, at most one of the sides of $C$ contains elements of $\Delta$.}
Suppose to the contrary there are elements $x,y\in \Delta$ in distinct sides of $C$, and let $W=w_1 \ldots w_k$ be a word with letters $w_i\in\{A,A^{-1},Z\}$ inducing  a \pth{x}{y}\ $P=xPy$ in $G$. Assume that $W$ has minimum length among such words. Define the corners of $P$ to be its vertices reachable from $x$ by paths induced by prefixes of $W$, and note that every corner of $P$ lies in $\Delta$ by definition. Moreover, by the minimality of $W$ the only corners of $P$ that do not lie on $C$ are $x$ and $y$.

By the definition of $Z$ and \eqref{cycfac} there is for every letter $w_i$ a \mpsc\ $C_i$ containing the corresponding subpath $P_i$ of $P$ induced by $w_i$. Even more, $C_i$ and $C_{i+1}$ have a common subpath $Q_i$ induced by $Z$ \fe\ relevant $i$, and the endvertices of $Q_i$ are corners of $P$. It might be the case that $C_i=C_{i+1}$. As  $x,y$ lie in distinct sides of $C$, and no $C_i$ can cross $C$ by \eqref{LII5}, $C$ \defi{bounds} $C_1$ from $C_k$, i.e.\ $C_1,C_k$ lie in distinct closed sides of $C$. This means that \ti\ an $i$ \st\ either $C$ bounds $C_i$ from $C_{i+1}$, or $C_{i+1}=C$ and $C_i\neq C$. If the former is the case then $Q_i$ must clearly be a subarc of $C$.. But then one of $C_i,C_{i+1}$ has a common subpath with $C$ that properly contains $Q$ since \g is cubic, and this contradicts the maximality in the definition of $Z$. If the latter is the case, then $Q_i$ is contained in $C=C_{i+1}$, and so $C_i$ must leave $C$ immediately before and after $Q_i$. But \eqref{socside} now implies that $P$ leaves $C$ entering the side from which it approached $C$, contradicting our assumptions. This proves \eqref{delta}.
\medskip

Using this we can prove the following observation. A \defi{metaedge} is a path of \g induced by one of the words $A,Z$.

\labtequ{indep}{Let $P\neq P'$ be metaedges. If both $P,P'$ have endvertices in $\Gam_2$ then they are independent.}
Recall that two paths are called \defi{independent} if their interiors are disjoint.

To prove this, let $C, C'$ be a \mpsc\ containing $P,P'$ respectively, which exists by \eqref{Z} and \eqref{cycfac}. Suppose first that both $P,P'$ are induced by the word $A$. Let $\cf,\cf'$ be the face whose boundary contains $P,P'$ respectively, and note that $\cf\neq\cf'$ since $P\neq P'$. It \ises\ that $C$ separates $\cf$ from $\cf'$, because none of $\cf,\cf'$ can contain the other. But then, $C$ separates the vertices of \gat\ that lie on the boundary of $\cf$ from the vertices of \gat\ that lie on the boundary of $\cf'$, which contradicts \eqref{delta}.

Suppose now that $P$ is induced by the word $Z$. Let $D\neq C$ be a further \mpsc\ containing $P$, provided by \eqref{Z}. By \eqref{LII5} the closure of one of the sides of $C'$ contains both $C,D$. Thus, by an easy topological argument, one of $C,D$, call it $K$, separates the other from $C'$. But then both sides of $K$ meet \gat\ as each of $C',C,D$ contains a vertex in \gat\ not contained in $K$. Again, this contradicts \eqref{delta}, and so \eqref{indep} is established.

\medskip
By \eqref{indep} the \Cg\ $G_2$ of $\Gam_2$ \wrt\ the generating set $\{a,z\}$ has a \topem\ in \G: we can obtain $G_2$ from $G$ by substituting for every two adjacent vertices $x,y$ of  $G_2$ the \pth{x}{y}\ in \g  induced by $A,A^{-1}$ or $Z$ by an $x$-$y$ edge labelled $a$ or $z$ accordingly. This yields indeed a  \topem\ of  $G_2$ in \G\ since by \eqref{indep} all these paths are independent. Starting with $G_2$ and replacing each edge back by the corresponding path induced by $A,A^{-1}$ or $Z$ we obtain the subdivision $G'_2$ of $G_2$ alluded to earlier. Compared with $G_2$ the graph $G'_2$ has the advantage that it is a subgraph of \g while still capturing the structure of $G_2$. As in earlier sections, this will come in handy later, when we will try to yield a presentation of \Gam\ from a presentation of $\Gam_2$.

Moreover, for every coset $\Delta$ of $\Gam_2$ in \Gam\ we find an isomorphic copy of $G'_2$ in \g whose vertices of degree 3 are precisely the elements of $\Delta$: such a copy can be obtained by mapping any vertex of $\Gam_2$ to any vertex in $\Delta$ by a \auto\ of \G. For every such copy define its \defi{corners} to be its vertices of degree 3; in other words, the elements of the corresponding coset.

\subsubsection*{A \plpr\ of $\Gam_2$.}

Note that $G_2$ is a cubic \Cg\ on two generators $a,z$, and our embedding \sig\ of \g induces, when combined with the aforementioned \topem\ of $G_2$ in \G, an embedding $\sig_2$ of $G_2$ in the sphere. It is straightforward to check that both $a,z$ preserve spin in $\sig_2$, for example using the fact that all vertices of \g have the same spin in \sig. It is also not hard to see that $G_2$ is \iicon: apply \Lr{Liicon} using the fact that $(az)^n$ is a relation by \eqref{azaz}. Now if $a$ has finite order, then \Lr{apres} implies that $G_2$ has at most one end, and so we can obtain a \plpr\ of $G_2$ from \Tr{LG2} \ref{oi}: $G_2 \isom Cay \left<a,z\mid z^2, a^k, (az)^n\right>$, $k\geq 3$, $n\geq 2$. Note that $n$ can be read off \eqref{azaz} in this case.

If $a$ has infinite order, then  $G_2$ cannot be \tcon\ by \Lr{apres2}. Thus $\kappa(G_2)=2$, and we can obtain a \plpr\ of $G_2$ from \Tr{main2} \ref{i}: $G_2 \isom Cay \left<a,z\mid z^2, (az)^n\right>$, where again $n$ is as in \eqref{azaz}. Cases \ref{ii} and \ref{iii} of \Tr{main2} cannot arise here, because we already know from \eqref{azaz} that $az$ has finite order and this is not the case in these groups; see \cite{cay2con} for more details.

As in earlier sections we are going to plug these presentations into \Tr{ccfdec} to obtain a presentation of \Gam.

\subsubsection*{Splitting \g into copies of $G'_2$.}

Before we can apply \Tr{ccfdec} we need a couple of further preparatory observations.
Call a cycle of a copy of $G'_2$ \defi{basic} if it is a \mpsc\ of \G. Also call the corresponding cycle of $G_2$ basic. 
\labtequ{basec}{For every pair of distinct copies $H,H'$ of $G'_2$ \ti\ a unique basic cycle $C$ of $H$ that bounds $H$ from $H'$.}
For this, let $e$ be an edge in $E(H')\sm E(H)$, which must exist if  $H,H'$ are distinct, and  let $P$ be a (possibly trivial) \pth{e}{H}\ in \G. Let $f$ be the unique edge in $P\cup e$ incident with a vertex $v$ of $H$. Then $v$ is not a corner of $H$ because $f\not\in E(H)$. Thus $v$ lies in the interior of a path $Q$ of $H$ induced by $A$ or $Z$. If $Q$ is induced by $A$, then the unique basic cycle $C$ of $H$ containing $Q$ bounds $f$ from the face boundary $F$ of \g containing $Q$. By \eqref{socside} $H$ meets the side of $C$ containing $F\sm Q$ as this side contains $Q'$. By \eqref{delta}, $H$ does not meet the other side of $C$ which contains $f$. Applying \eqref{delta} again but this time on the coset of $H'$, we obtain that $H'$ only meets the side of $C$ that does contain $f$ since $P$ must lie in that side. Thus $H -C$ and $H'-C$ lie in distinct sides of $C$, and so $C$ is as desired. 

In the other case, where $Q$ is induced by $Z$, a similar argument applies except that now there are two basic cycles of $H$ containing $Q$, and we have to choose $C$ to be the one bounding $f$ from the other.

The uniqueness of $C$ follows easily from the fact that for every other basic cycle $D$ of $H$ we now know that $C$ bounds $D$ from $H'$.
\medskip

Using \eqref{basec} we can prove the following.
\labtequ{inout}{For every copy $H$  of $G'_2$ and every basic cycle $C$ of $H$ there is a copy $H'\neq H$ of $G'_2$ \st\ $C\subseteq H'$ and $C$ bounds $H$ from $H'$.}
\note{you need $C$ to be a basic cycle of $H'$ too?}
To see this, recall that $C$ bounds a face \cf\ of $H$ by \eqref{delta}.
Let  $P$ be a maximal subpath of $C$ that is contained in a basic cycle $C'$ of some copy $H'$ of $G'_2$ that lies in $\cf \cup C$. To see that such copies exist, note that as $C$ is \ps\ in \G, there is an edge $xy$ of \g \st\ $x\in V(C)$ and $y$ lies in \cf. Easily, $x$ is not a corner of $H$ for otherwise $xy$ would lie in $H$. So let $M$ be the coset of $G_2$ containing $x$, and let $H_M$ be the corresponding copy of $G'_2$. Then $H_M$ must be contained in $\cf \cup C$ by \eqref{delta}, and  it meets $C$ at $x$. Thus any basic cycle of $H_M$ containing $x$ is a candidate for $C'$.

We claim that $P=C$ (in particular, $P$ is a closed path). Suppose to the contrary that $P$ has distinct endvertices $v,w\in V(C)$. Consider the copy  $J$ of $G'_2$ corresponding to the coset of $\Gam_2$ containing $v$, and note that, again by \eqref{delta}, $J$ lies in $\cf \cup C$ too as it contains an edge $e$, incident with $v$, that lies in \cf\ (\fig{finout}). Moreover $J$ contains the edge $f\in E(C)$ incident with $v$ that does not lie in $P$. 
\showFig{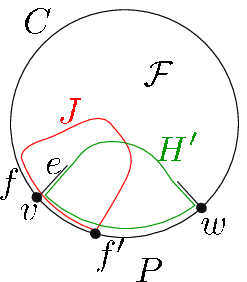}{The situation in the proof of \eqref{inout}.}

By \eqref{basec} \ti\ a basic cycle $D$ of $J$ that bounds $J$ from $H$. Thus, easily, $E(H) \cap E(J) \subseteq D$; in particular, $E(C) \cap E(J) \subseteq D$. Let $P'$ be the maximal common subpath of $C$ and $J$ containing $f$, and note that $P'\subseteq D$ too. By the maximality of $P$, and as $P'$ contains the edge $f\not\in E(P)$, we have $P\not\subseteq P'$. By an easy topological argument, $D$ now crosses $C'$ contradicting \eqref{LII5}. This contradiction proves our claim that $P=C$, which immediately implies   \eqref{inout}.
\medskip

It is straightforward to strengthen \eqref{inout} to demand that \ti\ a \auto\ $g$ of \g \st\ $H'=gH$ and $C=gC$; just use the fact that $C'$ in the above proof is a translate of $C$ by \eqref{azaz}. Thus, any \soc\ $S$ of $C$ is mapped by $g$ to a further \soc\ $S'$ of $C$, which is distinct from $S$ as, by \eqref{socside} and the choice of $H'$, it follows that $S$ and $S'$ point into distinct sides of $C$. Note that $C$ cannot have a third \soc, as this would have to share a side with one of $S,S'$ easily yielding a contradiction to \eqref{LII5}. Thus, every \mpsc\ has two \defi{dual} societies, one for each of its sides. This remarkable fact is one of the main ideas of this paper, and it is not peculiar to the current section: dual \socs\ were implicit in all cases we have seen so far. For example, the vertices on the big cycle of \fig{fiamal} (iii) pointing to one of its sides form a \soc, and the reader will easily spot the dual \socs\ of the cycles of \fig{fidih} (ii) corresponding the 4-cycles of \fig{fidih} (i).
\medskip

Back to our analysis of \G, we note that for every vertex $x$ of \g \ti\ a coset of $\Gam_2$ containing $x$, and so there is a copy $H^x$ of $G'_2$ in \g containing $x$ as a corner. We will use these graphs $H^x$ in our application of \Tr{ccfdec}.

For this, let $T$ be the graph with vertex set $\{H^x \mid x\in G\}$ in which two vertices $H^x, H^y$ are joined by an edge if they share a basic cycle bounding $H^x$ from $H^y$. We claim that $T$ is a tree. It is easy to see that $T$ is acyclic, since any cycle would yield a topological impossibility. To prove that $T$ is connected, suppose it is not and choose vertices $H^x, H^y$ in distinct components of $T$ minimizing the distance from $H^x$ to $H^y$ in \G. Let $C_0$ be the basic cycle of $H^x$ bounding it from $H^y$ as provided by \eqref{basec}. Let $H_1$ be the other vertex of $T$, provided by \eqref{inout}, containing $C_0$. If $H_1$ lies in the component of $T$ containing $H^y$ then we obtain a contradiction, since $H_1$ is joined to $H^x$ by an edge of $T$. If not, then we can repeat the above procedure, replacing $H^x$ by $H_1$, to obtain $H_2$ and $C_2$ bounding $H_2$ from $H^y$. Continuing like this, we obtain a sequence $H_0=H^x,H_1,H_2, \ldots$ of vertices of $T$ all of which lie in the component of $H^x$, and a sequence of basic cycles $C_i$ bounding $H^y$ from $H_0, \ldots, H_i$. But after finitely many steps we must obtain some $C_i$ that is disjoint from $C_0$, for all $C_i$ have the same length and \g is \lf. This means that if $P$ is a shortest \pth{H^x}{H^y}\ in \G, then $C_i$ intersects $P$ at a vertex $z\not\in V(H^x)$ for  $C_i$ bounds $H^y$ from $H^x$. But then the distance from $H_i$ to $H^y$ is shorter than the distance from $H^x$ to $H^y$, which contradicts the choice of the pair $H^x, H^y$. This contradiction proves that $T$ is a tree as claimed.

We can now apply \Tr{ccfdec}. 
%satisfies \ref{cci} of \Tr{ccfdec}. It also satisfies \ref{ccii} because our $H_i$ are distinct by  \eqref{striii}. Thus we can apply  \Tr{ccfdec}.
As in earlier sections, we choose \fe\ $H^x$ the generating set $\cf_x$ induced by the presentation of $\Gam_2$ we obtained above. This easily yields the following presentations: 
%$p_n$ lies in a face \cf\ of $H_j$. The boundary of \cf\ is a \mpsc\ $C$ of \G.
\begin{itemize}
\item  $G \isom Cay\left<b,c,d\mid {b^2}, c^2, d^2, (bcd)^k; (bcdZ)^n \right>$,  $k\geq 3$, $n\geq 2$ (faces of size $3k\geq 9$);
\item  $G \isom Cay\left<b,c,d\mid {b^2}, c^2, d^2, (bcdZ)^n \right>$, $n\geq 2$ (no finite faces).
\end{itemize}

Indeed, recall that either $G_2 \isom Cay \left<a,z\mid z^2, a^k, (az)^n\right>$, $k\geq 3$, $n\geq 2$, which is the case if $a$ has finite order, or  $a$ has infinite order and  $G_2 \isom Cay \left<a,z\mid z^2, (az)^n\right>$. By the above discussion, and replacing $bcd$ back for $a$,  we then obtain 
$G \isom Cay\left<b,c,d\mid {b^2}, c^2, d^2, Z^2, (bcd)^k, (bcdZ)^n \right>$, $k\geq 3$, $n\geq 2$ in the first case and 
$G \isom Cay\left<b,c,d\mid {b^2}, c^2, d^2, Z^2, (bcdZ)^n \right>$, $n\geq 2$ in the second. In both cases we can omit the relation $Z^2$ though because, as mentioned earlier, \eqref{prefix} implies that $Z$ is a symmetric word and so $Z^2$ can be deduced from the relations ${b^2}, c^2, d^2$. Our presentations are planar by \eqref{LII5}.

Note that these presentations coincide with \ref{IIcii} and \ref{IIciii} of \Tr{TIIc}. Presentation \ref{IIci} corresponds to the case when \g has hexagonal faces and was obtained in our earlier discussion for that case.
This completes the proof of the forward implication of \Tr{TIIc}.

\subsubsection*{The converse implication} \label{convIIc}

It remains to prove the converse implication of \Tr{TIIc} for presentations of type \ref{IIcii} and \ref{IIciii}: given a non-crossing pattern $\cp$ and a $k\in \{3,4,\ldots \} \cup \sgl{\infty}$, we have to show that the \Cg\  $G= Cay\left<b,c,d\mid {b^2}, c^2, d^2, (bcd)^k; \cp \right>$ has the desired properties.

As in the forward implication, we can prove the existence of a dominant and a weak colour in \cp; this time, instead of using \Cr{CII3p} like in the proof of \eqref{LII4} we can just use the fact that $cp$ is non-crossing by assumption. The existence of the word $Z$ can be proved as above too using this assumption. Thus, similarly to, \eqref{azaz}, we can prove that 
\labtequ{azaz2}{$\cp$  is of the form $(AZ)^n$}
up to rotation and inversion, where again $A=bcd$ and $Z$ is a symmetric word. This means that every cycle $C$ `induced' by $\cp$ has a \soc\ in the sense of Definition~\ref{defsoc}. A \defi{pending edge} of $C$ is an edge incident with but not contained in $C$. Similarly to \eqref{socside} we have:
\labtequ{pendside}{\Fe\ cycle $C$ complying with the pattern $\cp$, and every side $A$ of $C$, \ti\ a \soc\ $S$ of $C$ \st\ all pending edges of $C$ incident with $S$ lie in $A$.}
Indeed, if $s,t$ are subsequent elements of $S$ whose pending edges lie in distinct sides of $C$, then it is easy to construct a rotation of $C$ containing the \pth{s}{t}\ `induced' by $A$ or $Z$.

If $k<\infty$, i.e.\ if our given presentation is of type \ref{IIcii}, we apply the converse of \Tr{LG2} \ref{oi} to obtain a (finite or 1-ended) \tcon\ planar \Cg\ $G_2$ of the group $\Gam_2= \left<a,z\mid z^2, a^{k}, (az)^n\right>$, each face of which is induced by $a^{k}$ or $(az)^n$. If $k=\infty$,  i.e.\ if we are in type \ref{IIciii}, we apply the converse of \Tr{main2} \ref{i} to obtain the planar \iicon\ \Cg\ $G_2$ of the group $\Gam_2= \left<a,z\mid z^2, (az)^n\right>$. The embedding of the latter graph is not unique; we choose an embedding in which  each cycle induced by $(az)^n$ bounds a face. This graph, and the desired embedding, can be obtained from a 3-regular tree by replacing every vertex with a cycle of length $2n$ alternating in the colours $a,z$, and letting the cycles replacing two adjacent vertices of the tree share a single $a$-edge. Note that $G_2$ has infinite faces too, bounded by $a$-coloured \dray s.

In both cases, replace each $a$-edge of $G_2$ by a path of length three with edges coloured $bcd$ respecting the directions of the $a$-edges, and replace each $z$ edge of $G_2$ by a path along which the colours of the edges read as the letters in $Z$. Let $G'_2$ be the resulting edge-coloured graph. Notice the similarity with the $G'_2$ used for the forward implication of \Tr{TIIc}. 

We will use similar ideas as in the proof of the forward implication, except that instead of finding copies $H^x$ of $G'_2$ as subgraphs of a given graph \G, we now have to construct them from scratch. In the end we will have to check that their union is a \Cg, and that it has the desired properties.

Let $H_0:= G'_2$. As earlier, we call any cycle of $G'_2$ a \defi{basic cycle} if the corresponding cycle of $G_2$ was induced by $(az)^n$. We will construct a sequence \seq{H} by inductively glueing a copy of $G'_2$ in each face of $H_{i-1}$ bounded by a basic cycle. In order to be able to do so we will need the following assertion, which is reminiscent of \eqref{inout}.
\labtequ{inoutc}{\Fe\ cycle $C$ complying with the pattern $\cp$ and every side $A$ of $C$, \ti\ a \soc\ $S$ of $C$ \st\ all pending edges of $C$ incident with $S$ lie in $A$.}
We will imitate the proof of \eqref{inout}, except that we now have no underlying graph \g in which we can look for $C'$, and so we will have to make do with \rota s of $C$. Note that by \eqref{azaz2} \ti\ at least one \soc\ $T$ of $C$, so we can assume that the pending edges of $C$ incident with $S$ lie in the wrong side $B\neq A$ of $C$. We may also assume that every vertex $x$ of $C$ is incident with a pending edge, for otherwise we may attach such an edge, embedding it so that $x$ has the right spin and giving it the colour missing from $x$, without affecting the compliance of $C$ with \cp. So let $xy$ be a pending edge of $C$ \st\ $x\in V(C)$ and $y$ lies in $A$, which exists by \ref{ncpi} of Definition \ref{dncp}. Let $t$ be a vertex in $T$ such that the colours of the two edges of $C$ incident with $t$ are the same as the colours of the two edges of $C$ incident with $x$ (one of which colours must be the dominant colour $c$). To see that such a $t$ exists, note that $T$ has by definition two vertices joined by a $bcd$ subpath of $C$, and as $c$ is the dominant colour, one of these vertices is incident with a $b$ and a $c$ edge and the other is incident with a $d$ and a $c$ edge of $C$. Let $D$ be the \rot{t}{x}\ of $C$, and recall that $D$ does not cross $C$. 

\Btco\ $D$, $x$ is a member of a \soc\ of $D$. Thus we can ask for a maximal subpath $P$ of $C$ that is  contained in some \rota\ $D$ of $C$ \st\ some \soc\ member $x$ of $D$ lies in $C$ and has a pending edge in $A$. As in the proof of \eqref{inout} (recall \fig{finout}) we claim that $P=C$. Suppose, to the contrary, that $P$ has distinct endvertices $v,w\in V(C)$. By the previous arguments, we can construct a \rota\ $D'$ of $C$ having $v$ as a \soc\ vertex. Note that $D'$ contains the edge $f\in E(C)$ incident with $v$ that does not lie in $P$. 

Let $P'$ be the intersection of $C$ and $D'$. By the maximality of $P$, and as $P'$ contains the edge $f\not\in E(P)$, we have $P\not\subseteq P'$, and so there is an interior vertex $z$ of $P$ at which $D'$ leaves $C$. It follows that $D'$ crosses $D$, as can be easily seen by looking at the edges incident with the endpoints of the path $P\cap P' = zPv$. But as $D$ and $D'$ share some edges they are a \rota\ of each other, and this crossing contradicts the definition of \cp.  This proves that $P=C$ as claimed. Combined with \eqref{pendside}, this easily implies \eqref{inoutc} as $P=C$ has by definition a \soc\ member with a pending edge in $A$. 
\medskip

Note that a \soc\ as in \eqref{inoutc} is uniquely determined once we fix the side $A$: if there were two distinct such \socs\ $S,T$ on $C$ then it would be easy to find $s,s'\in S$ and $t,t'\in T$ with $s'=sz, t'=tz$ \st\ $s,t,s',t'$ appear in that order on $C$, but then the corresponding \rota s of $C$ would cross. This means that 
\labtequ{socs}{\fe\ cycle $C$ complying with $\cp$ \ta\ precisely two \socs\ on $C$, with pending edges on opposite sides. This \socs\ are disjoint.}
Back to our construction of $H_i$, let $B$ be a  basic cycle of $H_{0}$ one of the sides $\cf_B$ of which is a face of $H_{0}$. Note that all corners of $H_0$ have the same spin since this was the case in $G_2$. %Moreover, by the construction of $H_0$ the on 
Thus $B$ complies with \cp, and we may apply \eqref{socs} to obtain two distinct \socs\ of $B$. One of these \socs\ must be the set of corners of $B$ by the construction of $G'_2$; we will use the other \soc\ $S$ to extend $H_{0}$: 
build an isomorphic (respecting edge-colours) copy $H_B$ of $G'_2$ that has $B$ as a basic cycle too, and in which the vertices in $S$ are corners, embedding  $H_B$ in the closure of $\cf_B$. Let $H_i$ be the union of $H_{i-1}$ with all these graphs $H_B$, one for every facial basic cycle $B$ of $H_{0}$ as above. Note that every corner of $H_1$ has the same spin; we use this fact as an induction hypothesis in the construction of \seq{H}. The following steps $H_2,H_3, \ldots$ of this construction are similar to the last one: we consider all basic cycles $B$ bounding a face of $H_{i-1}$, and embed a copy of $G'_2$ in that face attaching it at the yet unused \soc\ of $B$. All basic cycles comply with \cp\ at every step by our induction hypothesis, and so we can always apply \eqref{socs}. Since the newly attached pending edges of any $B$ as above are put in the `right' side of $B$,  our induction hypothesis is preserved.  Note that a basic cycle $B$ considered in step $i$ will not be reconsidered in any subsequent step as it does not bound a face any more.

Let $G:= \bigcup_\iin H_i$. By construction \g is planar. To see that it is cubic, each vertex $x$ being incident with all three edge-colours $b,c,d$, note that if $x$ has an incident pending edge $e$ in $H_i$, then $e$ lies in a face bounded by a basic cycle $B_i$ of $H_i$. Thus in the next step, $e$ will either lie in $H_{B_i}\subseteq H_{i+1}$ or inside a face of some basic cycle $B_{i+1}\neq B_i$. It follows easily that $e$ lies in some $H_j$ for otherwise \ti\ an infinite sequence \seq{B}\ of nested distinct cycles of the same length with a common vertex $x$, which cannot be the case in a \lfg.

Furthermore, all vertices of \g have the same spin since this was the case in each $H_i$.

We claim that \g is a \Cg. This will follow from the following two assertions and Sabidussi's theorem. %Define a \defi{\ciso} of an edge-coloured graph to be an automorphism preserving edge colours. 

\labtequ{autoa}{For every two basic cycles $C$, $D$ of \G, any \socs\ $S,T$ of $C,D$ respectively, and any members $s\in S, t\in T$, \ti\ a \auto\ $g$ of \g \st\ $gC=D$ and $gS=T$ and $gs=t$.}
To begin with, it easy to construct a \ciso\ $g'$ from $C$ to $D$ with $g'S=T$ and $g's=t$. We will extend $g'$ to a \auto\ of the whole graph \G. By the construction of \g \ta\ copies $H_C$ and  $H_D$ of $G'_2$ in \g in which the members of $S$ and $T$ are corners. Indeed, if $C$ was constructed in step $i$, and $S$ has pending edges in the side of $C$ that is a face of $H_i$, then we explicitly constructed $H_C$ in the definition of $H_{i+1}$. If $S$ has pending edges in the other side, then we take $H_C$ to be the copy of $G'_2$ in the construction of $H_i$ containing $C$. The same goes to $D$. 

Now let $g_0$ be a \ciso\ from $H_C$ to $H_D$ that coincides with $g'$ on $C$. Such a \ciso\ exists because $H_C$ and  $H_D$ are both copies of $G'_2$, and any basic cycle of $G'_2$ can be `rotated' by the action of $\Gam_2$. 

Having defined $g_0$ we proceed inductively, in \oo\ steps, to define extensions $g_1 \subset g_2, \subset, \ldots$, as follows. In each step $i>0$ consider each new basic cycle $B$ in the domain of $g_{i-1}$, let $B':=g_{i-1}B$, and note that precisely one of the two \socs\ $S_B$ of $B$ ---provided by \eqref{socs}--- has the property that its pending edges lie in the domain of $g_{i-1}$. Let $S^B$ be the other \soc\ of $B$. Similarly, let $S^{B'}$ be the \soc\ of $B'$ whose pending edges do not lie in the range of $g_{i-1}$. Then let $g_i$ map $H_B$ to $H_{B'}$, extending $g_{i-1}$, where $H_B$ and $H_{B'}$ are defined as $H_C$ above, but this time \wrt\ the \socs\ $S^B$ and $S^{B'}$. 

Let $g:= \bigcup g_i: G \to G$. To see that $g$ is surjective notice the similarity in the definitions of $g_i$ and $H_i$. Injectivity follows easily by considering the embedding of \g and its subgraphs $H_B$. Thus $g$ is the desired \auto.
\medskip

It \ises\ that
\labtequ{autob}{\fe\ $v\in V(G)$  \ti\ a basic cycle $B$ of some $H_i$ \st\ $v$ is in a \soc\ of $B$.}
Indeed, let $i$ be the first index \st\ all three edges of $v$ lie in $H_i$. Then by the construction of $H_i$ there was a basic cycle $B$ that had $v$ as a \soc\ vertex.
\medskip

Combining \eqref{autoa} with \eqref{autob}, we see that \fe\ $v,w\in V(G)$ \ti\ a \auto\ of \g mapping $v$ to $w$. Sabidussi's theorem thus yields that \g is a \Cg\ as claimed.

It follows immediately from \Lr{Liicon} that \g is \iicon.

Finally, we claim that each pair of colours spans \dray s in \G, in other words, each of $bc, cd, db$ has infinite order. To see this, let  $P$ be a component  of \g spanned by two colours.  Note that \fe\ basic cycle $C$ of \G, if $P$ meets $C$ then $P$ crosses $C$; this is easy to see using the fact that all vertices have the same spin and every other edge of $C$ bears the dominant colour $c$. Pick a basic cycle $C_0$ meeting $P$, and let $x_1$ be the first vertex of $P$ after leaving $C_0$. Let $C_1$ be a basic cycle containing $x_1$,  and  let $x_2$ be the first vertex of $P$ after leaving $C_1$. Continue like this to define infinite sequences  $x_i$ and $C_i$. By our previous remark, each $C_i$ bounds $C_{i-1}$ from $C_{i+1}$; thus the $C_i$ are all distinct, and so $P$ is infinite proving our claim.

To sum up, starting with an arbitrary \ncp\ \cp\ and an arbitrary value $k\in \{3,4,\ldots\} \cup \sgl{\infty}$, we constructed a plane \Cg\ with the properties required by the  forward implication of \Tr{TIIc}. Thus we can now apply the forward implication, which yields that \g has one of the presentations of the assertion. As we constructed \g to have faces of size $3k$, and \psc s induced by \cp, it follows that this presentation is indeed the intended one. This completes the proof of the backward implication of \Tr{TIIc}. It only remains to prove the last statement about $\kap(G)$.

\subsubsection*{The connectivity of \G}

In this section we determine the connectivity of each of the graphs of \Tr{TIIc} of type \ref{IIcii} or \ref{IIciii}; we have already seen that those of type \ref{IIci} are \tcon. The most interesting result is that such a graph can be \tcon\ even if its faces have infinite size.

We will say that our \ncp\ \cp\ is \defi{regular} if it is, up to rotation, of the form $\cp\isom (d c(bc)^{m})^n$. Note that in this case we have
\labtequ{cbcm}{$Z=c(bc)^{m-1}$,}
by \eqref{azaz} and the definition of $A$.

We begin with a basic observation that will be useful later.
\labtequ{nod}{$Z$ contains the letter $d$ \iff\ \cp\ is not regular.}
Indeed, if \cp\ is regular then any $d$-edge $e$ uniquely determines a cycle  through $e$ induced by \cp. As $Z$ induces the intersection of two distinct cycles  induced by \cp, the forward implication follows.

For the backward implication, suppose now that \cp\ is not regular, which means that \cp\ has two 2-coloured $bc$ `subpaths' $P,Q$ of distinct lengths surrounded by $d$-edges. Now given a $d$-edge $e=uv$ of \G, we can let \cp\ induce two cycles $C,D$ of \g by starting at $u$, say, and reading \cp\ once starting at the beginning of $P$ and once at the beginning of $Q$. Then $C,D$ are distinct, and they both contain $e$. The definition of $Z$ and \eqref{prefix} now immediately imply that $Z$ contains the letter $d$ as claimed.
\medskip
%a subword\sss $W$ of the form $d b(cb)^{m-1} d b(cb)^{m} $. 

Let $N$ be the size of the faces of \G, and recall that $N=3k$ is finite if \g is of type \ref{IIcii} and it is infinite if \g is of type \ref{IIciii}.
The main result of this subsection is
\begin{lemma} \label{fiNtcon}
If $N<\infty$ or \cp\ is not regular then \g is \tcon.
\end{lemma}
\begin{proof}
 
\comment{ 
		We will prove the \tcon ness of \g by proving that any two adjacent vertices can be joined by three independent paths and applying \Lr{Lk1}.

Let $e=uv$ be an edge of \g bearing one of the non-dominant colours $d$ or $b$. Let $P$ be a $Z$-path containing $e$; such a path exists by \eqref{nod} or the fact that $d$ is the weak colour and so any $d$-edge of any basic cycle is followed by a $c$ and then a $b$ edge on either side. Let $C,C'$ be the two basic cycles containing $P$. 

Recall that both edges incident with $e$ in $P$ must bear the dominant colour $c$, and so the two other edges incident with $e$ lie in distinct sides of $P$ since all vertices have the same spin; see \fig{f3paths}.
\showFig{f3paths}{Finding three independent \pths{u}{v}.}

		Consider the copies $H_u,H_v$ of $G'_2$ in which $u$ or $v$ is a corner. Let  $D_u$, respectively $D_v$, be the basic cycle in $D_u$ (resp.\ $D_v$) containing $e$ as well as the $c$ edge incident with $u$ (resp.\ $v$). 
}
%%%%%%%%%%%%%%%

If for every copy  $H_C$ of $G'_2$ as in the construction of \g from page~\pageref{socs}, all vertices in  $V(H_C) \sm \{x,y\}$ lie in a common component of $G - \{x,y\}$ for every pair of vertices $\{x,y\}\subset V(G)$,  then  $G$ is \tcon\ by \Lr{lkcon}.
Thus we can assume that $V(H_C) \sm \{x,y\}$ is disconnected in $G - \{x,y\}$ for some $H_C$ and some pair of vertices $\{x,y\}$. This implies in particular that $x,y\in V(H_C)$, for $H_C$ is \iicon.

Let us first consider the case where both $x,y$ are corners of $H_C$. Then $N$ cannot be finite, because in this case $G_2$ in \tcon, and so no two corners of its subdivision $G'_2 \isom H_C$ can disconnect it. Thus we may assume that $N=\infty$, and so $G_2 = Cay \left<a,z\mid z^2, (az)^n\right>$. We claim that $x,y$ must lie on a common basic cycle of $H_C$. For if not, then the vertices of each basic cycle remain connected after removing $x,y$. Moreover, as $H_C$ is the union of basic cycles, of which any two incident ones $D,D'$ have two corners in common, $x,y$ cannot disconnect $D$ from $D'$ unless $x,y \subset V(D)$. But if $x,y$ leaves every basic cycle connected, and every two adjacent ones are connected to each other, then it leaves their union $H_C$ connected. Thus we can indeed assume that $x,y$ lie in $V(D)$ for some basic cycle $D$ of $H_C$, and since they are corners of $H_C$ they lie in a common \soc\ $S$ of $D$. 

It also follows from the previous argument that either $x,y$ are the common corners of two basic cycles $D,D'$, in which case they are joined by a subpath of $D$ induced by $Z$, or both components $R,Q$ of $D - \{x,y\}$ contain a vertex of $S$, in other words, $x,y$ are not neighbouring \soc\ vertices. In both cases, 
\labtequ{zint}{each of $R,Q$ contains the interior of at least one $Z$-path with endvertices in $S$.}
 The following assertion will now easily imply the existence of an \pth{R}{Q}\ in the copy $H_D$ of $G'_2$ containing the dual \soc\ $\overline{S}$ of $S$, provided by \eqref{socs}, as corners.
\labtequ{zadual}{If \cp\ is not regular then the interior of every $Z$-path in some \soc\ $S$ contains an $A$-path between vertices of the dual \soc.}
To prove \eqref{zadual}, note that if it fails then for some $u,v \in S$ and the \pth{u}{v}\ $P$ induced by $Z$, any vertex $z\in \overline{S} \cap P$ is at distance less that four from one of $u,v$ for $A$ has length three. As every second edge of $P$ bears the dominant colour $c$, this distance must be precisely two, so assume that $d(u,z)=2$ (\fig{fzadual}). Easily, $P$ must start and end with the dominant colour $c$ by the definition of $Z$. Similarly, the second edge of  $P$ cannot bear the weak colour $d$, for no  basic cycle has a $dcd$ subpath by \eqref{LII6}. Thus $z= u c b$. Consider the two basic cycles $C,C'$ containing $P$, as well as the basic cycle $D$ containing the $Z$-path starting at $z$. 
\showFig{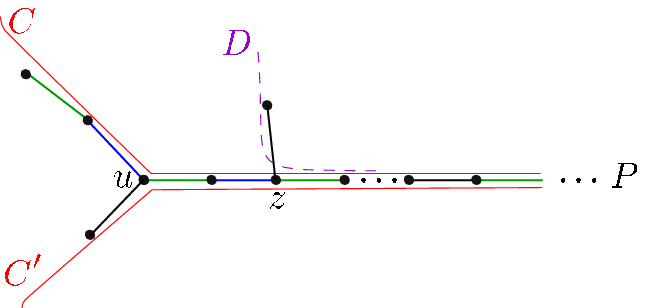}{The cycles  $C,C'$ and $D$ in the proof of \eqref{zadual}.}

Now recall that $P$ contains a $d$-coloured edge by \eqref{nod}, and so we can let $e$ be the $d$-coloured edge of $P$ that is closest to $u$. Let $B,B',B''$ be the maximal $bc$ 2-coloured subpaths of $C,C'$ and $D$ respectively containing $z$. Note that both  $B,B'$ start at $e$ by the definition of the latter, and precisely one of them stops at $u$. Thus $|B|\neq |B'|$. Note also that $B''$ starts at $z$ and cannot go past $e$, so that $|B''|< |B|,|B'|$. This means that $|B|,|B'|,|B''|$ are pairwise distinct, and so one of them contains at least 2 more $b$ edges than some other. But this contradicts \eqref{LII7}. This contradiction proves \eqref{zadual}.
\medskip

Let $T$ be the $Z$-path incident with $x$, say, and assume \obda\ that $T\subseteq R \cup \sgl{x}$, which we may by \eqref{zint}. Then  \eqref{zadual} yields a pair of elements $v,w$ of $\overline{S}$ that both lie in the interior of $T$. Now assume $v$ is closer to $x$ than $w$ on $T$, and let $v'= vZ$ be the element of $\overline{S}$ joined to $v$ by a $Z$-path $T'\subset D$. Note that by the choice of $v'$ and \eqref{zint} we must have $v'\in Q$. This means that the copy $H_D$ of $G'_2$ containing the dual \soc\ $\overline{S}$ as corners contains a \pth{P}{Q}\ $L$ in $G - \{x,y\}$: consider the basic cycle $D'$ of $H_D$ containing $T'$ that is not $D$ itself, and let $L= D' - \kreis{T'}$ be the other ${v}$-${v'}$~subarc of that basic cycle. But this path $L$ easily contradicts our assumption that $H_C$ meets two distinct components of $G - \{x,y\}$.

This contradiction implies that  $x,y$ cannot both be corners of $H_C$, and so we may assume that $y$, say, is not a corner of $H_C$ from now on. %We may assume however that $x$ is a corner of $H_C$: for if not, then we consider some other copy $H_{C'}$ of which $x$ is a corner. If $y\not\in V(H_{C'})$ then the neighbourhood $N(x)$ of $x$ is connected in $H_{C'} - \{x,y\}$, which means that $y$ separates \g by itself, contradicting the \iicon ness of \G. Thus $y$ must lie in $H_{C'}$, and so we could have chosen $H_{C'}$ instead of $H_C$.

%By arguments similar to the ones of the previous case we can see that  
Note that $x,y$ cannot disconnect any two corners of $H_C$ from each other because, as the reader will easily check,  $G'_2$ cannot be disconnected by removing a vertex and an edge. This implies that $x,y$ lie on a common path $B$ of $H_C$ induced by $Z$ or $A$, disconnecting part of  $B$ from the rest of  $H_C$. Let 
 $D$ be a basic cycle of $H_C$ containing $B$. Note that $D - \{x,y\}$ has two components $R,Q$ one of which, $R$ say, is a proper subpath of $B$. 

Now consider the other copy $H_D \supset D$ of $G'_2$ provided by \eqref{inout}. If $H_D$ has a corner in $R$, then $H_D$ contains an \pth{R}{Q} $L$ as above, and such a path easily implies that $H_C$ cannot meet two distinct components of $G - \{x,y\}$ contrary to our assumptions. If, on the other hand, $H_D$ has no corner in $R$, then  $H_D$ contains a further basic cycle $D_1\neq D$ \st\ $\{x,y\}\cup R$ is contained in some path $B_1$ of $D_1$ induced by $Z$ or $A$. We claim that $\{x,y\}$ must also disconnect $R$ from $H_D - R$ in \G. Indeed, if it does not, then there is a path from $R$ to $(D - R)\subset H_D$ in $G -  \{x,y\}$, and this path contradicts our assumption that $\{x,y\}$ disconnects $R$ from $H_C \supset D$ in $G -  \{x,y\}$. This proves our claim.

Now repeating these arguments on $D_1$ and the corresponding copy of $G'_2$ instead of $D$ and $H_C$, and iterating, we either obtain a contradiction after finitely many such steps, or an infinite sequence $D_i$ of distinct basic cycles containing $R$. But as \g is \lf, such a sequence cannot exist and we have a contradiction in any case. 

%proving that some copy  $H_{C'}$ of $G'_2$ contains both $x,y$ but $V(H_{C'}) \sm \{x,y\}$ is contained in a component of $G - \{x,y\}$.

%\soc\ $S$ of $D$ that is dual to the one containing $x$. If 
\end{proof}

\Lr{fiNtcon} implies that once \cp\ is non-regular, \g is \tcon\ even if all its faces have infinite boundary: 
\begin{corollary} \label{monsters}
\Fe\ non-regular \ncp\ \cp, the graph\\ $Cay\left<b,c,d\mid {b^2}, c^2, d^2; \cp \right>$ is \tcon\ and has no finite face boundary.
\end{corollary}
Since non-regular \ncp s do exist, take for example $ (d\ cbc\ d\ cbcbc)^n$, we obtain examples of planar \tcon\ \Cg s no face of which is bounded by a cycle, which is interesting in view of our discussion of \Sr{intBW}.

\medskip
Our next result is that if \cp\ is regular and the faces have infinite size then \g is not \tcon. This means that \Lr{fiNtcon} is best possible.
\begin{lemma} \label{Lsimple}
Let $G = Cay\left<b,c,d\mid {b^2}, c^2, d^2; (b(cb)^m d)^n \right>$, $m\geq 1$, $n\geq 2$. Then $\kappa(G)=2$. Moreover, \g has a \sepe\ (coloured $c$) \iff\ $m=1$.
\end{lemma}
\begin{proof}
Recall that \g is \iicon\ by \Lr{Liicon}. Let us check that \g is not \tcon.

In our proof of the converse implication of \Tr{TIIc} we constructed \g as the union of the sequence of its subgraphs $H_i$, and we will base our proof on that sequence. Recall that when $N=\infty$, $G'_2=H_0$ was obtained from the \Cg\ $G_2 = Cay \left<a,z\mid z^2, (az)^n\right>$ after replacing each $a$-edge by a $bcd$-path and replacing each $z$ edge by a path induced by $Z= c(bc)^{m-1}$, where we used \eqref{cbcm}. Consider a basic cycle $C$ of $H_0$, and two $d$-edges $e,f$ of $C$. We claim that these two edges separate \G.

To see this, let $P,Q$ be the two components of $C - \{e,f\}$. Note that no basic cycle $D$ of $H_0$ meets both $P$ and $Q$, which can be seen by applying \eqref{nod} to the pair $C,D$. This means that $\{e,f\}$ separates $H_0$, as the latter is the union of its basic cycles. Similarly, $\{e,f\}$ separates $H_C$, the copy of $G'_2$ that we glued along $C$ in the construction of $H_1$, into two components containing $P$ and $Q$ respectively. It now follows easily that $\{e,f\}$ separates $G$ into two components containing $P$ and $Q$ respectively, as \g can be obtained from $H_0 \cup H_C$ by inductively glueing copies of $G'_2$ along some basic cycle, and so none of these  copies meets both $P$ and $Q$. Now choosing one endvertex from each of $e,f$ yields a separator of \G, which means that \g is not \tcon.

It remains to prove that \g has a \sepe\ (coloured $c$) \iff\ $m=1$. Firstly, it \ises\ that, in every case, no edge coloured $b$ or $d$ is a \sepe: for if $uv$ is such an edge, then consider the $Z$-path $P$ starting at $u$, and the two basic cycles $C,C'$ containing $P$ (\fig{fzadual}). Note that three of the four vertices adgecent with $uv$ lie in $C \cup C'$ and are thus in a common component $K$ of $G - \{u,v\}$. Moreover, the fourth vertex $x$, which is adjacent with $v$, is connected to the neighbours of $u$ by a path that avoids $u$ since \g is \iicon. Thus, $x$ also lies in $K$. This proves that $uv$ is not a \sepe\ as it fails to disconnect its neighbourhood. Let us now consider the case when $uv$ is a $c$-edge instead. If $m>1$, then by \eqref{cbcm} the path $P$ defined as above has length greater than 1, and the we can apply the same arguments to show that $uv$ is not a \sepe. If $m=1$ however, we have $Z=c$, and so the endvertices of any $c$ edge separates in $G'_2$ the two basic cycles containing it. This easily implies that any $c$ edge is a \sepe\ of \G.
\end{proof}

\Lr{Lsimple} combined with  \Lr{fiNtcon} determine $\kap(G)$ for each of the graphs \g as in \Tr{TIIc} \ref{IIcii}, \ref{IIciii}.

\subsubsection{All edges reverse spin} \label{secIIdum}

If \g is \tcon\ and all its edges reverse spin, then by \Lr{LIIab1} \g cannot be multi-ended. Thus we can proceed with the next case.

\subsubsection{Mixed spin behaviour} \label{secIId}

The last case we have to consider is that of a planar, multi-ended \tcon\ \Cg\ $G \isom Cay\left<b,c,d\mid b^2, c^2, d^2, \ldots \right>$ with no 2-coloured cycles, with both spin-preserving and spin-reversing colours. 
The main result of this section, characterizing these graphs, is the following.
\begin{theorem} \label{TIId}
Let $G= Cay\left<b,c,d\mid b^2, c^2, d^2, \ldots \right>$ be a planar \tcon\ multi-ended \Cg\ with both spin-preserving and spin-reversing edges. Suppose that all three of $bc, cd$ and $db$ have infinite order. Then precisely one of the following is the case.
\begin{enumerate}\addtolength{\itemsep}{-0.5\baselineskip}
\item \label{IId2i} $G \isom Cay\left<b,c,d\mid {b^2}, c^2, d^2, (bdb cdc)^k; (c(bc)^nd)^{2m}\right>, k\geq 2, n,m\geq 1,\ n+m\geq 3;$ \note{(\gt\ is 1-ended)}
\item \label{IId2ii} $G \isom Cay\left<b,c,d\mid {b^2}, c^2, d^2, (bdb cdc)^q; (c(bc)^{n-1}d)^{2m}, (c(bc)^{n} d)^{2r} \right>, n,r,m,q\geq 2$;
\item \label{IId2iii} $\g \isom Cay \left< b,c,d\mid b^2, c^2, d^2;  (c(bc)^{n-1}d)^{2m}, (c(bc)^{n} d)^{2r} \right>$, $n,r,m\geq 2.$ \note{($\kap(\gt)=2$).} 
\end{enumerate}
All these presentations are planar.

Conversely, \fe\ $k,n,m$ in the specified domains the above presentation yields a \Cg\ as above.
\end{theorem}

The rest of this section is devoted to the proof of \Tr{TIId}, so let us fix a \Cg\ \g as in its assertion. Our analysis will be similar to that of \Sr{secIIc}.

It follows from \Lr{LIIab1} that only one of the colours $b,c,d$ can be spin-reversing. Assume from now on that this colour is $d$. This implies that
\labtequ{bdbcdc}{every facial walk is of the form $\ldots bdb cdc \ldots$.}

\subsubsection*{The \mpsc s}

As in the previous section our analysis will be based on the \mpsc s of \G.

\begin{proposition} \label{LIIAd}
\Fe\ \mpsc\ $C$ of \G, there is a colour $a\in \{bcd\}$, \st\  every other edge of $C$ is coloured $a$. (In particular, $|C|$ is even.)
\end{proposition}
As earlier, we will call this colour $a$ the \defi{dominant} colour of $C$.
\begin{proof}
We begin by showing that
\labtequ{Ld1}{no facial subpath of a \mpsc\ contains two $d$-edges.}
For let $P$ be a facial subpath of an \mpsc\ chosen so as to maximize the number $p_d$ of $d$-edges in $P$, and suppose that  $p_d\geq 2$. Let $e=uv$ be the last $d$-edge of $P$, and let $C$ be a \mpsc\ containing $P$. Then the \auto\ of \g exchanging $u$ and $v$ maps $C$ to a \mpsc\ $D$, and it is not hard to see, using \eqref{bdbcdc}, that $D$ crosses $C$. However, as in the proof of \Cr{CII3p}, such a crossing gives rise to a new \mpsc\ $C'$ containing a facial subpath $P'\supseteq P$ with more than $p_d$ $d$-edges, contradicting the choice of $P$. This proves \eqref{Ld1}.

Now let $C$ be any \mpsc\ of \G. Note that as $C$ contains all three colours, it must contain two adjacent edges coloured $b$ and $c$ unless every other edge of $C$ is coloured $d$. In the latter case our assertion is already proved, so assume from now one that the former is the case. So let $P$ be a maximal $bc$-coloured subpath of $C$ with $||P||\geq 2$. If $||P||= 2$ then both edges $e,f$ of $C$ incident with $P$ are coloured $d$, and so $e \cup P \cup f$ is, by \eqref{bdbcdc}, a facial subpath of $C$ containing two $d$-edges contradicting \eqref{Ld1}. Thus $||P||\geq 3$. 

Similarly to the proof of \eqref{LIb1} we can now prove that
\labtequ{Ldodd}{every maximal $bc$-coloured subpath of a \mpsc\ of \g has odd length.}

It follows that $||P||$ is odd, and so $P$ starts and ends with the same colour, $c$ say. Let $x$ be the last vertex of $P$, let $e$ be the $d$-edge of $C$ following $P$, and let $f$ be the edge of $C$ after that. We claim that the colour of $f$ is $c$. For if it is $b$, then the \auto\ exchanging the endvertices of $e$ maps $C$ to a cycle $C'$ crossing $C$. Note that the maximal $bc$-coloured subpath $Q$ of $C'$ starting at $x$ has odd length by \eqref{Ldodd}. Now let $C''$ be the translate of $C'$ obtained by mapping $x$ to $w:=xcb$ (\fig{fd3}). Note that $C''$ also crosses $C$, and by \eqref{CII3m} we can replace a subpath of $C$ containing $e$ by a subpath of $C''$ not containing $e$ to obtain a new \mpsc\ $D$. As each of $Q$ and $P$ had odd length, $D$ has a maximal $bc$-coloured subpath of even length contradicting \eqref{Ldodd}. This contradiction proves our claim that the colour of $f$ is $c$.
\showFig{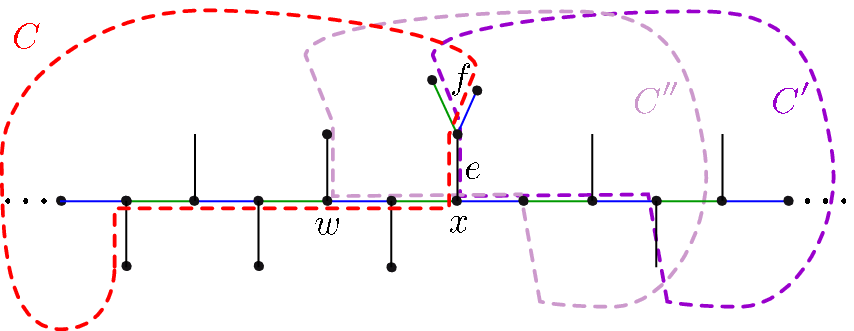}{The construction of $D$ in the proof of \Prr{LIIAd}.}

Our next assertion will allow us to determine the colour of the edge of $C$ following $f$.
\labtequ{Ld2}{No \mpsc\ of \g has a 2-coloured subpath of length at least four containing $d$-edges.}
To prove this, let $Q$ be a 2-coloured subpath of a \mpsc\ containing $d$-edges that maximises $|Q|$ among all such paths, and suppose that $Q$ contains at least four edges. Let $D$ be a \mpsc\ containing $Q$. Then, exchanging the endvertices of one of the $d$-edges $e$ of $Q$ by a \auto\ of \g we obtain one of the two situations depicted in \fig{fd2}. In both cases though, \eqref{CII3m} implies the existence of a \mpsc\ containing a 2-coloured path that extends $Q$ contradicting its maximality. This contradiction proves \eqref{Ld2}.
\showFig{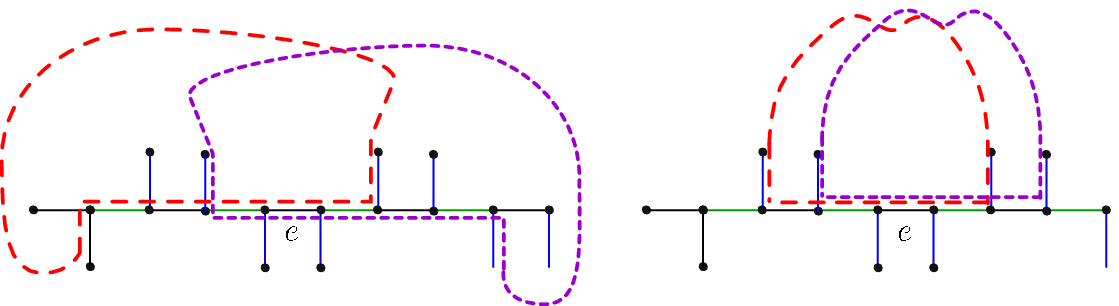}{The situation in the proof of \eqref{Ld2}.}

Now as $P$ ends with a $c$ edge, and $f$ is also coloured $c$, \eqref{Ld2} implies that the edge of $C$ following $f$ is coloured $b$. Let $P_1$ be the maximal $bc$-coloured subpath of $C$ containing $f$. Then $||P_1||\geq 2$, and by \eqref{Ldodd} $||P_1||$ is odd. Thus, $P_1$ starts and ends with a $c$ edge, as was the case for $P$. This allows us to apply the same arguments to $P_1$ to show that it is followed by a $d$-edge and another odd $bc$-coloured path $P_2$, and so on. This proves that every other edge of $C$ is coloured $c$ as desired.
\end{proof}

As in \Sr{secIIc} it is important to know whether \g has a hexagonal face. Fortunately, it does not:

\begin{proposition} \label{nohex}
\g has no hexagonal face.
\end{proposition}
\begin{proof}
Suppose, to the contrary, that \g does have a hexagonal face-boundary $H$. Let $C$ be a \mpsc\ of \G. We distinguish two cases according to the spin-behaviour of the dominant colour $a$ of $C$.

If $a$ is the spin-reversing colour $d$, then our approach is similar to that of \Sr{secTidy}. Let $R$ be  $bc$-double-ray incident with $H$. Considering the \auto s mapping $R$ to itself it is easy to check that every face boundary incident with $R$ must be a hexagon too, as it can be mapped to $H$. We can proceed as in \Sr{secTidy} to prove an assertion similar to \eqref{LIb8} there. For this, we superimpose a new dummy colour $o$ to each edge of \g coloured $b$ or $c$, and direct all edges bearing that colour `to the right', as in \fig{hex}. We can now repeat the proof of \eqref{LIb8} to show that $o$ appears in $C$ the same number of times in each direction, and use this fact to construct a new \mpsc\ $C'$ that `reads' $(odo^{-1}d)^n$. Note however that the way we directed the $o$ edges combined with the fact that $d$ reverses spin implies that such a cycle $C'$ contains only two of the colours $b,c,d$, contradicting our assumption that \g has no 2-coloured cycle.

%As in \Sr{secTidy}, we can now check that any two adjacent $bc$-double-rays form a `strip' of hexagons with the $d$-edges joining them as in \fig{hex}. Moreover, the existence of $C$ proves that \g contains only finitely many $bc$-double-rays, in other words, the subgroup spanned by $b,c$ has finite index in $\Gam(G)$.

If $a$ is one of the spin-preserving colours, $c$ say, then as $C$ must contain all three colours, it must have a $bcdc$ subpath. But such a subpath can be shortcut by a $db$ path since, by \eqref{bdbcdc}, $bdbcdc$ is the relation inducing the hexagonal face-boundary $H$. This contradicts the minimality of $C$.
\end{proof}

Using this fact we can now enrich our knowledge about the \mpsc s. Let $C$ be a \mpsc\ of \G, and let $\ct$ be the set of translates of $C$ by \auto s of $G$.
\labtequ{LII5d}{No two elements of $\ct$ cross.}
Suppose, to the contrary, two elements $D,D'$ of $\ct$ cross. We will now use the ideas of \Sr{secTidy}  to obtain a contradiction from Euler's formula \eqref{euler}. By \eqref{BF} there is a subpath 
$P$ of $D$ and a subpath $P'$ of $D'$, with common endvertices $u,v$,  \st\ $P \cup P'$ is a cycle $K$ bounding a region $B$ containing finitely many vertices. We repeat the construction of $H'$ of \Sr{secTidy}: let again $H$ be the finite plane subgraph of $G$ spanned by $K$ and all vertices in $B$. Let $H'$ be the graph obtained from two copies of $H$ by joining corresponding vertices of degree two by an edge, and note that $H'$ is cubic. Consider an embedding of $H'$ in the sphere \st\ the two copies of $H$ occupy two disjoint discs $D_1,D_2$, and the newly added edges and their incident faces lie in an annulus $Z$ that joins these discs. Since by \Prr{nohex} all faces within these discs have size greater than 6, contributing a negative curvature to \eqref{euler}, it suffices to show that the total curvature of the faces in $Z$ does not exceed 12 to obtain a contradiction to \eqref{euler}. The faces in $Z$ have even size by construction, so that it suffices to show that $Z$ cannot contain more than six 4-gons. The presence of spin-reversing edges makes this task slightly harder than in \Sr{secTidy}, and we need a new argument.

We are going to reduce $Z$ to a small auxiliary graph $Z'$ by performing operations that leave the curvature, as well as the spin-behaviour of the colours, invariant. We will end up with only few possibilities for $Z'$, in which it will be easy to count the 4-gons. The first of these operations suppresses a pair of spin-preserving edges on opposite sides of $Z$. Let $e$ be an edge of $Z$ coloured $b$ or $c$ which is an interior edge of $P$ or $P'$ \st\ both edges $f_1,f_2$ of $P$ or $P'$ incident with $e$ bear the other spin-preserving colour (this could be the dominant colour for example). Let $e',f_1',f_2'$ be the corresponding edges on the other side of $Z$. Note that because of the spin-behaviour, precisely one of the endvertices of $e$ is adjacent to its counterpart in $e'$ by a $d$-edge $g$. The situation is thus as in the left part of \fig{fdelb}. Now let us delete $e,e'$ and $g$, and identify $f_1$ with $f_2$ and $f_1'$ with $f_2'$ to obtain a new graph $Z^*$. Note that $Z^*$ has the same `curvature' as $Z$: indeed, $Z^*$ has one face less than $Z$, and precisely 6 edges less than $Z$ in the count of the sum of face-sizes of \eqref{euler} ($g$ is counted twice because it lies in two face boundaries). Thus, this operation is indeed neutral as far as Euler's formula is concerned.

We can perform a similar operation whenever $Z$, or any auxiliary graph $Z^*$ obtained after performing the above operation a number of times, has a subpath of the form $cdcdc$ or $bdbdb$ on its boundary: the right part of \fig{fdelb} shows how to remove this kind of subpath and its counterpart affecting neither the total curvature nor the spin-behaviour of the colours on the boundary of $Z$. 

We now distinguish two cases according to which colour is dominant in $C$ (and thus in any element of \ct).

If the dominant colour is a spin-preserving one, let us say $c$, then we can apply the operation of the left half of \fig{fdelb} repeatedly to eliminate from $Z$ each $b$ edge that was an interior edge of $P$ or $P'$. 
\showFig{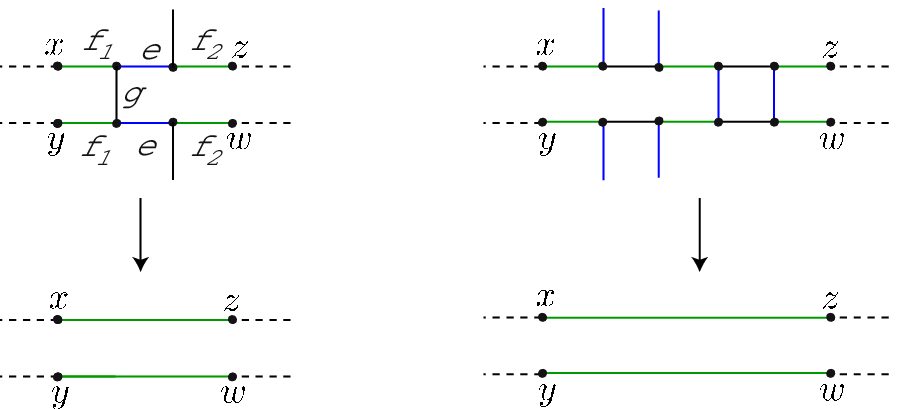}{Two operations that simplify $Z$ without changing its curvature.}
Then, we can apply the operation of the right half of \fig{fdelb} repeatedly to eliminate pairs of $d$-edges, thus leaving an auxiliary graph $Z'$ with very few edges: the vertices $u,v$ (recall that these where the common endpoints of $P$ and $P'$) and their incident $c$ edges split $Z'$ into two parts, each of which contains no $b$ edge that is not incident with $u$ or $v$, and contains at most one pair of $d$-edges not incident with $u,v$. In fact, it must contain precisely one pair of $d$-edges not incident with $u,v$, for any attempt to construct such a part without such $d$-edges leads to a contradiction to the spin-behaviour as displayed in \fig{fZnoparts}. It follows that each of those parts must be one of the three graphs in the upper row of \fig{fZparts}. 
\showFig{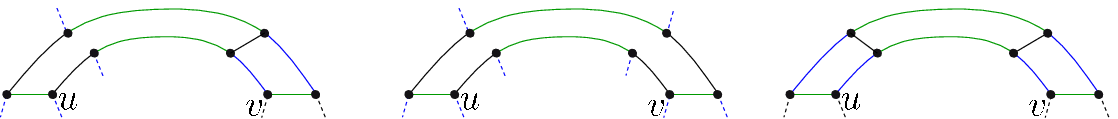}{Some failed attempts to create a building block for $Z'$. Note that the $c$ edges in the middle are not spin-preserving as they should. The dashed edges are not part of $Z'$ and are only a visual aid.}
\showFig{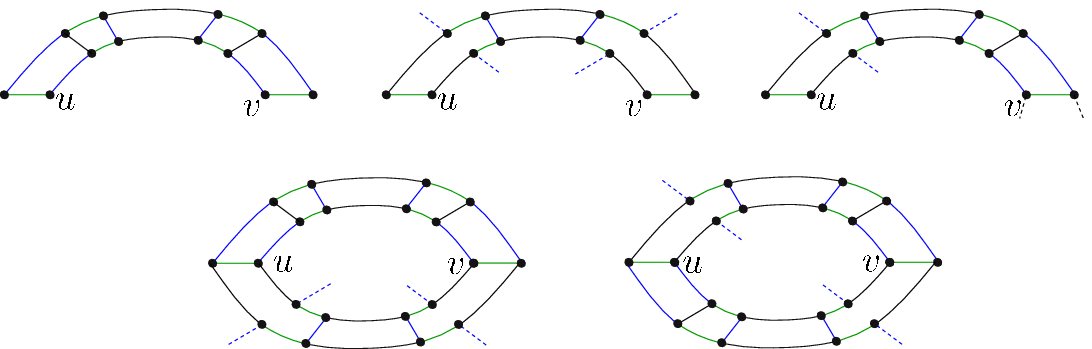}{The possible combinations for $Z'$.}
Now as $u$ and $v$ are incident with one edge of each colour, $Z'$ is either the union of the two leftmost graphs of the upper row of \fig{fZparts}, or the union of two copies of the rightmost one. Thus $Z'$ is one of the graphs of the bottom row. In both cases, it has precisely six 4-gons and two 6-gons, accounting for a contribution of 12 `curvature' units to Euler's formula \eqref{euler}. To sum up, the contribution of the faces of $H'$ that lie inside $Z$ to the left part of \eqref{euler} equals 12. The contribution  of the remaining faces of $H'$ is strictly negative, since each such face has size larger than 6. This contradicts \eqref{euler} in the case when the dominant colour is a spin-preserving one.

We now turn to the case when the dominant colour of $C$ is $d$. In this case we can still use the reducing operation in the right half of \fig{fdelb}. In addition, we can use two further operations, shown in \fig{fdelbdc}, that can recursively suppress any path of the form $cdbdc$, or $bdcdcdbdb$, and similarly with the roles of $b$ and $c$ interchanged. After performing such operations as often as possible, we are left with a graph $Z'$ which again we think of as the union of two parts joined at $u$ and $v$. Each of those parts has now at most three pairs of $d$-edges not incident with $u$ or $v$, for if it has more then one of the above operations can be applied. We are left with few possibilities, and an easy case study (for which \fig{fdelbdc} might still be helpful) shows that none of these two parts can contain more than three 4-gons, again leading to a contradiction of \eqref{euler}. This completes the proof of \eqref{LII5d}.
\showFig{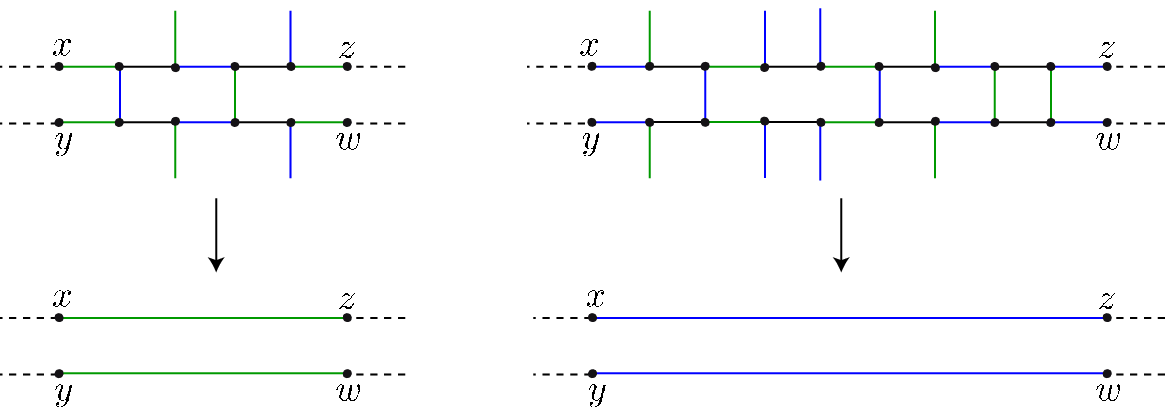}{Further reducing operations for the case that $d$ is the dominant colour of $C$.}
\medskip

We can now apply \eqref{LII5d} to exclude $d$ as a dominant colour:
\labtequ{Ld5}{The spin-reversing colour $d$ is not dominant in any \mpsc\ of \G.}
Indeed, if $d$ was the dominant colour of a \mpsc\ $C$, then $C$ would contain a $d$-edge $e$ incident with both a $b$ and a $c$ edge of $C$. But then the \auto\ of \g that exchanges the endvertices of $e$ would map $C$ to a \mpsc\ $C'$ that crosses $C$, contradicting \eqref{LII5d}.
\medskip

In fact, we can prove a bit more about dominant colours:
\labtequ{Ld1dom}{If $C,B$ are \mpsc s of \G\ then they have the same dominant colour.}
Indeed, if not, then by \eqref{Ld5} we can assume that the dominant colour of $C$ is $c$ and that of $B$ is $b$. Note that as all cycles must be 3-coloured, $C$ must have a subpath of the form $cdcbc$ and $B$ must have a subpath of the form $bdbcb$. Translating one of those paths to the other as in \fig{fd1dom} we obtain a crossing of two \mpsc s. By \eqref{CII3m} this implies the existence of a new \mpsc\ containing a subpath of the form $dbcd$. But this contradicts \Prr{LIIAd} and so \eqref{Ld1dom} is proved.
\medskip
\showFig{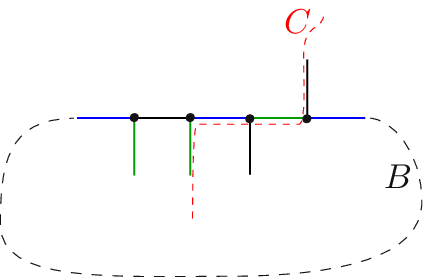}{The crossing cycles in the proof of \eqref{Ld1dom}.}

By \eqref{Ld5} and \eqref{Ld1dom} all \mpsc s of \g have the same, spin-preserving, dominant colour. Assume from now one that this colour is $c$. We can now describe the \mpsc\ precisely:
\labtequ{Ld6}{Every \mpsc\ of \g is induced by the word $(c(bc)^nd)^{2m}$ for fixed $n,m\geq 1$.}
Let $C$ be a \mpsc, and define a \defi{$bc$-interval} of $C$ to be a maximal subpath of $C$ not containing a $d$-edge. Note that \eqref{Ld6} is equivalent to saying that all $bc$-intervals of \mpsc s of \g have the same length $2n+1$. Suppose, to the contrary, that $C$ has a $d$-edge $e$ \st\ the two $bc$-intervals on either side of $e$ have different lengths. Then, observing the spin behaviour of the edges, it is easy to see that the \auto\ of \g that exchanges the endvertices of $e$ maps $C$ to a \mpsc\ that crosses $C$ (\fig{fd6}),
contradicting \eqref{LII5d}. In particular, $C$ does not contain a $dcd$ subpath, for such a subpath contains a $bc$-interval of length 1, and $C$ has some $bc$-interval of length at least 3 as it is 3-coloured. This proves that every \mpsc\ $C$ has the desired form $(c(bc)^nd)^m$, with $n\geq 1$, and it only remains to show that $n$ and $m$ cannot vary for a different \mpsc\ $D$. The fact that $n$ cannot vary can be proved with a similar argument, by considering a \auto\ that maps a $d$-edge of $D$ to a $d$-edge of $C$ (see \fig{fd6} again). It follows that $m$ cannot vary either, since all \mpsc s have the same length by definition. The fact that $C$ has an even number of $d$-edges, giving rise to the exponent $2m$, can easily be proved by observing the spin behaviour of the edges.

\showFig{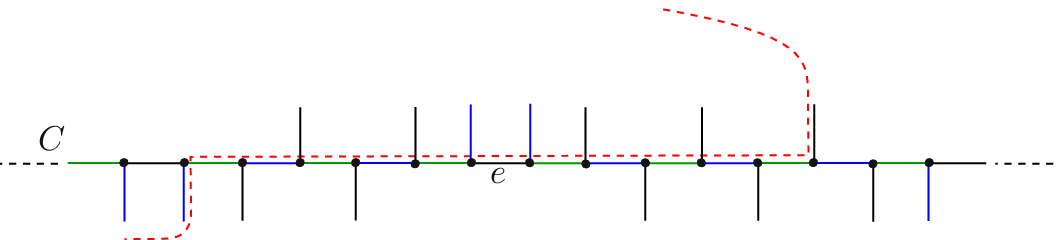}{Finding a crossing if two $bc$-intervals of $C$ have different lengths in the proof of \eqref{Ld6}.}

\subsubsection*{The subgroup $\Gam_2$ and the subgrpaph $G'_2$.}

Define the words $B:=bcdcb$ and $Z:=c(bc)^{n-1}$, where $n$ is supplied by \eqref{Ld6}. Note that every face boundary of \g is, by \eqref{bdbcdc}, of the form $\ldots dBdBdB \ldots$. Moreover, every \mpsc\ is, by \eqref{Ld6}, of the form $(BZdZ)^m$ (compare this with \eqref{azaz}).

Let $b^*$ be the element of $\Gam$ corresponding to the word $B$ and let $z$ be the element of $\Gam$ corresponding to the word $Z$. 
Let $\Gam_2$ be the subgroup of $\Gam$ generated by $\{b^*,z,d\}$. 

As in \eqref{delta} we still have:
\labtequ{deltad}{for every \mpsc\ $C$ of \g and every coset $\Delta$ of $\Gam_2$ in \G, at most one of the sides of $C$ contains elements of $\Delta$.}
To prove this, one can use similar arguments as in the proof of \eqref{delta}, however, there is an easier way: it is easy to check, just by observing the spin behaviour, that no path  induced by a word in the letters $b^*,z,d$ can cross a cycle $C$ induced by the word $(BZdZ)^m$, the word inducing the \mpsc s.
%f \eqref{deltad} is false, then there is a path $P$ of minimum length induced by a word in the letters $b^*,z,d$ whose endpoints lie in different sides of $C$. Then the first of $P$ must meet $C$. 
\medskip

Next, we check that 
\labtequ{Zd}{\fe\ path $P$ in \g induced by $Z$ \ta\ precisely two \mpsc s $C,D$ containing $P$. Moreover, $C \cap D=P$, that is, $C$ and $D$ have no common edge outside $P$.} 
The first part of this assertion is much easier to prove than the corresponding assertion \eqref{Z} in the previous section: it follows immediately from \eqref{Ld6}. The second part can be proved like \eqref{Z}: if $C,D$ have a common edge outside $P$, then using a subpath of each we can form a cycle $K$ shorter than $|C|$, and using Euler's formula as in \eqref{LII5d} we can prove that $K$ is \ps, a contradiction.
\medskip

The following assertion strengthens \eqref{deltad} and can be proved like \eqref{indep}. %implies that the \Cg\ of $\Gam_2$ corresponding to the generating set $\{b^*,z,d\}$ is embedded in \G\ similarly to earlier sections. 
A \defi{metaedge} is a path of \g induced by one of the words $B,Z,d$.
\labtequ{Ldind}{Any two metaedges of $\Gam_2$ are independent.}
% 
%Indeed, this follows easily from \eqref{deltad} and  \eqref{Zd}, for if two metaedges $e,f$ of $\Gam_2$ meet at an interior vertex of $e$ say, then this would imply that \gat\ meets the `wrong' side of one of the \mpsc s containing $e$. %, see \fig{fIId}. 

This means that, as in earlier sections, the \Cg\ $G_2$ of $\Gam_2$ \wrt\ the generating set $\{b^*,z,d\}$ has a \topem\ in \G: we can obtain $G_2$ from $G$ by substituting, for every two vertices $x,y$ of  $G_2$ that are adjacent by a  $b^*$ or $z$ edge, the \pth{x}{y}\ in \g  induced by $B,Z$ with an $x$-$y$ edge of the corresponding colour $b^*,z$. This yields indeed a  \topem\ of  $G_2$ in \G\ since by \eqref{Ldind} all these paths are independent. Starting with $G_2$ and replacing each $b^*$ or $z$ edge back by the corresponding path induced by $B$ or $Z$  we obtain a subdivision $G'_2\subseteq G$ of $G_2$ that will be useful later.

Note that every edge of \gt\ corresponds to an involution. Moreover, every $b^*$ or $d$-edge of \gt\ is spin-reversing while every $z$ edge is spin-preserving; this can be deduced from the spin behaviour of the original edges of \G. 

%FALSE: We claim that, unlike what was the case in \Sr{secIIc}, face boundaries must be finite:
\comment{
	\labtequ{finfad}{There is $k\geq 2$ \st\ every face boundary of \g is induced by the word $(bdb cdc)^k = (Bd)^k$.}
	Indeed, the fact that all face boundaries have the same length follows from \Tr{imrcb} and the fact that every vertex is incident with two spin-preserving edges. Thus it only remains to show that this length is finite.
}

\subsubsection*{The \plpr\ of \G}

As in earlier sections we will express \g as a union of copies of \gtp, and then apply \Tr{ccfdec} in order to deduce a presentation of \g from a presentation of \gt. 

In this section we define a \defi{\soc} as in \Dr{defsoc} of the previous section, except that we now base this definition on the word $(BZdZ)^m$, which induces the \mpsc s in the current case.

To begin with, we claim that 
\labtequ{socsd}{every \mpsc\ $C$ of \g contains precisely two distinct \socs.}
Indeed, recall that $C$ is induced by the word $(c(bc)^nd)^{2m}$ by \eqref{Ld6}. Note that \fe\ two `consecutive' $d$-edges $d_1,d_2$ on $C$, the two $b$ edges incident with $d_1$ lie in one side of $C$, while the $b$ edges incident with $d_2$ lie in the other side of $C$. Thus \ti\ a bipartition $\{D_1,D_2 \}$ of the set of $d$-edges of $C$ \st\ all $b$ edges incident with an element of $D_i$ lie in the same side of $C$. It is now straightforward to check that the endvertices of all the edges in each of the $D_i$ lie in a common \soc\ of $C$, and these two \socs\ are distinct for $i=1,2$.

This allows us to create a structure tree $T$ on the set of left \gat\ cosets in \Gam\ as we did in \Sr{secIIc}: join two such cosets with an edge, if the corresponding copies of \gtp\ share a \mpsc\ of \G. It follows from \eqref{deltad} that $T$ is acyclic, and from \eqref{socsd} that it is connected. Thus once more, we can apply \Tr{ccfdec}, with the $H_i$ being the vertices of $T$ and the $F_i$ being the \mpsc s of \g giving rise to the edges of $T$. This yields that given any presentation of \gat\ we can transform it into a presentation of \Gam\ by replacing any occurrence of the letters $b^*, z$ by the corresponding words $B=bcdcb$ and $Z=c(bc)^{n-1}$ and adding the involution relations $b^2, c^2$. 

So let us find a presentation of \gat. We distinguish three cases according to the connectivity and number of ends of \gt.
\medskip

{\bf Case I:} \gt\ is \tcon\ and finite or 1-ended.
In this case we apply \Tr{LG5} \ref{ziii} to \gt, which yields $\gt \isom Cay\left<b^*,z,d\mid {b^*}^2, z^2, d^2, (b^*d)^r, (b^*zdz)^m\right>$, $r\geq 2, m\geq 1$. Note that $b^*$ and $d$ are interchangeable in this presentation, while $z$, being the only spin-preserving colour, plays a special role.

Substituting $b^*, z$ as suggested above we obtain the following presentation for \Gam:
$$G \isom Cay\left<b,c,d\mid {b^2}, c^2, d^2, (c bdbcd)^k; (c(bc)^nd)^{2m}\right>, k\geq 2, n,m\geq 1.$$
\note{2m because we compactified the word}
This is a \plpr: the relation $(c bdbcd)^k$ corresponds to face-boundaries, and the last relation corresponds to \mpsc s. We will prove in the next subsection that $n+m\geq 3$ must hold in this case.
\medskip

{\bf Case II:} \gt\ is \tcon\ and multi-ended.
In this case we can apply \Lr{LIIab1}, which yields that \gt\ must have a 2-coloured cycle involving the only spin-preserving colour $z$. We have already characterized the graphs of this type: \Tr{TIId11} yields that $\gt \isom Cay\left<b^*,z,d\mid {b^*}^2, z^2, d^2, (b^* d)^q, (dzb^* z)^m;(zb^*)^{2r} \right>$, $r,m,q\geq 2$, or (exchanging $b^*$ and $d$ in the above presentation and rearranging)
$\gt \isom Cay\left<b^*,z,d\mid {b^*}^2, z^2, d^2; (b^* d)^q, (dzb^* z)^m;(zd)^{2r} \right>$, $r,m,q\geq 2$.

Substituting $b^*, z$ as above we obtain in the first case
$$G \isom Cay\left<b,c,d\mid {b^2}, c^2, d^2; (c bdbcd)^q; (c(bc)^nd)^{2m}, (c(bc)^{n+1} d)^{2r} \right> \text{ with } n\geq 1 \text{ and } r,m,q\geq 2,$$ 
where we used the fact that $(c(bc)^{n-1} bcdcb )^{2r}= (c(bc)^{n} dcb )^{2r} =   (c(bc)^{n+1} d)^{2r}$. In the second case we obtain
$$G \isom Cay\left<b,c,d\mid {b^2}, c^2, d^2; (c bdbcd)^q; (c(bc)^nd)^{2m}, (c(bc)^{n-1} d)^{2r} \right>, \text{ with } n,r,m,q\geq 2.$$
In the latter presentation we are demanding $n\geq 2$ because if $n=1$ then \g has 2-coloured cycles contrary to our assumption.

Note that these two presentations are the same, as can be seen by exchanging $m$ with $r$ and $n$ with $n+1$. Thus we omit the first one.  We have now obtained possibility \ref{IId2ii} of \Tr{TIId}.
\medskip

{\bf Case III:} \gt\ is not \tcon.

It follows from \Lr{Liicon} that \gt\ must be \iicon. Thus, in this case \gt\ is one of the graphs of \Tr{main2}. Since it has three generators, it has to belong to one of the types \ref{iv}--\ref{ix} of that theorem. We will be able to eliminate most of these types as a possibility for \gt, leaving only type \ref{v} as a possibility.

It is made clear in \cite{cay2con} that for every graph of type \ref{iv}, all edges participating in the 4-cycles induced by the relation $(bc)^2$ must preserve spin in any embedding. But our \gt\ has an embedding in which two of the colours reverse spin, and so \gt\ cannot be of this type.

For a graph of type \ref{vi}, the colour participating in both relations has the property that any edge $e$ of that colour is a \sepe, and it separates the graph in two components each of which sends two edges of the same colour to $e$. Let us check that \gt\ cannot have a \sepe\ $e=uv$. If $uv$ is coloured $d$ or $b^*$, then note that it is contained in a basic cycle $C$, and each of its endvertices $u,v$ is incident with a basic cycle $D_u,D_v \neq C$. Now note that $(C \cup D_u \cup D_v) - \{u,v\}$ is connected, which means that if $\gt - \{u,v\}$ is disconnected, then \gt\ was already disconnected before removing $\{u,v\}$, a contradiction. If $uv$ is coloured $z$ instead, then note that the $d$-edge incident with $u$ and the $b^*$-edge incident with $v$ lie in a common component of $\gt - \{u,v\}$ because there is a basic cycle containing these two edges and $uv$. But this contradicts the property of the separating colour described above. Thus in all cases we obtain a contradiction if \gt\ is of type \ref{vi}.

It also follows from the analysis in \cite{cay2con} that for every \prem\ of a graph of type \ref{vii} at least two colours preserve spin, and so again the embedding of \gt\ we have implies that \gt\ cannot be of that type either.

Suppose now \gt\ is of type  \ref{viii}, which means that $\gt \isom Cay\left<b^*,z,d\mid {b^*}^2, z^2, d^2, (dzb^* z)^m \right>$. Then replacing $b^*$ and $z$ as above, we obtain the following presentation for \G: 
$G \isom Cay\left<b,c,d\mid {b^2}, c^2, d^2, (c(bc)^nd)^{2m} \right>, n\geq 2$. Note however, that the latter presentation is identical with that of \Tr{main2}~\ref{vii}; thus, by the converse implication of that theorem, \g is not \tcon\ in this case contradicting our assumption. 

If \gt\ is of the degenerate type  \ref{ix}, then there must be a pair of edges of \gt\ that have common endvertices. No $z$-edge can participate in such a pair, because $ZB$ and $Zd$ are both subwords of the word $ZdZB$ inducing the \mpsc s. But if a $d$ and a $b^*$ edge form such a pair, then the corresponding cycle of \g bounds a hexagonal face, which cannot be the case by \Prr{nohex}. Thus \gt\ is not of type  \ref{ix} either.

The only possible candidate left is type  \ref{v}, and this possibility can indeed occur as we will see in the next subsection. In this case we have either\\
$\gt \isom Cay \left< b^*,z,d\mid {b^*}^2, z^2, d^2, (b^*z)^{2r}, (zb^*zd)^m\right>$, $r,m\geq 2$, or \\ 
$\gt \isom Cay \left< b^*,z,d\mid {b^*}^2, z^2, d^2, (dz)^{2r}, (zb^*zd)^m\right>$, $r,m\geq 2$,\\
depending on which of the two spin reversing colours $b^*,d$ forms 2-coloured cycles with $z$. Replacing $b^*$ and $z$ as above, we obtain the following two presentations respectively:
 
$\g \isom Cay \left< b,c,d\mid b^2, c^2, d^2; (c(bc)^{n+1} d)^{2r},  (c(bc)^nd)^{2m} \right>$, $n\geq 1$, $r,m\geq 2$, or 
 
$\g \isom Cay \left< b,c,d\mid b^2, c^2, d^2; (c(bc)^{n-1} d)^{2r},  (c(bc)^nd)^{2m} \right>$, $n,r,m\geq 2$.
 
Again, there is no difference between these two presentations except for the naming of the parameters, and we can omit the first one. We have thus obtained possibility \ref{IId2iii} of \Tr{TIId}.

This completes the proof of the forward implication of \Tr{TIId}.

\subsubsection*{The converse implication}

In this section we show that \fe\ presentation as in \Tr{TIId} the corresponding \Cg\ is planar and \tcon. Our approach is very similar to that of the proof of \Tr{TIIc} (page~\pageref{convIIc}), and it more or less goes through the proof of the forward implication the other way round. %We assume that the reader is acquainted with the arguments used in the proof of \Tr{TIIc}, thus providing only a sketch here.

Consider first a \Cg\ of  {\bf type \ref{IId2i}}: 
$$G \isom Cay\left<b,c,d\mid {b^2}, c^2, d^2, (c bdbcd)^k, (c(bc)^nd)^{2m}\right>, k\geq 2,\ n,m\geq 1.$$

We are going to construct an embedding of \G. For this, consider first the auxiliary \Cg\ $G_2 = Cay\left<b^*,z,d\mid {b^*}^2, z^2, d^2, (b^*d)^k, (b^*zdz)^m\right>$. Then \Tr{LG5} \ref{ziii} yields an embedding $\sig_2$ of \gt\ in which only $z$ preserves spin, and \gt\ is \tcon\ by \Tr{endtcon}. Modify \gt\ into a further auxiliary graph \gtp\ by replacing each $b^*$ edge of \gt\ by a path of length 5 with edges coloured $bcdcb$ (recall that this was the word $B$), and replacing each $z$ edge of \gt\ by a path of length $2n-1$ with edges coloured $c(bc)^{n-1}$ (the word $Z$). Note that these words are symmetric, and so it does not matter at which end of those paths we start colouring the new edges.

Note that every cycle of \gt\ induced by $(b^*zdz)^m$ has turned into a cycle of \gtp\ induced by the word of \eqref{Ld6}. Moreover, every such cycle bounds a face of \gtp. We call these cycles the \defi{basic cycles} of \gtp. A \defi{corner} of \gtp\ is a vertex of degree 3. By \eqref{socsd} every basic cycle $B$ of \gtp\ contains precisely two distinct `\socs'.  Note that precisely one of these \socs\ $S_B$ consists of corners of \gtp, while the elements of the other \soc\ $T_B$ are non-corners.  Now \fe\ basic cycle $B$ of \gtp, construct a copy $H_B$ of \gtp\ \st\ $B\subseteq H_B$ and the elements of $T_B$ are corners of $H_B$, and embed $H_B$ in the face of \gtp\ bounded by $B$. Repeat this inductively ad infinitum for each of the newly appeared basic cycles. Let \g be the resulting plane graph and \sig\ its embedding. It follows from Sabidussi's Theorem that \g is a \Cg: assertions \eqref{autoa} and \eqref{autob} are still valid, and imply that the \auto s of \g act transitively on its vertices. 

We will now prove that \g is \tcon\ using \Lr{lkcon}. For this, let \seq{H}\ be an enumeration of the  copies of \gtp\ in \G, and \fe\ $i$ let $K_i$ be the set of corners of $H_i$. The requirements \ref{coi} and \ref{coii} of \Lr{lkcon} are satisfied for $k=3$ since \gt\ is \tcon. We will show that the third requirement \ref{coiii} is satisfied unless $n=m=1$. Indeed, if the latter is the case, then we have $Z= c$ and so the $z$ edges of \gt\ are also edges of \G. Moreover, any basic cycle $C$ is induced by $(b^*zdz)^m$, and so after subdividing it only contains one metaedge. Note that such a metaedge contains all elements of one of the \socs\ of $C$. This implies that removing the first and last edge of this metaedge disconnects the two copies of \gtp\ that share $C$ in \G, which means that \g is not \tcon\ in this case.

If on the other hand one of $n,m$ is greater than 1, then each basic cycle $C$ contains more than one metaedge, and it follows that $C$ contains at least four edges joining its two \socs. Thus requirement \ref{coiii} of \Lr{lkcon} is satisfied too, and so  \g is \tcon\ in this case.

Next, we claim that \g has no 2-coloured cycle unless $n=m=1$. We begin with showing that $bd$ and $cd$ have infinite order independently of the values of $n$ and $m$. For this, note that \fe\ $b$-edge $e$ of \g \ti\ a $b^*$ metaedge starting with $e$, and a basic cycle $C_e$ containing this metaedge. The spin behaviour implies that the two $d$-edges incident with $e$ lie in distinct sides of $C_e$. Now given a  path $P$ of \g the edges $b_1 d_1 b_2 d_2 \ldots$ of which alternate in the colours $b,d$, consider the cycles $C_{b_i}$ obtained as above, and note that $C_{b_i}\neq C_{b_{i+1}}$; indeed if $C_{b_i}= C_{b_{i+1}}$ then $d_i$ is a \defi{chord} of $C_{b_i}$, i.e.\ an edge having both vertices on that cycle, but a basic cycle cannot have a chord since $k>1$. Note moreover that, by our previous remark about the spin, $C_{b_{i-1}}$ and $C_{b_{i+1}}$ lie in distinct sides of $C_{b_{i}}$. This immediately implies that $P$ cannot be a closed path, and so  $bd$ has infinite order indeed. 

Similarly, note that \fe\ path $P$ induced by $b\ c(bc)^n\ b$ \ti\ a basic cycle $C$ that contains the interior of $P$ and the first and last edge of $P$ lie in distinct sides of $C$. Adapting the above argument we conclude that $bc$ has infinite order too.

If $n>1$ then \fe\ $dcd$-path \ti\ a basic cycle $C$ containing only the middle edge, and the two incident $d$-edges lie in distinct sides of $C$, which again allows us to prove that $cd$ has infinite order. If $n=1$,  then note that if \g has a finite $cd$-cycle $C$, then $C$ is also a cycle of \gt\ since $Z=c$ in this case. Moreover, since $c$ preserves spin, some of the $b^*$-edges incident with $C$ lie in one of its sides and some of them lie in its other side; the situation around $C$ looks like \fig{fIId1} after adapting the colours. If no such edge is a chord of $C$ then, easily, $C$ is a dividing cycle, which contradicts the fact that \gt\ is at most 1-ended. If $C$ has a chord $e=uv$, then consider the shortest subarc $P$ of $C$ with endvertices $u,v$. If $||P||>3$, then \ti\ a \auto\ \gt\ that fixes $C$, maps $u$ to a vertex $u'$ of $P$ that has an incident  $b^*$-edge in the side of $C$ in which $e$ also lies, and maps $v$ to a vertex $v'$ outside $P$ (use \fig{fIId1} again to see this). But then the edges $uv, u'v'$ must cross, yielding a contradiction. 

Note that $||P||\neq 1$ because $k>1$. If $||P||=3$, and so $P$ is induced by $cdc$, then $P\cup e$ is a face-boundary of \gt. But every face-boundary of \gt\ is induced by one of the relators $(b^*d)^k, (b^*zdz)^m$ in its presentation. Thus $m$ must equal 1 in this case. To sum up, we proved that if $n+m\geq 3$ then $cd$ has infinite order (in \G) too. The interested reader will be able to check that this assertion is best possible: if $n=m=1$ then  $cd$ must have finite order in \G. We do not need this fact for our proof though since $n=m=1$ is already forbidden because of the connectivity.

The fact that \g has the desired presentation now follows from the forward implication of \Tr{TIId}, which we have already proved, since we checked that \g has the desired properties.
\medskip

Consider now a presentation of  {\bf type \ref{IId2ii}}:
$$G \isom Cay\left<b,c,d\mid {b^2}, c^2, d^2, (bdb cdc)^q; (c(bc)^{n}d)^{2m}, (c(bc)^{n+1} d)^{2r} \right>, n\geq 1, r,m,q\geq 2$$
We can then construct an embedding of \g by the same method, except that we have to start with a different \gt: this time we let\\ $\gt \isom Cay\left<b^*,z,d\mid {b^*}^2, z^2, d^2, (b^* d)^q, (dzb^* z)^m, (zb^*)^{2r} \right>$, and it follows from \Tr{TIId11} that \gt\ is again planar and \tcon, and has the desired spin behaviour. Otherwise, the construction remains the same. % :QUATSCH:In this case however, we allow two types of basic cycles: those induced by the relator $(dzb^* z)^m$ as above, and those induced by the relator $(zb^*)^{2r}$. Note that after replacing the letters $b^* ,z$ by the corresponding words as above, not only the former but also the latter relator is transformed into a word as is \eqref{Ld6}, where the value of $n$ is larger by 1 in the latter. This allows us to perform glueing operations as in the construction of a graph of type  \ref{IId2i} in both kinds of basic cycles, but of course we have to be careful enough to glue along pairs of basic cycles of the same type. 

The parameter $m$ is now large enough to make sure that the requirements of \Lr{lkcon} for $k=3$ are satisfied in all cases, and so \g is \tcon. Moreover, we can prove that \g has no 2-coloured cycle by the same arguments, and our task is made easier by the fact that $m\geq 2$ now.

\medskip
Finally, consider a presentation of {\bf type \ref{IId2iii}}:
$$\g \isom Cay \left< b,c,d\mid b^2, c^2, d^2,  (c(bc)^{n-1}d)^{2m}, (c(bc)^{n} d)^{2r} \right>, n,r,m\geq 2.$$ 
We use the same approach again, except that it is now trickier to show that the resulting \Cg\ \g is \tcon. In this case we start our construction letting $\gt \isom Cay \left< b^*,z,d\mid {b^*}^2, z^2, d^2, (b^*z)^{2r}, (zb^*zd)^m\right>$, which corresponds to type  \ref{v} of \Tr{main2} and has connectivity 2. It is proved in \citeIIconCorNosp\ that this \gt\ has the following properties:
\begin{enumerate} 
\item \label{nosi} \gt\ has a \prem\ $\sig_2$ in which $z$ preserves spin while $b^*, d$ reverse spin (see \citeCayIIFignospiral). In this embedding, each vertex is incident with two faces bounded by a cycle induced by the relator $(zb^*zd)^m$ and one face that has infinite boundary.
\item \label{nosii} \gt\ cannot be separated by removing two edges $e,f$ unless both $e,f$ are coloured $d$, and it cannot be separated by removing a vertex and an edge $e$ unless $e$ is coloured $d$; 
\item \label{nosiii} If a pair of vertices $s,t$ of a cycle $C$ of \gt\ induced by 
$(zb^*zd)^m$ separates \gt, then both $s,t$ are incident with a $d$-edge of $C$;
\item \label{nosiv} \gt\ has no \sepe;
\item \label{nosvi}	\fe\  cycle $C$ of \gt\ induced by the word $(zb^* zd)^m$, and every $b^*$ edge $vw$ of $C$, \ti\  a \pth{v}{w}\ in $\gt$ meeting $C$ only at $v,w$, and
\item \label{nosv} If two cycles $C,D$ of \gt\ induced by the word $(zb^* zd)^m$ share an edge $uv$, then \ti\ path from $C$ to $D$ in $\gt - \{u,v\}$.
\end{enumerate}
Construct \g and an embedding of its using \gt\ and $\sig_2$ (provided by \ref{nosi} as in the previous cases. 

Although $\kap(\gt)=2$, we will be able to prove 
\begin{proposition}
\g is \tcon.
\end{proposition}
\begin{proof}
We will apply \Lr{lkcon}, with $K_i$ being, as usual, the set of corners of a copy $H_i$ of \gtp\ in \G, and \seq{H}\ being an enumeration of these copies.

So let us check that $K_i$ is \tcon\ in \G. To begin with, note that $H_i$ is the union of basic cycles, that is, cycles induced by the word $(zb^* zd)^m$, because this word contains all colours, and so every edge is in a basic cycle. 
\labtequ{bctcon}{\Fe\ basic cycle $C$ of $H_i$, $K_i\cap C$ is \tcon\ in \G.}
Indeed, suppose \ta\ vertices $s,t\in V(G)$ separating two vertices $x,y\in K_i\cap C$. Since $C$ is a cycle, both $s,t$ must lie on $C$, with each of the two components of $C- \{s,t\}$ containing one of $x,y$. It follows easily from \ref{nosii} that none of $s,t$ can be an interior vertex of a $b^*$ or $z$ metaedge contained in $C$, and so $s,t\in K_i\cap C$.  By \ref{nosiii} each of $s,t$ is incident with a $d$-edge of $C$, and these two $d$-edges $d_s,d_t$ are distinct by \ref{nosiv}. %Let $S_1,S_2$ be the two $zb^*z$-paths of $C$ having an endvertex in common with $d_s$ (\fig{fds}). 
Note that by the choice of the words $Z,B$, the $b^*$ metaedge $b_s$ containing $d_s$ is contained in $C$, %and each of $S_1,S_2$ contains an endvertex of $b_s$ as an interior vertex. 
and note also that $d_t$ is not contained in $b_s$ as the latter contains only one $d$-edge. Thus, each of  the two components $S_1,S_2$ of $C- \{s,t\}$ contains one of the  endvertices $p,q$ of $b_s$.
Now by \ref{nosvi}, \ti\ a path $P$ joining $p$ to $q$ in the copy of \gtp\ sharing $C$ with $H_i$ which path  has no interior vertex on $C$. Thus  $s,t\not\in V(P)$, and $P$ connects the two components of $C - \{s,t\}$. This contradicts our assumption that $s,t$ separate  $x,y$ in \G, and proves \eqref{bctcon}.

\showFig{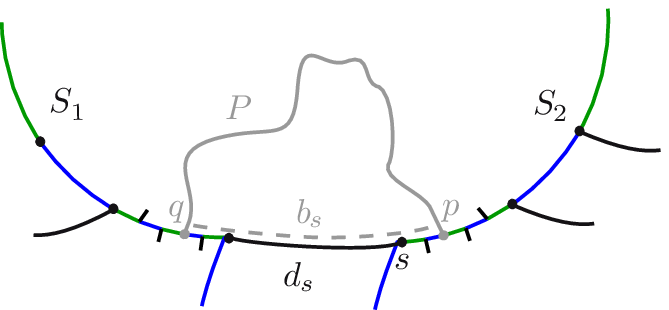}{The path $P$ accounting for the \tcon ness of \G.}

Next, we claim that 
\labtequ{shar}{\fe\ $v,w\in K_i$, there is a finite sequence of basic cycles $C_0, \ldots C_k$ of $H_i$ \st\ $v\in C_0$,  $w\in C_k$, and $C_i$ shares a metaedge with $C_{i+1}$ \fe\ relevant $i$.}
Indeed, since the word $(zb^* zd)^m$  inducing the basic cycles involves all three colours, any two edges of $H_i$ sharing a vertex lie in a common basic cycle, and so \eqref{shar} can be proved by induction on the length of a \pth{v}{w}\ in $\gt$.
\medskip

Now \ref{nosv} yields that \fe\ two basic cycles $C,D$ of $H_i$ sharing a metaedge, $C$ is \tcon\ to $D$ in $H_i$. Combining this with \eqref{shar} and \eqref{bctcon} implies that $K_i$ is \tcon\ in \G: given $v,w\in K_i$, and a sequence $C_i$ as in  \eqref{shar}, we can construct a \pth{v}{w}\ in \g avoiding any fixed pair of vertices by combining paths joining two suitable vertices of $C_i$ \fe\ relevant $i$; see the proof of \Lr{lkcon} for a more detailed exposition of this argument.

Thus we have proved that the set of corners $K_i$ of any copy $H_i$ of \gtp\ in \g is \tcon\ in \G. Moreover, for any two copies $H_i,H_j$ of \gtp\ sharing a basic cycle $C$, it is clear that $K_i$ is \tcon\ to $K_j$ because $C$ contains more than 2 pairwise disjoint paths joining its two \socs, and each of  $K_i,K_j$ contains a distinct \soc\ of $C$. We can thus apply \Lr{lkcon} to prove that \g is \tcon.
\end{proof}

This completes the proof of the converse implication of \Tr{TIId}.

The graphs of the last type \ref{IId2iii} also have the surprising property that we already encountered in \Cr{monsters} that none of their faces is bounded by a cycle:
\begin{corollary} \label{monsters2}
For every $n,r,m\geq 2$, the graph\\ $\g \isom Cay \left< b,c,d\mid b^2, c^2, d^2;  (c(bc)^{n-1}d)^{2m}, (c(bc)^{n} d)^{2r} \right>$  is \tcon\ and has no finite face boundary.
\end{corollary}
\begin{proof}
Recall that in order to construct a \g as above, we started with the graph \gt\ and its embedding $\sig_2$, and inductively glued copies of \gt\ inside the basic cycles. By \ref{nosi} $\sig_2$ had faces with infinite boundary, and these faces were left intact by our construction, except for subdividing the edges in their boundary. Thus \g also has at least one face with infinite boundary at every vertex. But as two of the colours $b,c$ preserve spin in \G, such a face can be mapped to any other by a \auto\ of \G, implying that every face has infinite boundary.
\end{proof}

\section{Outlook} \label{afterIId}  

%\subsection{Word extensions}
In \Sr{intEx} we showed two examples of a \defi{word extension}. This operation was implicit throughout the paper whenever we used a presentation of \gt\ to obtain one of \G. It would be interesting to study  word extensions in greater generality. One aim could be to refine Stallings' theorem into a theorem about all \Cg s rather than a theorem about their groups: prove that every multi-ended \Cg\ can be obtained from simpler \Cg s by means of certain operations including  word extensions. A modest first step in this direction would be to prove that every multi-ended cubic \Cg\ is a word extension of a cubic \Cg\ of a subgroup. To put it in a different way:

\begin{problem}
Every multi-ended cubic \Cg\ $G=Cay(\Gam,S)$ contains a subdivision of a cubic \Cg\ \gt\ of a proper subgroup of \Gam, and \g is the union of the translates of \gt\ under \Gam. 
\end{problem}

One of the most important ideas in this paper was the use of \socs\ (\Dr{defsoc}), and realising that they can be found in each of our multi-ended graphs. It would thus be interesting to prove that they appear in all  multi-ended planar \Cg s, not just the cubic ones. It also seems promising to try to generalise the concept to non-planar \Cg s, perhaps using the ideas of \cite{dunKro}. For example, the big cycle of \fig{fiamal} (iii) accommodates two dual \socs, each being the set of its vertices sending edges to one of its sides. This graph was the result of a  word extension of the graph of \fig{fiamal} (i), where a generator $a$ was replaced using the word $a_*^2$; see \Sr{secIa}. But we could have used the word  $a_*^3$ instead, in which case each edge of the original cycle would be subdivided into three, we would have three `dual' \socs, and removing the cycle would leave three infinite components. Thus the resulting graph is not planar, but still it can be analysed by our methods.

\medskip

\comment{

As mentioned in \Sr{intEx}, our results prove a variant of Mohar's \Cnr{conjBM} for the cubic case. Our proof, as summarised in \Sr{secSke}, and the recent vertex-cut theory \cite{dunKro}, suggest that this might be true in general, with no restriction on the vertex degree and without assuming planarity:
\begin{conjecture} \label{conjStal}
Every accessible \Cg\ \G\ can be obtained as the tree amalgamation of subdivisions of one or more Cayley graphs, each of which is either finite or 1-ended.
\end{conjecture}

}

%\subsection{Other problems}

\Tr{main} shows that every cubic planar \Cg\ admits a planar presentation with at most 6 relators. This motivates
\begin{problem}
Let $f(n)\in \N\cup \{\infty\}$ be the smallest cardinal such that every $n$-regular planar \Cg\ admits a planar presentation with at most $f(n)$ relators. Is $f(n)$ finite for each $n$? If yes how fast does it grow with $n$?
\end{problem}

In this paper we constructed surprising examples of planar \tcon\ \Cg s in which no face is bounded by a cycle. Thus one can ask
\begin{problem}
Is there, \fe\ $k\in\N$, a planar $k$-connected \Cg\ in which no face is bounded by a cycle?
\end{problem}

\acknowledgements{I am very grateful to Bojan Mohar, for triggering my interest in the topic and for later valuable discussions, and to Martin Dunwoody, for discussions leading to improvements in the final version.}

\bibliographystyle{plain}
\bibliography{collective}
\end{document}